\newif\ifaos
\aosfalse

\ifaos
\documentclass[aos,submission]{imsart}
\else
\documentclass{article}
\RequirePackage[margin=1.5in]{geometry}
\fi

\RequirePackage{amsmath, amsthm, amssymb}
\RequirePackage{graphicx}
\RequirePackage{verbatim}
\RequirePackage{natbib}
\RequirePackage{caption}
\RequirePackage{subcaption}
\RequirePackage{fancyvrb}
\RequirePackage{enumerate}
\RequirePackage{relsize}
\RequirePackage{mathtools}
\RequirePackage{tikz}
\RequirePackage{float}
\RequirePackage{hyperref}
\RequirePackage[ruled,vlined]{algorithm2e}
\hypersetup{colorlinks,citecolor=blue,urlcolor=blue, linkcolor=blue}

\RequirePackage{stefan_tex}

\usetikzlibrary{decorations.pathreplacing}
\usetikzlibrary{arrows}
\tikzset{
	node1/.style={circle,draw,thick, fill = red!10},
	node0/.style={circle,draw,thick,fill = blue!10}
}

\newcommand\independent{\protect\mathpalette{\protect\independenT}{\perp}}
\def\independenT#1#2{\mathrel{\rlap{$#1#2$}\mkern2mu{#1#2}}}

\theoremstyle{definition}
\newfloat{algbox}{tbp}{lop}
\newtheorem{alg}{Procedure}
\newcommand{\myalg}[3]{
\begin{center}
\fbox{
\parbox{10cm}{
\begin{alg}\label{#1}{\textsc{ #2}}
\vspace{.1cm}\\ #3 
\end{alg}
}}
\end{center}
}

\theoremstyle{plain}
\newtheorem{prop}{Proposition}

\newtheorem{coro}[prop]{Corollary}
\newtheorem{lemm}[prop]{Lemma}
\newtheorem{theo}[prop]{Theorem}

\theoremstyle{definition}
\newtheorem{exam}{Example}

\newtheorem{assu}{Assumption}

\theoremstyle{remark}

\ifaos
\newwrite\zzz
\immediate\openout\zzz=from\jobname.tex
\AtEndDocument{%
\immediate\write\zzz{\string\setcounter{prop}{\number\value{prop}}}
\immediate\closeout\zzz
}
\fi

\newcommand{\TOT}{\operatorname{TOT}}
\newcommand{\DIR}{\operatorname{DIR}}
\newcommand{\IND}{\operatorname{IND}}
\newcommand{\HT}{\operatorname{HT}}
\newcommand{\HAJ}{\operatorname{HAJ}}

\newcommand{\hbsm}{\bar{h}_n}
\newcommand{\hsml}{h_n}
\newcommand{\Hfcn}{H_n}

\newcommand{\gbsm}{\bar{g}_n}
\newcommand{\gsml}{g_n}
\newcommand{\Gfcn}{G_n}

\newcommand{\hbsmc}{\bar{h}}
\newcommand{\hsmlc}{h}
\newcommand{\Hfcnc}{H}

\newcommand{\gbsmc}{\bar{g}}
\newcommand{\gsmlc}{g}
\newcommand{\Gfcnc}{G}

\newcommand{\betac}{\hat{\beta}}
\newcommand{\tauc}{\htau^{\PC}}

\newcommand{\Au}{A}
\newcommand{\Bu}{B}

\newcommand{\clower}{c_l}

\newcommand{\cupper}{c_u}
\newcommand{\rhon}{\rho_n}

\newcommand{\U}{\operatorname{U}}
\newcommand{\PC}{\operatorname{PC}}
\newcommand{\tauha}{\htau^{\HAJ}}
\newcommand{\tauht}{\htau^{\HT}}
\newcommand{\tauuu}{\htau^{\U}}
\newcommand{\bV}{\bar{V}}
\newcommand{\vep}{\varepsilon}
\newcommand{\psih}{\hat{\psi}}
\newcommand{\psis}{\psi^{*}}
\newcommand{\psit}{\tilde{\psi}}
\newcommand{\Es}{E^{*}}
\newcommand{\Gm}{G}
\newcommand{\lambdah}{\hat{\lambda}}
\newcommand{\lambdas}{\lambda^*}
\newcommand{\lambdat}{\tilde{\lambda}}
\newcommand{\expb}{C e^{-C n^{-\kappa_1/2}}}

\newcommand{\psiRh}{\hat{\psi}^R}
\newcommand{\psiRs}{\psi^{* R}}
\newcommand{\psiRt}{\tilde{\psi}^R}

\newcommand{\Psih}{\hat{\Psi}}
\newcommand{\Psis}{\Psi^{*}}
\newcommand{\Psit}{\tilde{\Psi}}
\newcommand{\PsiRh}{\hat{\Psi}^R}
\newcommand{\PsiRs}{\Psi^{* R}}
\newcommand{\PsiRt}{\tilde{\Psi}^R}

\newcommand{\hbetaR}{\hat{\beta}^R}

\begin{document}

\ifaos
\else

\title{Random Graph Asymptotics for Treatment Effect Estimation under Network Interference}

\author{
	\makebox[45mm]{Shuangning Li} \\ Stanford University \and
	\makebox[45mm]{Stefan Wager} \\  Stanford University
}

\date{Draft version \ifcase\month\or
	January\or February\or March\or April\or May\or June\or
	July\or August\or September\or October\or November\or December\fi \ \number%
	\year\ \  }

\maketitle

\begin{abstract}
The network interference model for causal inference places
experimental units at the vertices of an undirected exposure graph, such that treatment assigned to
one unit may affect the outcome of another unit if and only if these two units are connected by an edge.
This model has recently gained popularity as means of incorporating interference effects into the
Neyman--Rubin potential outcomes framework; and several authors have considered estimation of
various causal targets, including the direct and indirect effects of treatment.
In this paper, we consider large-sample asymptotics for treatment effect estimation under network interference
in a setting where the exposure graph is a random draw from a graphon. When targeting the direct
effect, we show that---in our setting---popular estimators are considerably more accurate than existing results suggest,
and provide a central limit theorem in terms of moments of the graphon. Meanwhile, when targeting the indirect effect,
we leverage our generative assumptions to propose a consistent estimator in a setting where no
other consistent estimators are currently available. We also show how our results can be used to conduct a
practical assessment of the sensitivity of randomized study inference to potential interference effects.
Overall, our results highlight the promise of random graph asymptotics in understanding the practicality and limits
of causal inference under network interference.
\end{abstract}

\fi

\section{Introduction}

In many application areas, we seek to estimate causal effects in the presence of cross-unit interference,
i.e., when treatment assigned to one unit may affect observed outcomes for other units.
One popular approach to modeling interference is via an exposure graph or network,
where units are placed along vertices of a graph and any two units are connected by an edge if treating
one unit may affect exposure of the other:
For example, \citet*{athey2018exact} and \citet*{leung2020treatment}
discuss experiments whose study units may interact via a social network, e.g., a friendship or professional network,
and consider network interference models whose exposure graph corresponds to this social network.
The statistical challenge is then to identify and estimate
causal quantities in a way that is robust to such interference.

The existing literature on treatment effect estimation under network interference is formalized using a
generalization of the strict randomization inference approach introduced by \citet{neyman1923applications}. In a sense made precise
below, these papers take both the interference graph and a set of relevant potential outcomes as deterministic,
and then consider inference that is entirely driven by random treatment assignment
\citep*{aronow2017estimating,hudgens2008toward}.
A major strength of this approach is that any conclusions derived from it are simple to interpret because they do not rely on
any stochastic assumptions on either the outcomes or the interference graph.
However, despite the transparency of the resulting analyses, it is natural to ask about the cost of using such
strict randomization inference. If an approach to inference needs to work uniformly for any possible set of
potential outcomes and any interference graph, does this limit its power over ``typical'' problems?
Can appropriate stochastic assumptions enable more tractable analyses of treatment effect estimation
under network interference, thus pointing the way to useful methodological innovations?

In this paper, we investigate the problem of treatment effect estimation under random graph asymptotics;
specifically, we assume that the interference graph is a random draw from an (unknown) graphon.
When paired with a number of regularity assumptions discussed further below, including an anonymous interference assumption,
we find that our use of such random graph asymptotics lets us obtain
considerably stronger guarantees than are currently available via randomization inference. When estimating
direct effects, we find that standard estimators used in the literature are unbiased and asymptotically
Gaussian for substantially denser interference graphs than was known before. And, when estimating
indirect effects, our analysis guides us to a new estimator that has non-negligible power in
a setting where no existing results based on randomization inference are available.


\subsection{Graphon Asymptotics for Network Interference}
\label{section:stats_setting}

Suppose that we collect data on subjects indexed $i = 1, \, ..., \, n$, where each subject is
randomly assigned a binary treatment $W_i \in \cb{0, \, 1}$, $W_i \sim \text{Bernoulli}(\pi)$ for some $0 \leq \pi \leq 1$,
and then experiences an outcome $Y_i \in \RR$. Following the Neyman-Rubin causal model \citep{imbens2015causal}, we
posit the existence of potential outcomes $Y_i(w) \in \RR$ for all $w \in \cb{0, \, 1}^n$, such that the
observed outcomes satisfy $Y_i = Y_i(W)$.
For notational convenience,
we will often write $Y_i(w_j = x; \, W_{-j})$
to reference specific potential outcomes; here, $Y_i(w_j = x; \, W_{-j})$ means the outcome we
would observe for the $i$-th unit if we assigned the $j$-th unit to treatment status $x \in \cb{0, \, 1}$, and
otherwise maintained all but the $j$-th unit at their realized treatments
$\smash{W_{-j} \in \cb{0, \, 1}^{n-1}}$. We sometimes use shorthand $Y_i(x; W_{-i}) := Y_i(w_i = x; W_{-i})$ for the $i$-th index.
Finally, we posit a graph with edge set $\cb{E_{ij}}_{i, \, j = 1}^n$ and vertices at the $n$ experimental subjects
that constrains how potential outcomes may vary with $w$: The $i$-th outcome may only depend on the $j$-th
treatment assignment if there is an edge from $i$ to $j$, i.e., $Y_i(w) = Y_i(w')$ if $w_i = w'_i$ and
$w_j = w'_j$ for all $j \neq i$ with $E_{ij} = 1$.


We seek to estimate the direct, indirect and total effects of the treatment on the outcome,
\begin{equation}
\label{eq:HH}
\setlength{\jot}{3pt}
\begin{split}
&\btau_{\DIR}(\pi) = \frac{1}{n} \sum_i  \EE[\pi]{Y_i(w_i = 1; \, W_{-i}) - Y_i(w_i = 0; \, W_{-i})| Y(\cdot)}, \\
&\btau_{\IND}(\pi) =  \frac{1}{n} \sum_i  \sum_{j \neq i}\EE[\pi]{Y_j(w_i = 1; \, W_{-i}) - Y_j(w_i = 0; \, W_{-i})| Y(\cdot)},\\
&\btau_{\TOT}(\pi) = \frac{d}{d\pi} \bigg\{\frac{1}{n} \sum_i  \EE[\pi]{Y_i | Y(\cdot)}\bigg\}, 
\end{split}
\end{equation}
where the expectations above are taken over the random treatment assignment $W_i \sim \text{Bernoulli}(\pi)$.
This definition \smash{$\btau_{\DIR}$} of the direct effect is by now standard \citep*{halloran1995causal,savje2017average},
while \smash{$\btau_{\IND}$} is a formal analogue of this definition for the indirect effect. These estimands are
further discussed by \citet{hu2021average}, who show that in any Bernoulli experiment (and including in our current setting),
\smash{$\btau_{\DIR}$} and \smash{$\btau_{\IND}$} decompose the total effect \smash{$\btau_{\TOT}$},
i.e.,\footnote{Several recent papers have also considered network interference in completely randomized experiments
where the number of treated units is fixed (e.g., in our setting, \smash{$n_1 = \lfloor n\pi \rfloor$} randomly chosen units are
assigned to treatment). This, however, gives rise to a number of subtle difficulties when studying estimands of the
type \eqref{eq:HH} because treatment assignment across different units is not independent and so, in general,
$\EE[\pi]{Y_j(w_i = x; \, W_{-i})} \neq \EE[\pi]{Y_j(w_i = x; \, W_{-i}) \cond W_i = x}$; see \citet{savje2017average} and
\citet{vanderweele2011effect} for further discussion. Throughout this paper, we avoid such issues by only considering
Bernoulli-randomized experiments.}
\smash{$\btau_{\TOT}(\pi) = \btau_{\DIR}(\pi) + \btau_{\IND}(\pi)$}. 
Given a sampling model on the potential outcomes, we also consider limiting population estimands
\begin{equation}
\label{eq:pop_estimand}
\begin{split}
\tau_{\TOT}(\pi) = \limn \EE{\btau_{\TOT}(\pi)}, \ \ \ \tau_{\DIR}(\pi) = \limn \EE{\btau_{\DIR}(\pi)}, \ \ \ \ldots
\end{split}
\end{equation}
provided these limiting objects exist. In this paper, we focus on estimating the quantities
$\btau_{\DIR}(\pi), \, \btau_{\DIR}(\pi), \, \btau_{\IND}(\pi)$, etc, at the treatment probability $\pi$ used for
data collection.

Qualitatively, the total effect captures the effect of an overall shift in treatment intensity, while the direct effect
captures the marginal responsiveness of a subject to their own treatment. Notice that the classical no-interference
setting where $Y_i$ only depends on the treatment assigned to the $i$-th unit is a special case of this setting with a
null edge set; moreover, in the case without interference, $\btau_{\TOT}(\pi)$ and $\btau_{\DIR}(\pi)$ match and are equal
to the sample average treatment effect, while the indirect effect is 0.

In the existing literature on treatment effect estimation under network interference, both the potential outcomes
$Y_i(w)$ and the edge set $E_{ij}$ are taken as deterministic, and inference is entirely driven by the random treatment
assignment $W_i \sim \text{Bernoulli}(\pi)$
\citep*{aronow2017estimating,athey2018exact,basse2019randomization,leung2020treatment,savje2017average}.
This strict randomization-based approach, however,
may limit the power with which we can estimate the causal quantities \eqref{eq:HH}, and judicious stochastic modeling may
help guide methodological advances in causal inference under interference. To this end, we consider
a variant of the above setting that makes the following additional assumptions:

\begin{assu}[Undirected Relationships]
\label{assu:undirected}
The interference graph is undirected, i.e., $E_{ij} = E_{ji}$ for all $i \neq j$.
\end{assu}

\begin{assu}[Random Graph]
\label{assu:random_graph}
The interference graph is randomly generated as follows. Each subject has
a random type \smash{$U_i \simiid \operatorname{Uniform}[0, \, 1]$}, and there is a symmetric measurable function $G_n : [0, \, 1]^2 \rightarrow [0, \, 1]$
called a \emph{graphon} such that $E_{ij} \sim \operatorname{Bernoulli}\p{G_n(U_i, \, U_j)}$ independently for all $i < j$. 
\end{assu}

\begin{assu}[Anonymous Interference]
\label{assu:anon}
The potential outcomes do not depend on the identities of their neighbors, and
instead only depend on the fraction of treated neighbors: $Y_i(w_i; \, w_{-i}) = f_i(w_i; \, \sum_{j \neq i} E_{ij} w_j / \sum_{j \neq i} E_{ij})$,\footnote{We take the convention of $0/0 = 0$.} where $f_i \in \mathcal{F}$ is the potential outcome function of the $i$-th subject, which may depend arbitrarily on $U_i$. We assume the pairs $(U_i, \, f_i)$ are independent and identically sampled from some distribution on $[0,1] \times \mathcal{F}$. 
\end{assu}

Relative to the existing literature, the most distinctive assumption we make here is our use of random graph asymptotics.
This type of graphon models are motivated by fundamental results on exchangeable arrays \citep{aldous1981representations,lovasz2006limits},
and have received considerable attention in the literature in recent years \citep*[e.g.,][]{gao2015rate,parise2019graphon,zhang2017estimating};
however, we are not aware of previous uses of this assumption to the problem of treatment effect estimation under network interference.
For our purposes, working with a graphon model gives us a firm handle on how various estimators behave in the large-sample
limit, and opens the door to powerful analytic tools that we will use to prove central limit theorems. In Section \ref{sec:sensitivity}, we discuss a number of example graphon models in the context of an application.

The anonymous interference assumption was proposed by \citet{hudgens2008toward} and is commonly used in the literature;
Figure \ref{fig:illustration} illustrates the anonymous interference assumption on a small graph.
The specific form of the anonymous interference assumption---where interference only depends on the ratio of treated neighbors
but not on the total number of neighbors---is called the ``distributional interactions'' assumption by
\citet{manski2013identification}.\footnote{It is plausible that similar analyses could also be applied to more general cases, e.g., when the potential
outcome function is asymptotically additive in the treatments of its neighbors, i.e., 
\smash{$Y_i(w_i; w_{-i}) \approx f_i(w_i, \pi) + 1/(n \rhon) \sum_{j \neq i, E_{ij} = 1} A_i(U_i, U_j) (w_j - \pi)$},
where $A_i$ is randomly drawn from some function class. Under this model, different neighbors can affect a unit
differently depending on their type $U_i$. Here, however, we don't pursue this further in order to keep the statistical assumptions simpler.}

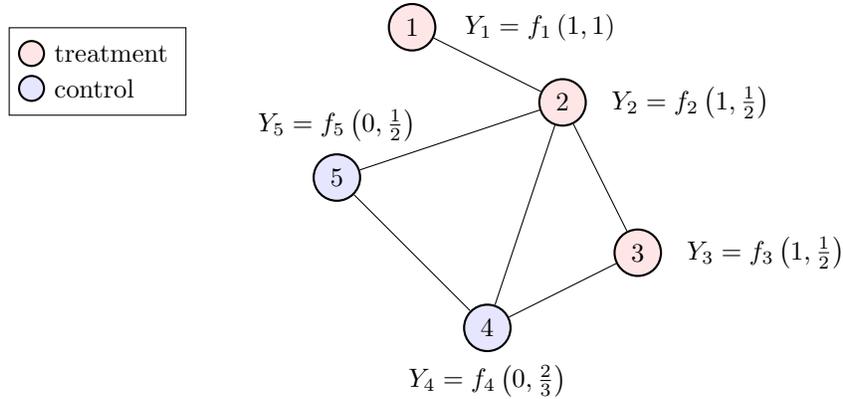
\begin{figure}[t]
	\centering

	\vspace{0.7cm}

	\begin{tikzpicture}
	\node[node1] (1) at (1,4) {$1$};
	\node[node1] (2) at (3,3) {$2$};
	\node[node1] (3) at (4,1) {$3$};
	\node[node0] (4) at (2,0) {$4$};
	\node[node0] (5) at (0,2) {$5$};

	\node[] at (2.7,4) {$Y_1 = f_1\p{1,1}$};
	\node[] at (4.7,3) {$Y_2 = f_2\p{1,\frac{1}{2}}$};
	\node[] at (5.7,1) {$Y_3 = f_3\p{1,\frac{1}{2}}$};
	\node[] at (2,-0.7) {$Y_4 = f_4\p{0,\frac{2}{3}}$};
	\node[] at (0,2.7) {$Y_5 = f_5\p{0,\frac{1}{2}}$};

	\path[every node/.style={font=\sffamily\small}]
	(2) edge node[] {} (3)
	(3) edge node[] {} (4)
	(2) edge node[] {} (5)
	(2) edge node[] {} (1)
	(4) edge node[] {} (5)
	(2) edge node[] {} (4);

	\matrix [draw,below left] at (-2,4) {
		\node [node1,label=right:treatment] {}; \\
		\node [node0,label=right:control] {}; \\
	};
	\end{tikzpicture}
		\caption{An illustration of a small graph}
	\label{fig:illustration}
\end{figure}

Given these assumptions, we can characterize our target estimands \eqref{eq:HH} and \eqref{eq:pop_estimand}
in terms of primitives from the graphon sampling model. The following assumption is designed to
let us handle both dense graphs, and graphs that are sparse in the sense of \citet*{borgs2019L}.

\begin{assu}[Graphon Sequence]
\label{assu:graphon}
The graphon sequence $G_n(\cdot, \, \cdot)$ described in Assumption \ref{assu:random_graph}
satisfies $G_n(U_i, U_j) = \min\cb{1,  \rhon \Gfcnc(U_i, U_j)}$, where $\Gfcnc(\cdot, \, \cdot)$ is
a symmetric, non-negative function on $[0, \, 1]^2$ and $0< \rhon \leq 1$ satisfies one of the following
two conditions: $\rhon = 1$ (dense graph), or $\limn \rhon = 0$ and $\limn n \rhon = \infty$ (sparse graph).
In the case of dense graphs, we simply write $\Gfcn = \Gfcnc$.
\end{assu}

Finally, we make an assumption on the smoothness of the potential outcome functions. Intuitively, this assumption states that the potential outcomes do not change much if the fraction of treated neighbors changes a little bit.

\begin{assu}[Smoothness]
\label{assu:smooth}
The potential outcome functions $f(w, \, x)$ satisfy
\begin{equation}
\label{eq:deriv}
\abs{f \p{w, x}}, \, \abs{f' \p{w, x}}, \, \abs{f'' \p{w, x}}, \, \abs{f''' \p{w, x}} \leq B 
\end{equation}
uniformly in $f \in \mathcal{F}$, $w \in \cb{0, \, 1}$ and $x \in [0, \, 1]$,
where all derivatives of $f$ are taken with respect to the second argument. 
\end{assu}

Proposition \ref{prop:estimands} provides a simple way of writing down our target estimands
in the random graph model spelled out above. Roughly speaking, the direct effect measures how much
$f$ changes with its first argument, while the indirect effect is the derivative of $f$ with respect
to its second argument. In other words, the direct effect captures the effect of a unit's own treatment status, 
while the indirect effect captures the effect of its proportion of treated neighbors. 
Here---and throughout this paper unless specified otherwise---all proofs are given in
\ifaos
\else
Appendix \ref{sec:proofs}.
\fi

\begin{prop}
\label{prop:estimands}
Consider a randomized trial under network interference satisfying Assumptions
\ref{assu:undirected}, \ref{assu:anon} and \ref{assu:smooth}, with treatment assigned independently as
$W_i \sim \text{Bernoulli}(\pi)$ for some $0 < \pi < 1$. Let $N_i = \sum_{j \neq i} E_{ij}$ be the number of neighbors of subject $i$ in the interference graph. Conditional on the interference graph and the potential outcome functions, the estimands
\eqref{eq:HH} can be expressed as follows, where $B$ is the smoothness constant in \eqref{eq:deriv}:
\begin{equation}
\label{eq:HH_graphon} 
\begin{split}
&\btau_{\DIR} = \frac{1}{n} \sum_{i = 1}\p{f_i(1,\pi) - f_i(0,\pi)} + \oo\p{\frac{B}{\min_i N_i}}, \\
&\btau_{\IND} = \frac{1}{n}\sum_i \p{\pi f'_i(1,\pi) + (1-\pi) f'_i(0,\pi)}  + \oo\p{\frac{B}{\sqrt{\min_i N_i}}}.
 \end{split}
\end{equation} 
Furthermore, if $\EE{1/(\min_i N_i)} = o(1)$, then the limits taken in \eqref{eq:pop_estimand} exist, and satisfy
\begin{equation}
\label{eq:pop_graphon}
\tau_{\DIR} = \EE{f_i(1,\pi) - f_i(0,\pi)}, \ \ \ \tau_{\IND} = \EE{\pi f'_i(1,\pi) + (1-\pi) f'_i(0,\pi)}.
\end{equation}
\end{prop}

\subsection{Overview of Main Contributions}

The key focus of this paper is estimation of the targets \eqref{eq:HH_graphon} and \eqref{eq:pop_graphon} under
random graph asymptotics. First, in Section \ref{section:dir}, we consider estimation of the direct
effect. It is well known there exist simple estimators of
$\btau_{\DIR}$ that are unbiased under considerable generality and that do not explicitly reference the graph structure
$\cb{E_{ij}}$; however, as discussed further in Section \ref{section:dir}, getting a sharp characterization of the error
distribution of these estimators has proven difficult so far. We add to this line of work by showing that, under our
random graph model, these simple estimators in fact satisfy a central limit theorem with $\sqrt{n}$-scale errors,
regardless of the density of the interference graph as captured by $\rho_n$. We also provide a quantitative expression
for variance inflation due to interference effects in terms of the graphon $G$.

Next, while the point estimators for the direct effect studied in Section \ref{section:dir} have a simple functional form,
the asymptotic variance in the corresponding central limit theorem appears challenging to estimate. To address this
challenge, in Section \ref{sec:sensitivity} we develop upper bounds for this asymptotic variance that can be used for
conservative inference. Our upper bounds are sharp enough to enable meaningful inference in the context of an application,
and are robust to having only generic knowledge about the structure of the interference graph $E$.

Finally, in Section \ref{sec:indirect}, we consider estimation of the indirect effect. This task appears to be substantially more difficult
than estimation of the direct effect, and we are aware of no prior work on estimating the indirect effect without either assuming
extreme sparsity (e.g., the interference graph has bounded degree), or assuming that the interference
graph can be divided up into cliques and that we can exogenously vary the treatment fraction in each clique.
Here, we find that natural unbiased estimators for $\btau_{\IND}$ that build on our discussion in Section \ref{section:dir}
have diverging variance and are thus inconsistent, even in reasonably sparse graphs. We then propose a new
estimator which we call the PC-balancing estimator, and provide both formal and numerical evidence that its
error decays as $\sqrt{\rho_n}$ for sparse interference graphs in the sense of Assumption \ref{assu:graphon},
provided the graphon $G$ admits low-rank structure.

\subsection{Notation}

Throughout this paper, we use $C, C_1, C_2 \dots$ for constants not depending on $n$. Note that $C$ might mean different things in different settings. 
We let $f'_i \p{w, x}$, $f''_i \p{w, x}$, etc., denote derivatives of $f$ with respect to the second argument $x$.
We write $N_i = \sum_{j \neq i} E_{ij}$ for the number of neighbors of subject $i$, and $M_i = \sum_{j \neq i} E_{ij} W_j$ for the number of treated neighbors. 
We use $\Omega(), \oo(), \oo_p(), \Omega_p(), o_p(), \sim, \asymp, \ll$ in the following sense: $a_n = \Omega(b_n)$ if $a_n \geq C b_n$ for $n$ large enough, where $C$ is a positive constant. $a_n = \oo(b_n)$ if $|a_n| \leq C b_n$ for $n$ large enough. $X_n = \oo_p(b_n)$, if for any $\delta>0$, there exists $M,N > 0$, s.t. $\PP{|X_n| \geq M b_n} \leq \delta$ for any $n > N$. $X_n = \Omega_p(b_n)$, if for any $\delta>0$, there exists $M,N > 0$, s.t. $\PP{|X_n| \leq M b_n} \leq \delta$ for any $n > N$. $X_n = o_p(b_n)$, if $\lim\PP{|X_n| \geq \epsilon b_n} \to 0$ for any $\epsilon > 0$. $a_n \sim b_n$ if $\lim a_n/b_n = 1$. $a_n \asymp b_n$ if there exist $C>0$, s.t. $\limsup a_n/b_n \leq C$ and $\liminf a_n/b_n \geq 1/C$. $a_n \ll b_n$ if $\lim a_n/b_n = 0$. 
Finally, the following functions of the limiting graphon $\Gfcnc$ from Assumption \ref{assu:graphon}
will occur frequently in our analysis:
For $i,j,k$ all different, we define expected neighbor count metrics
\begin{equation}
\begin{split}
&\gsmlc(U_i) = \EE{ \Gfcnc(U_i, U_j)| U_i }, \ \ \ \
\gbsmc = \EE{\Gfcnc(U_i,U_j)} = \EE{ \gsmlc(U_i)},
\end{split}
\end{equation}
and write related quantities of the graphon $G_n$ with an $n$-superscript.


\section{Estimating the Direct Effect}
\label{section:dir}

First, we consider estimation of the direct effect, i.e., the effect of treatment $W_i$ assigned to
the $i$-th unit on the outcome $Y_i$ for the $i$-th unit itself. In the classical setting without interference,
the direct effect corresponds exactly to the (sample) average treatment effect, which has been a focus
of the causal inference literature ever since \citet{neyman1923applications}. Furthermore,
several natural estimators of the average treatment effect designed for
the no-interference setting in fact converge to the direct effect in the presence of interference \citep*{savje2017average}.
Thus, one might expect the direct effect to be a particularly well behaved estimand---and our formal
results support this intuition.

An important property of the direct effect is that we can design a good, unbiased estimator for it
that only relies on randomization. To this end, consider the well known Horvitz-Thompson estimator
of the average treatment effect in the no-interference setting (also called inverse propensity
weighted, IPW, estimator)
\begin{equation}
\label{eq:HT}
\hat{\tau}^{\HT}_{\DIR} = \frac{1}{n}\sum_i \frac{W_i Y_i }{\pi} - \frac{1}{n}\sum_i \frac{(1-W_i) Y_i}{1-\pi},
\end{equation}
 where as always $0 < \pi < 1$ denotes the randomization probability $W_i \simiid \text{Bernoulli}(\pi)$.
A simple calculation then verifies that, under interference, the Horvitz-Thompson estimator is unbiased
for the $\btau_{\DIR}$ from \eqref{eq:HH} conditionally on potential outcomes (i.e., conditionally on both the
exposure graph and each unit's response functions):
\begin{align}
\notag
\EE{\hat{\tau}^{\HT}_{\DIR} \cond Y(\cdot)}
&=  \frac{1}{n}\sum_{i = 1}^n\frac{\EE{W_i Y_i(1,W_{-i}) \cond Y_i(\cdot)}}{\pi}  -  \frac{1}{n}\sum_i \frac{\EE{(1-W_i) Y_i(0,W_{-i}) \cond Y_i(\cdot)}}{1-\pi} \\
\label{eq:HT_unb}
&=  \frac{1}{n}\sum_{i = 1}^n \frac{\EE{W_i \cond Y_i(\cdot)} \EE{Y_i(1,W_{-i}) \cond Y_i(\cdot)}}{\pi} \\
\notag
&\ \ \ \ \ \ \ \ \ -  \frac{1}{n}\sum_{i = 1}^n \frac{ \EE{1-W_i \cond Y_i(\cdot)} \EE{Y_i(0,W_{-i}) \cond Y_i(\cdot)}}{1-\pi}
= \bar{\tau}_{\DIR}.
\end{align}
\citet*{savje2017average} use this fact along with concentration arguments to argue that
the Horvitz-Thompson estimator is consistent for the direct effect in sparse graphs, with
a rate of convergence that depends on the degree of the graph and approaches the
parametric $1/\sqrt{n}$ rate as we push towards a setting where its degree is bounded.
Specifically, in their Proposition 2, they argue that 
\begin{equation}
\label{eq:savje_bound}
\hat{\tau}^{\HT}_{\DIR} - \bar{\tau}_{\DIR} = \oo_p\p{\sqrt{\frac{1}{n^2} \sum_{i,j = 1}^n H_{ij}}},
\end{equation}
where the $H$ matrix tallies second-order neighbors, i.e., $H_{ii} = 1$ for all $i = 1, \, \ldots, \, n$ and for $i \neq j$
if there exist a node $k \neq i, \, j$ such that $E_{ik} = E_{jk} = 1$; and $H_{ij} = 0$ else.

Here, we revisit the setting of \citet*{savje2017average} under our graphon generative model.
Our qualitative findings mirror theirs: Familiar estimators of the average treatment effect without
interference remain good estimators of the direct effect from the perspective of random graph asymptotics.
However, our quantitative results are substantially sharper. We show that the Horvitz-Thompson estimator
is consistent for the direct effect in both sparse and dense graphs, and find that it has a $1/\sqrt{n}$ rate
of convergence regardless of the degree of the exposure graph. Furthermore, we establish a central limit
theorem for the estimator, and quantify the excess variance due to interference effects.

\subsection{A Central Limit Theorem}

As discussed above, our goal is to establish that natural estimators of the average treatment
effect in the no-interference setting are asymptotically normal around the direct effect once interference
effects appear. To this end, we consider both the Horvitz-Thompson estimator \eqref{eq:HT},
and the associated H\'ajek (or ratio) estimator
\begin{equation}
\label{eq:hajek}
\hat{\tau}^{\HAJ}_{\DIR} = \frac{\sum_{i = 1}^n W_i Y_i }{\sum_{i = 1}^n W_i} -  \frac{\sum_{i = 1}^n (1-W_i) Y_i}{\sum_{i = 1}^n 1-W_i}.
\end{equation}
Unlike the Horvitz-Thompson estimator, the H\'ajek estimator is not exactly unbiased; however,
its ratio form makes it invariant to shifting all outcomes by a constant.

Our first result is a characterization of the estimators $\hat{\tau}^{\HT}_{\DIR}$ and $\hat{\tau}^{\HAJ}_{\DIR}$ in large samples under the assumption of anonymous interference. This result does not require our graphon generative model, and instead only relies on smoothness of the potential outcome functions $f(w, \, \pi)$ as well as concentration of quadratic forms of $W_i - \pi$. In particular, this results holds conditionally on the exposure graph and the potential outcome functions. 

\begin{lemm}
\label{lemm:direct}
Under the conditions of Proposition \ref{prop:estimands} and conditionally on the graph and the potential outcome functions, the estimators of the direct effect defined in \eqref{eq:HT} and \eqref{eq:hajek} respectively satisfy
\begin{equation}
\label{eq:direct_result}
\begin{split}
&\htau^{\HT}_{\DIR} - \btau_{\DIR} = \frac{1}{n} \sum_{i = 1}^n \p{\frac{f_i(1,\pi)}{\pi} + \frac{f_i(0,\pi)}{1-\pi}} (W_i - \pi) \\
& \quad\quad\quad\quad+ \frac{1}{n} \sum_{i = 1}^n \p{\sum_{j \neq i} \frac{E_{ij}}{\sum_{k \neq j} E_{jk}}\p{f'_j(1,\pi) - f'_j(0,\pi)}} (W_i - \pi)+ \oo_p\p{\delta}, \\
&\htau^{\HAJ}_{\DIR} - \btau_{\DIR} = \frac{1}{n} \sum_{i = 1}^n \p{\frac{f_i(1,\pi)}{\pi} + \frac{f_i(0,\pi)}{1-\pi} - \EE{\frac{f_i(1,\pi)}{\pi} + \frac{f_i(0,\pi)}{1-\pi}} } (W_i - \pi)\\
& \quad\quad\quad\quad + \frac{1}{n} \sum_{i = 1}^n \p{\sum_{j \neq i} \frac{E_{ij}}{\sum_{k \neq j} E_{jk}}\p{f'_j(1,\pi) - f'_j(0,\pi)}} (W_i - \pi)+ \oo_p\p{\delta},
\end{split}
\end{equation}
where $\btau_{\DIR}$ is as defined in \eqref{eq:HH} and 
\begin{equation}
\label{eqn:delta}
\delta = \frac{B}{\sqrt{n \min_i N_i}} + \frac{B\sqrt{\sum_{i,j} \gamma_{i,j}}}{n \min N_i^{3/2}}, \quad\quad
\gamma_{i,j} = \sum_{k \neq i,j} E_{ik} E_{jk}.
\end{equation}
\end{lemm}

\begin{proof}
Here we only provide a sketch of proof for the Horvitz-Thompson estimator to illustrate the main idea. The full proof will be given in
\ifaos
\else
Appendix \ref{sec:proofs}.
\fi
To start, as justified by Assumption \ref{assu:smooth}, we can Taylor expand $f_i\p{w, M_i /N_i}$ into four terms,
\begin{equation}
\begin{split}
\label{eq:taylor}
 f_i\p{w, M_i/N_i} &= f_i(w, \pi) 
+ f'_i(w, \pi) \p{M_i/N_i - \pi} \\
& \qquad  + \frac{1}{2} f_i''(w, \pi) \p{M_i/N_i - \pi}^2 + 
r_i(w, M_i/N_i), 
\end{split}
\end{equation}
for any $w \in \cb{0,1}$, where $r_i(w, M_i/N_i) = \frac{1}{6} f_i'''(w_i, \pi_i^{\star}) \p{M_i/N_i - \pi}^3$
for some $\pi_i^{\star}$ between $\pi$ and $M_i/N_i$.
A careful application of this expansion to both \smash{$\htau^{\HT}_{\DIR}$} and \smash{$\btau_{\DIR}$} establishes that
\begin{equation}
\begin{split}
\label{eqn:lemma_decompose_short}
\htau^{\HT}_{\DIR} - \btau_{\DIR} 
&= \frac{1}{n} \sum_{i = 1}^n(W_i - \pi) \p{\frac{f_i\p{1, \pi}}{\pi} + \frac{f_i\p{0, \pi}}{1 - \pi}}\\
& \qquad \qquad +  \frac{1}{n} \sum_{i = 1}^n \p{ \frac{M_i}{N_i} - \pi} \p{f'_i\p{1, \pi} - f'_i\p{0,\pi} }
\\
& \qquad \qquad + \frac{1}{n} \sum_{i = 1}^n(W_i - \pi)\p{\frac{M_i}{N_i} - \pi} \p{\frac{f'_i\p{1, \pi}}{\pi} + \frac{f'_i\p{0, \pi}}{1 - \pi}}\\
& \qquad \qquad + S_1 + S_2 + \oo_p\p{\delta},
\end{split}
\end{equation}
where $S_1$ and $S_2$ are as given in \eqref{eq:S1and2}, and $\delta$ is as defined in \eqref{eqn:delta};
details of the derivation are given in
\ifaos
\else
Appendix \ref{sec:proofs}.
\fi

We observe that the first summand in \eqref{eqn:lemma_decompose_short} matches the first term in \eqref{eq:direct_result},
while the second summand can be rearranged as follows (while preemptively relabeling the summation index as $j$): With $d_j = f'_j\p{1, \, \pi} - f'_j\p{0, \, \pi}$, we have
\begin{equation}
 \sum_{j = 1}^n \p{\frac{M_j}{N_j} - \pi} d_j
  =  \sum_{j= 1}^n \frac{\sum_{i \neq j} E_{ij}\p{W_i - \pi}}{\sum_{k \neq j} E_{jk}} d_j =   \sum_{i = 1}^n  \p{W_i - \pi} \sum_{j \neq i} \frac{E_{ij}}{\sum_{k \neq j} E_{jk}} d_j.
\end{equation}
Thus the first two summands in \eqref{eqn:lemma_decompose_short} complete our target expression.

Now, the third summand can be rewritten into a quadratic form in $W_i - \pi$,
\begin{equation}
\frac{1}{n} \sum_{i}(W_i - \pi)\p{\frac{M_i}{N_i} - \pi} \zeta_i
= \frac{1}{n} \sum_{i \neq j} (W_i - \pi) M_{ij} (W_j - \pi), 
\end{equation}
where $\zeta_i = f'_i\p{1, \pi}/\pi + f'_i\p{0, \pi}/(1 - \pi)$, $M_{ij} =  \zeta_i E_{ij}/N_j$. Since the vector $W_i - \pi$ has independent and mean-zero entries, we can use the Hanson-Wright inequality
as stated in \citet{rudelson2013hanson} to verify that the above term is bounded in probability to order
$\Norm{M}_{\operatorname{F}}  /n$, which in turn is bounded as $\oo_p(B/\sqrt{n \min_i N_i})$. It remains
to control
\begin{equation}
\label{eq:S1and2}
\begin{split}
&S_1 = \frac{1}{2n} \sum_{i = 1}^n(W_i - \pi) \p{M_i/N_i - \pi}^2 \p{\frac{f''_i\p{1, \pi}}{\pi} + \frac{f''_i\p{0, \pi}}{1 - \pi}}, \\
&S_2 = \frac{1}{n} \sum_{i=1}^n(W_i - \pi)  \p{\frac{r_i(1, M_i/N_i)}{\pi} + \frac{r_i(0, M_i/N_i)}{1 - \pi}}.
\end{split}
\end{equation}
Here, both $S_1$ and $S_2$ have the form of $\sum_{i} (W_i - \pi) \alpha_i(M_i/N_i)/n$, where the
function $\alpha_i$ is measurable with respect to $\{f_j\}_{j = 1}^n$. We will use Proposition \ref{prop:function_mi_ni}
stated below to bound them. In doing so recall that by properties of the Binomial distribution there are constants $C_k$
such that \smash{$\EE{(M_i/N_i - \pi)^{2k} \mid G, f(\cdot)} \leq C_k/N_i^k$} for all $k = 1, \, 2, \, \ldots$\, Thus, by
Assumption \ref{assu:smooth}, $\mathbb{E}[(M_i/N_i - \pi)^4 f''_i\p{1, \pi}^2\mid G, f(\cdot)] \leq C_2B^2/N_i^2$ and
$\mathbb{E}[r_i(w, M_i/N_i)^2 \mid G, f(\cdot)] \leq C_3 B^2/N_i^3$, giving us the needed second moment bounds on $\alpha_i(M_i/N_i)$.
\end{proof}

\begin{prop}
\label{prop:function_mi_ni}
Under the conditions of Lemma \ref{lemm:direct},
let $\alpha_i : [0,\, 1] \to \RR$ be measurable with respect to \smash{$\{f_j\}_{j = 1}^n$}, and suppose that
\smash{$\mathbb{E}[\alpha_i(M_i/N_i)^2 \mid G, f(\cdot)] \leq C B^2/N_i^2$} almost surely for some universal constant $C$.
Then, conditionally on $G$ and \smash{$\cb{f_i(\cdot)}_{i = 1}^n$},
\begin{equation} 
\frac{1}{n} \sum_{i = 1}^n (W_i - \pi) \, \alpha_i(M_i/N_i) = \oo_p\p{ \frac{B}{\sqrt{n \min_i N_i}} }. 
\end{equation}
\end{prop}

A sufficient set of conditions for $\delta$ to be negligible is the following: 
If the minimum degree of the exposure graph is bounded from below as $\min_i\cb{N_i} = \Omega_p(n \rho_n)$ and the number of common neighbors $\gamma_{i,j}$ satisfies $\sum_{i,j}\gamma_{i,j} = \oo_p(n^3 \rhon^2)$, then the $\delta$ term in Lemma \ref{lemm:direct} obeys 
\smash{$\delta = \oo_p({B} \,/\,{\sqrt{n^2 \rhon}})$}. 
Under our graphon generative model (Assumptions \ref{assu:random_graph} and \ref{assu:graphon}), then \eqref{eqn:G_lower_bound} and \eqref{eqn:G_upper_bound} as used in Theorem \ref{theo:directCLT} below imply the above conditions.

The characterization of Lemma \ref{lemm:direct} already gives us some intuition about the behavior of
estimators of the direct effect. In the setting without interference, it is well known that the Horvitz-Thompson
estimator satisfies
\begin{equation}
\htau^{\HT} - \btau = \frac{1}{n} \sum_{i} \p{\frac{Y_i(1)}{\pi} + \frac{Y_i(0)}{1-\pi}} (W_i - \pi),
\end{equation}
where $\btau = \frac{1}{n} \sum_{i = 1}^n \p{Y_i(1) -Y_i(0)}$ is the sample average treatment effect, and
a similar expression is available for the H\'ajek estimator. Here we found that, under interference,
\smash{$\htau^{\HT}_{\DIR}$} preserves this error term, but also acquires a second one that involves
interference effects. Qualitatively, the term $\sum_{j \neq i} {E_{ij}}(f'_j(1,\pi) - f'_j(0,\pi)) \,/\,({\sum_{k \neq j} E_{jk}})$,
captures the random variation in the outcomes experienced by the neighbors of the $i$-th unit due to the treatment $W_i$ assigned
to the $i$-th unit.

It is now time to leverage our graphon generative model. The following result uses this assumption
to characterize the behavior of the terms given in Lemma \ref{lemm:direct}, and to establish a
central limit theorem that highlights how interference effects play into the asymptotic variance of
estimators of the direct effect.

\begin{theo}
\label{theo:directCLT}
Consider a randomized trial under network interference satisfying Assumptions
\ref{assu:undirected}--\ref{assu:smooth}, with treatment assigned independently as
$W_i \sim \text{Bernoulli}(\pi)$ for some $0 < \pi < 1$. Suppose that the function
$g_1(u) := \int_0^1\min(1,\Gfcnc(u,t)) dt$ is bounded away from 0, 
\begin{equation}
\label{eqn:G_lower_bound}
g_1(u_1) \geq \clower \text{ for any } u_1,
\end{equation} 
and that the graphon has a finite second moment, i.e. 
\begin{equation}
\label{eqn:G_upper_bound}
\EE{\Gfcnc(U_1, U_2)^k} \leq \cupper^k, \text{ for } k = 1, 2.
\end{equation}
Finally, suppose that
	\smash{$\liminf  { \log\rhon}\,/\,{\log n} > - 1$}.
Then, both the Horvitz-Thompson and H\'ajek estimators
have a limiting Gaussian distribution around the direct effect \eqref{eq:HH},
\begin{equation}
\label{eq:CLT_DIR}
\begin{split}
&\sqrt{n}\p{\hat{\tau}_{\DIR}^{\HT} -\bar{\tau}_{\DIR}} \Rightarrow \mathcal{N}\p{0, \pi(1-\pi) \EE{(R_i+Q_i)^2}},\\
&\sqrt{n}\p{\hat{\tau}_{\DIR}^{\HAJ} -\bar{\tau}_{\DIR}} \Rightarrow \mathcal{N}\p{0, \pi(1-\pi)\p{\Var{R_i+Q_i} + \p{\EE {Q_i}}^2}}.
\end{split}
\end{equation}
where
\begin{equation}
\begin{split}
&R_i = \frac{f_i(1,\pi)}{\pi} + \frac{f_i(0,\pi)}{1-\pi},\ \ \ \
Q_i = \EE{\frac{\Gfcnc(U_i, U_j)(f'_j(1,\pi) - f'_j(0,\pi))}{\gsmlc(U_j)} \,\Big|\, U_i}.
\end{split}
\end{equation}
If furthermore $\sqrt{n} \rhon \to \infty$, then similar results hold for the population-level estimand \eqref{eq:pop_estimand},
with $\sigma_0^2 = \Var{f_i(1,\pi) - f_i(0,\pi)}$:
\begin{equation}
\label{eq:CLT_DIR_POP}
\begin{split}
&\sqrt{n}\p{\hat{\tau}_{\DIR}^{\HT} - {\tau}_{\DIR}} \Rightarrow \mathcal{N}\p{0, \sigma_0^2 +  \pi(1-\pi) \EE{(R_i+Q_i)^2}},\\
&\sqrt{n}\p{\hat{\tau}_{\DIR}^{\HAJ} - {\tau}_{\DIR}} \Rightarrow \mathcal{N}\p{0, \sigma_0^2 + \pi(1-\pi)\p{\Var{R_i+Q_i} + \EE{Q_i}^2}}.
\end{split}
\end{equation}
\end{theo}

Our first observation is that, in contrast to the upper bounds of \citet*{savje2017average},
our additional random graph assumptions, paired with anonymous interference and smoothness,
enable us to establish that
the asymptotic accuracy of the estimators does not depend on the sparsity level $\rhon$.
In our setting, direct effects are accurately estimable even in dense graphs, and
$\rhon$ at most influences second-order convergence to the Gaussian limit.

To further interpret this result, we note that, in the case without interference (i.e., omitting all contributions of $Q_i$),
the results \eqref{eq:CLT_DIR} and \eqref{eq:CLT_DIR_POP} replicate well known results about estimators for the average
treatment effects.  In general, unless $R_i$ and $Q_i$ are strongly negatively correlated, then we would expect
and $\EE{(R_i + Q_i)^2} > \EE{R_i^2}$ and $\Var{R_i+Q_i} + \EE {Q_i}^2 >  \Var{R_i}$, meaning that interference
effects inflate the variance of both the Horvitz-Thompson and H\'ajek estimators. However, it is possible to
design special problem instances where interference effects in fact reduce variance.
The variance inflation between \eqref{eq:CLT_DIR} and \eqref{eq:CLT_DIR_POP} arises from
targeting $\tau_{\DIR}$ versus $\btau_{\DIR}$. This corresponds exactly to the familiar variance inflation term
that arises from targeting the average treatment effect as opposed to the sample average treatment effect
in the no-interference setting; see \citet{imbens2004nonparametric} for a discussion. The condition
$\sqrt{n} \rhon \to \infty$ for \eqref{eq:CLT_DIR_POP} is required to make the error term in \eqref{eq:HH_graphon} small.

\sloppy{Theorem \ref{theo:directCLT} also enables us to compare the asymptotics of the Horvitz-Thompson and
H\'ajek estimators. Here, interestingly, the picture is more nuanced. The asymptotic variance of the
Horvitz-Thompson estimator depends on \smash{$\EE{(R_i+Q_i)^2} = \Var{R_i+Q_i} + \EE{R_i+Q_i}^2$}, and so the
H\'ajek estimator is asymptotically more accurate than the Horvitz-Thompson estimator if and only if
\smash{$\EE{Q_i}^2 \leq \EE{R_i+Q_i}^2$}.
Thus, neither estimator dominates the other one in general.
This presents a marked contrast to the case without interference, where the H\'ajek estimator always
has a better asymptotic variance than the Horvitz-Thompson estimator (unless $\EE{R_i} = 0$, in which case
they have the same asymptotic variance).}

One question left open above is how to estimate the asymptotic variances that arise in
Theorem \ref{theo:directCLT}, which depends on unknown functionals of $G$ and the $f_i$ that may be
difficult to estimate.\footnote{The fundamental difficulty here is that interference creates intricate dependence
patterns that break standard strategies for variance estimation. In particular, standard non-parametric bootstrap or
subsampling-based methods do not apply here, because removing the $i$-th datapoint from the sample does not erase the
spillover effects due to the treatment $W_i$ received by the $i$-th person. Thus, in order to estimate the asymptotic
variance, it is likely one would need to develop plug-in estimators for the expressions in \eqref{eq:CLT_DIR_POP}.
The main challenge in doing so is with terms of the form \smash{$\mathbb{E}[Q_i^2]$}, because
the $Q_i$ involve derivatives of the potential outcome function $f_i$ with
respect to the fraction of treated neighbors. Estimating the first moment of $Q_i$ appears to be a task whose
statistical difficulty is comparable to estimating the indirect effect $\btau_{\IND}$, while estimating its second moment poses
further challenges. We leave a discussion of point-estimators for these quantities to further work.}
However, even when point estimation of the asymptotic variance is
difficult, we may be able to use subject matter knowledge to derive practically useful upper bounds for this
asymptotic variance that can be paired with Theorem \ref{theo:directCLT} to build asymptotically conservative
confidence intervals for the direct effect. We further investigate this approach below in the context of an application.

\subsection{Numerical Evaluation}

To validate our findings from Theorem \ref{theo:directCLT}, we consider a simple numerical example.
Here, we simulate data as described in Section \ref{section:stats_setting}, for a graph with $n = 1000$ nodes
generated via a constant graphon $\Gfcn(u_1,u_2) = 0.4$, i.e., where any pair of nodes are connected with probability 0.4.
We then generate treatment assignments as $W_i \smash{\simiid} \text{Bernoulli}(\pi)$ with $\pi = 0.7$, and potential outcome
functions as $f_i(w, \, x) = w x / \pi^2 + \varepsilon_i$ with $\varepsilon_i \sim \mathcal{N}(0,\,1)$.
Figure \ref{fig:DE} shows the distribution of the estimators $\tauht_{\DIR}$ and $\tauha_{\DIR}$ across $N = 3000$ simulations.
We see that the distribution of the estimators closely matches the limiting Gaussian distribution from
Theorem \ref{theo:directCLT} (in red). In contrast, a simple analysis that ignores interference effects would
result in a limiting distribution (shown in blue) that's much too narrow.
In other words, here, ignoring interference would lead one to underestimate the variance of the estimator. 

\begin{figure}[t]
    \centering
    \includegraphics[width = \textwidth]{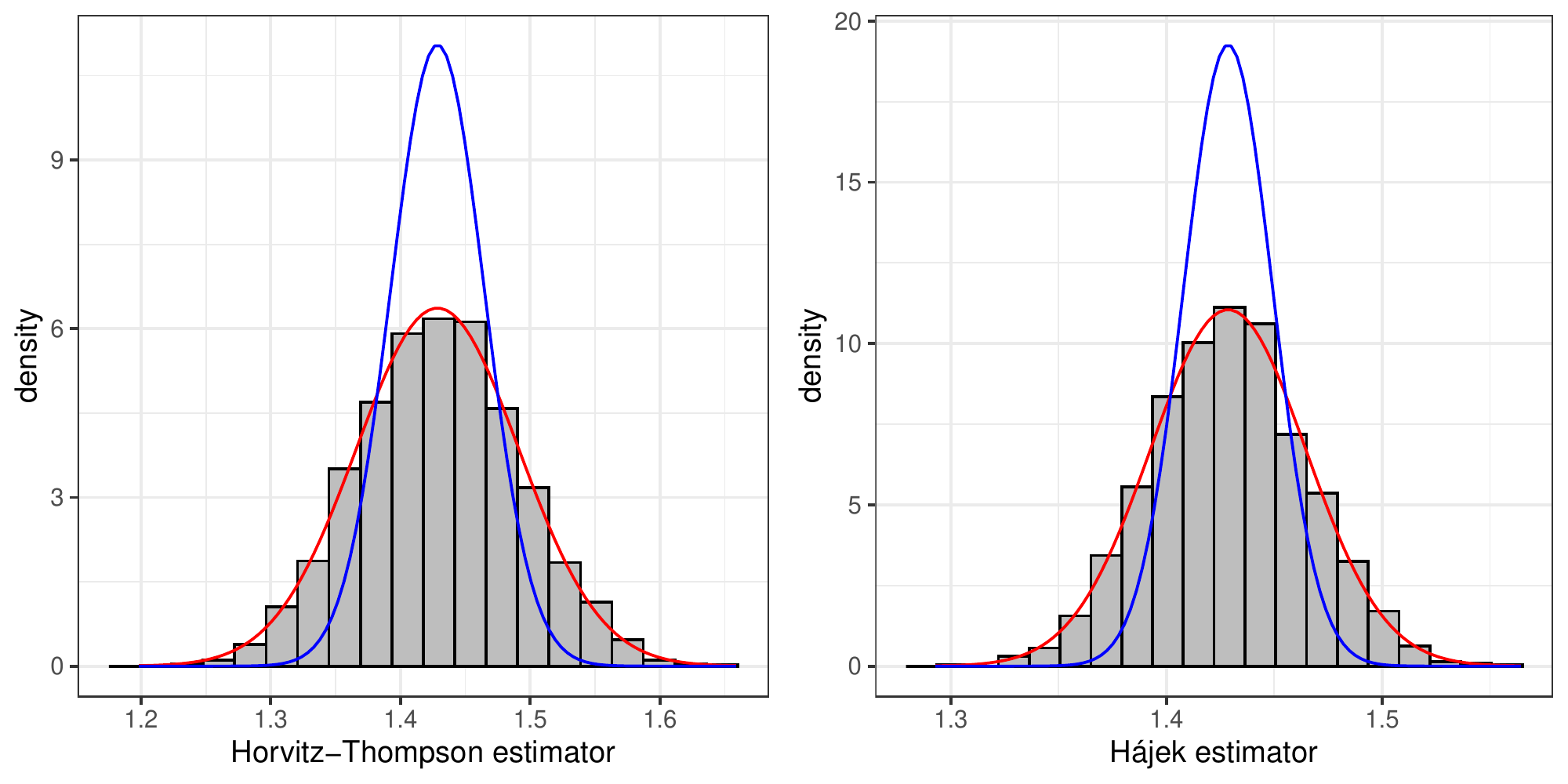}
    \caption{Histogram of $\tauht_{\DIR}$ and $\tauha_{\DIR}$ across $N = 3000$ replications, for a graph of size $n = 1,000$.
    The overlaid curves denote, in red, the limiting Gaussian distribution derived in Theorem \ref{theo:directCLT} and,
    in blue, the limiting Gaussian distribution we would get while ignoring interference effects.}
    \label{fig:DE}
\end{figure}

\section{Conservative Intervals for the Direct Effect}
\label{sec:sensitivity}

Theorem \ref{theo:directCLT} implies that, under our random graph model, accurate estimation of the
direct effect is simple and practical. The estimators \smash{$\hat{\tau}_{\DIR}^{\HT}$} and 
\smash{$\hat{\tau}_{\DIR}^{\HAJ}$} have an elementary functional form, and do not explicitly depend on the
interference graph $E$ and so can be implemented even if we only have incomplete or potentially inaccurate
knowledge of it. Using Theorem \ref{theo:directCLT} to build confidence intervals, however, is more challenging.
First, the relevant asymptotic variance depends on the graph $E$ and will be difficult to estimate if we don't have
accurate information about it. Second, as discussed above, even in an ideal setting where $E$ is known,
estimating the asymptotic variance may pose challenges and consistent point estimators are not currently available.

In this section, we explore an alternative approach to using Theorem \ref{theo:directCLT} in practice,
based on conservative bounds for unknown components in the asymptotic variance.
We illustrate this strategy using a study by \citet*{duflo2013truth} on how
environmental regulations can help curb industrial pollution in Gujarat, India.
Our main finding is that, in this example, we can translate reasonably weak assumptions on the high-level structure
of the interference graph $E$ into meaningful bounds on the asymptotic variance of estimators of the direct effect.
We believe that similar bounds-based strategies may also be useful in other applications.

\citet{duflo2013truth} start from
a status quo where the state had specified
limits on how much firms may pollute, and industrial plants needed to hire independent auditors to verify compliance. The authors were
concerned, however, about a conflict of interest: Because plants hire their own auditors, the auditors may be incentivized
to turn a blind eye to potential non-compliance in order to get hired again.
To test this hypothesis, \citet{duflo2013truth} considered a sample of
$n = 473$ audit-eligible plants in Gujarat, and randomly assigned half of these plants (i.e., $\pi = 0.5$) to a treatment designed to make
auditors work more independently, while the control group remained with the status quo. The treatment had multiple components,
including pre-specifying the auditor (instead of letting the plants hire their own auditors), and using a fixed fee
rather than a fee negotiated between the plant and the auditor; see \citet{duflo2013truth} for details. The authors found
a substantial effect of changing the audit mechanism. In particular, they found that plants in the treatment condition reduced
combined water and air pollutant emissions by $\htau = 0.211$ standard deviations of the pollutant emission distribution
for control plants, with an associated standard error estimate of 0.99 and a 95\% confidence interval $\tau \in (0.017, \, 0.405)$.

The analysis used in \citet{duflo2013truth} did not consider interference effects, i.e., it assumed that enrolling a specific plant $i$
in the treatment condition only affected pollution levels for the $i$-th plant, not the others. This is, however, a potentially
problematic assumption: For example, one might be concerned that some plants are in close contact with each other, and
that having one plant be enrolled in the treatment condition would also make some of its closely associated plants reconsider
non-compliant pollution.

If we're worried about interference, how should we reassess the point estimate $\htau = 0.211$? What about the associated
confidence interval? As discussed in Section \ref{section:dir}, the work of \citet*{savje2017average} already
provides a good answer to the first question: In the presence of interference, we should understand $\htau = 0.211$ as an
estimate of the direct effect of treatment, while marginalizing over the ambient treatment assigned to other
plants. The answer to the second question is more delicate. As shown in Theorem \ref{theo:directCLT}, in this case, the width
of confidence intervals built around $\htau$ need to be adjusted to account for interference; however, we do not have
access to estimators for the variance parameters in \eqref{eq:CLT_DIR_POP}---and in fact, here, we don't even observe
the interference graph $E$. Thus, in order to assess the sensitivity of the findings in \citet{duflo2013truth},
the best we can hope for is to pair the structure of our result from Theorem \ref{theo:directCLT} with subject-matter knowledge
in order to derive conservative bounds for the variance inflation induced by interference.\footnote{\citet{savje2017average}
also consider bounds for the variance of $\htau$, but they are not sharp enough to quantitatively engage with the confidence interval
of \citet{duflo2013truth}. More specifically, \citet{savje2017average} propose a number variance estimators that take the form of the
product of an inflation factor $\alpha$ and the baseline variance estimator
$\smash{\widehat{S} = 1/n^2(\sum_i W_i Y_i^2/\pi_i^2 + \sum_i (1-W_i) Y_i^2/(1-\pi_i)^2)}$, i.e., they use
$\smash{\widehat{\operatorname{Var}}[\htau] \leq \alpha \widehat{S}}$. With this approach, however, the inflation factor $\alpha$
can be large: If we follow the definition of the $H$ matrix as in \eqref{eq:savje_bound}, and define $h_i = \sum_{j \neq i} H_{ij}$ to
be the number of second order neighbors, then choices of the inflation factor include the average of $h_i$, the maximum of $h_i$,
and the largest eigenvalue of the $H$ matrix. Then, for example, in a simple disjoint-community model where there are 20 communities and
roughly 25 plants in each community, the inflation factor would be roughly 25---meaning that confidence interval would need
to be widened by a factor of 5 to accommodate interference. Furthermore $\smash{\widehat{S}}$ itself is conservative
even without interference.}


To this end recall that, under the assumptions of Theorem \ref{theo:directCLT}, the H\'ajek estimator $\htau$ satisfies
a central limit theorem\footnote{The estimate $\htau = 0.211$ of \citet{duflo2013truth} was derived from a linear regression with
fixed effects for sub-regions of Gujarat. The regression also included multiple observations per plant, and then clustered
standard errors at the plant level. Here, we conduct a sensitivity analysis as though the point estimate $\htau$ had been
derived via a H\'ajek estimator, which is equivalent to linear regression without fixed effects and without multiple observations
per plant. It is likely that a sensitivity analysis that also took into account fixed effects and repeated observations would
give a similar qualitative picture, but our formal results are not directly applicable to that setting.}
\begin{equation*}
\label{eq:dir_recap}
\begin{split}
&\sqrt{n}\p{\htau - \btau_{\DIR}} \Rightarrow \nn\p{0, \, \sigma_0^2 + \pi (1 - \pi) V}, \ \ \ \ V = \Var{R_i} +2 \Cov{R_i, \, Q_i} + \EE{Q_i^2} \\
&Q_i = Q(U_i) = \EE{\frac{\Gfcnc(U_i, U_j)(f'_j(1,\pi) - f'_j(0,\pi))}{\gsmlc(U_j)} \,\Big|\, U_i}.
\end{split}
\end{equation*}
Now, let $V_0 = \Var{R_i}$ measure the asymptotic variance of the H\'ajek estimator for the average treatment effect that ignores
interference effects and note that, by Cauchy-Schwarz,\footnote{One
point left implicit here is that standard variance estimators that ignore interference should be seen as
estimators of $\sigma_0^2 + \pi (1 - \pi) V_0$ in our model. We discuss this point further in
\ifaos
\else
Appendix \ref{section:variance_estimator},
\fi and provide a formal result for the basic plug-in variance estimator that could be used without interference.}
\begin{equation}
\label{eq:CS_bound}
V \leq V_0 + 2\sqrt{V_0  \EE{Q_i^2}} + \EE{Q_i^2}.
\end{equation}
Thus, if we use the original standard error estimate from \citet{duflo2013truth} for $V_0$,
i.e., we set $(\sigma_0^2 + \pi ( 1 - \pi) V_0)/n = 0.099^2$,
then bounding the asymptotic variance term in \eqref{eq:dir_recap} reduces to bounding $\EE{Q_i^2}$.

It now remains to develop useful bounds for $\EE{Q_i^2}$, in terms of assumptions on both the graphon $G$
and the potential outcome functions $f_i(w, \, \pi)$. To gain an understanding of the trade-offs at play here,
we consider one example where the variance inflation due to interference is exactly zero, one with non-zero but
manageable variance inflation, and one where the variance inflation may get out of control easily.

\begin{exam}[Additive Interference]
Suppose interference is additive and that units respond to their neighbors treatment
in the same way regardless of their own treatment status, i.e., $f_i(w, \, \pi) = a_i(w) + b_i(\pi)$.
Then \smash{$f'_j(0,\pi) = f'_j(1,\pi)$} and \smash{$\EE{Q_i^2} = 0$}, meaning that interference has no effect on the
asymptotic variance in \eqref{eq:dir_recap}, i.e., $V = V_0$. Thus interference only affects the precision of the
H\'ajek estimator if $f_i(w, \, \pi)$ is non-additive in its arguments, regardless of the graphon $G$.
\end{exam}

\begin{exam}[Disjoint Communities Model]
\label{ex:disjoint}
Now suppose that we can divide the graphon $G$ into $k = 1, \ ..., \, K$ non-overlapping communities,
such that the $G$ is the sum of an overall rank-1 term and $K$ community-specific rank-1 terms. More
specifically, we assume that there exist intervals $I_k \subset [0, \, 1]$ such that
$G(u, \, v) = a_0(u)a_0(v) + \sum_{k = 1}^K 1\p{\cb{u, \, v \in I_k}} a_k(u) a_k(v)$ for some
functions $a_k : I_k \rightarrow [0,\, 1]$. Given this setting, we can check that for any $u \in I_k$
\begin{equation*}
g(u) =  a_0(u) \bar{a}_0 +  a_k(u) \bar{a}_k, \ \ \ \ \bar{a}_0 =  \int_{0}^1 a_0(v) \ dv, \ \ \ \ \bar{a}_k =  \int_{I_k} a_k(v) \ dv, 
\end{equation*}
and so
\begin{equation*}
\begin{split}
Q(u) &= \int_0^1 \frac{a_0(u)a_0(v) + 1\p{\cb{v \in I_k}} a_k(u)a_k(v)}{a_0(v) \bar{a}_0 +  \sum_{k' = 1}^K 1\p{\cb{v \in I_{k'}}} a_{k'}(v) \bar{a}_{k'}} \EE{f'_j(1,\pi) - f'_j(0,\pi)) \cond U_j = v} \\
&\leq \frac{a_0(u)}{\bar{a}_0} \EE{f'_j(1,\pi) - f'_j(0,\pi))} + \frac{a_k(u)}{\bar{a}_k}  \EE{f'_j(1,\pi) - f'_j(0,\pi)) ; U_j \in I_k}.
\end{split}
\end{equation*}
Pursuing this line of reasoning and applying Cauchy-Schwarz, we then find that
\begin{equation}
\label{eq:disjoint_bound}
\begin{split}
&\frac{1}{2} \EE{Q_i^2} \leq \frac{\EE{a_0^2(U_i)}}{\EE{a_0(U_i)}^2}  \EE{f'_i(1,\pi) - f'_i(0,\pi))}^2 \\
& \ \ \ \ \ \ \ \ + \sum_{k = 1}^K \PP{U_i \in I_k}\frac{\EE{a_k^2(U_i) \cond U_i \in I_k}}{\EE{a_k(U_i) \cond U_i \in I_k}^2}  \EE{f'_i(1,\pi) - f'_i(0,\pi)) \cond U_i \in I_j}^2.
\end{split}
\end{equation}
In other words, we've found that \smash{$\EE{Q_i^2}$} can be bounded in terms of moments of $f'_i(1,\pi) - f'_i(0,\pi)$
across the disjoint communities, and in terms of the coefficient of variation of the functions $a_k(u)$ that determine
the average degree of different nodes.
\end{exam}

\begin{exam}[Star Graphon]
We end with star-shape interference graphs, and find that they exhibit strong variance inflation due to interference. Pick some small $\eta > 0$ and
some $a \in (0, \, 1)$, and let $G(u, \ v) = 1\p{\cb{u \leq \eta \text{ or } v \leq \eta}} \, a$. Then $g(u) = a$ if
$u \leq \eta$ and $g(u) = \eta a$ if $u > \eta$, meaning that
\begin{equation*}
\begin{split}
Q(u) &= 
\int_0^\eta \EE{f'_j(1,\pi) - f'_j(0,\pi)) \cond U_j = v} \ dv \\
&\ \ \ \ \ \ \ \ \ \ \ \ \ \ \ + \eta^{-1} 1 \p{\cb{u \leq \eta}} \int_\eta^1 \EE{f'_j(1,\pi) - f'_j(0,\pi)) \cond U_j = v} \ dv.
\end{split}
\end{equation*}
Then, in the limit where the nucleus of the ``star'' gets small, i.e., $\eta \rightarrow 0$, we see that
\begin{equation*}
\lim_{\eta \rightarrow 0} \eta \EE{Q_i^2} =  \EE{f'_i(1,\pi) - f'_i(0,\pi)}^2,
\end{equation*}
i.e., the variance inflation term diverges at rate $\eta^{-1}$. The reason this phenomenon occurs is that
the treatment assignments for a small number units in the nucleus has a large effect on the outcomes
of everyone in the system, and this leads to a considerable amount of variance.
\end{exam}

In order to study the sensitivity of the findings of \citet{duflo2013truth} to interference, we first need to choose
some high-level assumptions on $G$ to work with. Here, we move forward in the setting of Example \ref{ex:disjoint},
i.e., under the assumption that interference effects is dominated by links between disjoint and unstructured communities.
It thus remains to bound the terms in \eqref{eq:disjoint_bound}. We consider the following:
\begin{enumerate}
\item We assume that both the main effects and interference effects are negative (i.e., independent audits reduce
pollution overall), and that indirect effects are weaker than the main effects,
i.e., $f_i(1, \, \pi) - f_i(0, \, \pi) \leq   f'_i(0, \, \pi), \,  f'_i(1, \, \pi) \leq 0$,
and in particular \smash{$\EE{f'_i(1,\pi) - f'_i(0,\pi))}^2 \leq \tau_{\DIR}^2$}.
\item We assume that all terms in \eqref{eq:disjoint_bound} that depend on stochastic fluctuations
of $a_k(u)$ and $\EE{f'_i(1,\pi) \cond U_i = u}$ can be controlled by considering these terms
constant and then inflating the resulting bound by a factor 2.
\end{enumerate}
Pooling all this together, we get that
\begin{equation}
\label{eq:sensi_bound}
\EE{Q_i^2} \leq 8 \, \tau^2_{\DIR}.
\end{equation}
We do not claim that all the steps leading to \eqref{eq:sensi_bound} are all undisputable, but simply that it's a
potentially reasonable starting point for a sensitivity analysis; other subject-matter knowledge may lead to other alternatives to
\eqref{eq:sensi_bound} that can be discussed when interpreting results of an application.

Our final goal is the to use this bound on $\EE{Q_i^2}$ to see how much we might need to inflate confidence intervals to account
for interference. To do so, we proceed by inverting a level-$\alpha$ hypothesis test. Given \eqref{eq:CS_bound} and
\eqref{eq:sensi_bound}, the following chi-squared test will only reject with probability at most $\alpha$ under the null-hypothesis
$H_0: \tau_{\DIR} = \tau_0$:
\begin{equation}
\label{eq:hypt}
1\p{\cb{\p{\htau - \tau_0}^2 \geq \frac{\Phi(1- \alpha/2)^2}{n} \p{\sigma_0^2 + \pi(1 - \pi)\p{V_0 + 2\sqrt{8V_0\tau_0^2} + 8\tau_0^2}}}}
\end{equation}
In our specific case, recall that $\pi = 1/2$ and $n = 473$, and we assumed that $(\sigma_0^2 + \pi ( 1 - \pi) V_0)/n = 0.099^2$.
This leaves the relationship between $\sigma_0^2$ and $V_0$ unspecified; however, we can maximize the
noise term in \eqref{eq:hypt} by setting $\sigma_0^2 = 0$ and $V_0 = 4 \times 473 \times 0.99^2$, which is what we do here,
resulting in a fully specified hypothesis test,
\begin{equation}
\label{eq:hyp_quant}
1\p{\cb{\p{\htau - \tau_0}^2 \geq \Phi(1- \alpha/2)^2 \p{0.0098 + 0.0129 \abs{\tau_0} + 0.0042\tau_0^2}}},
\end{equation}
which we can now invert.

\begin{figure}[t]
	\centering
	\includegraphics[width = 0.6 \textwidth]{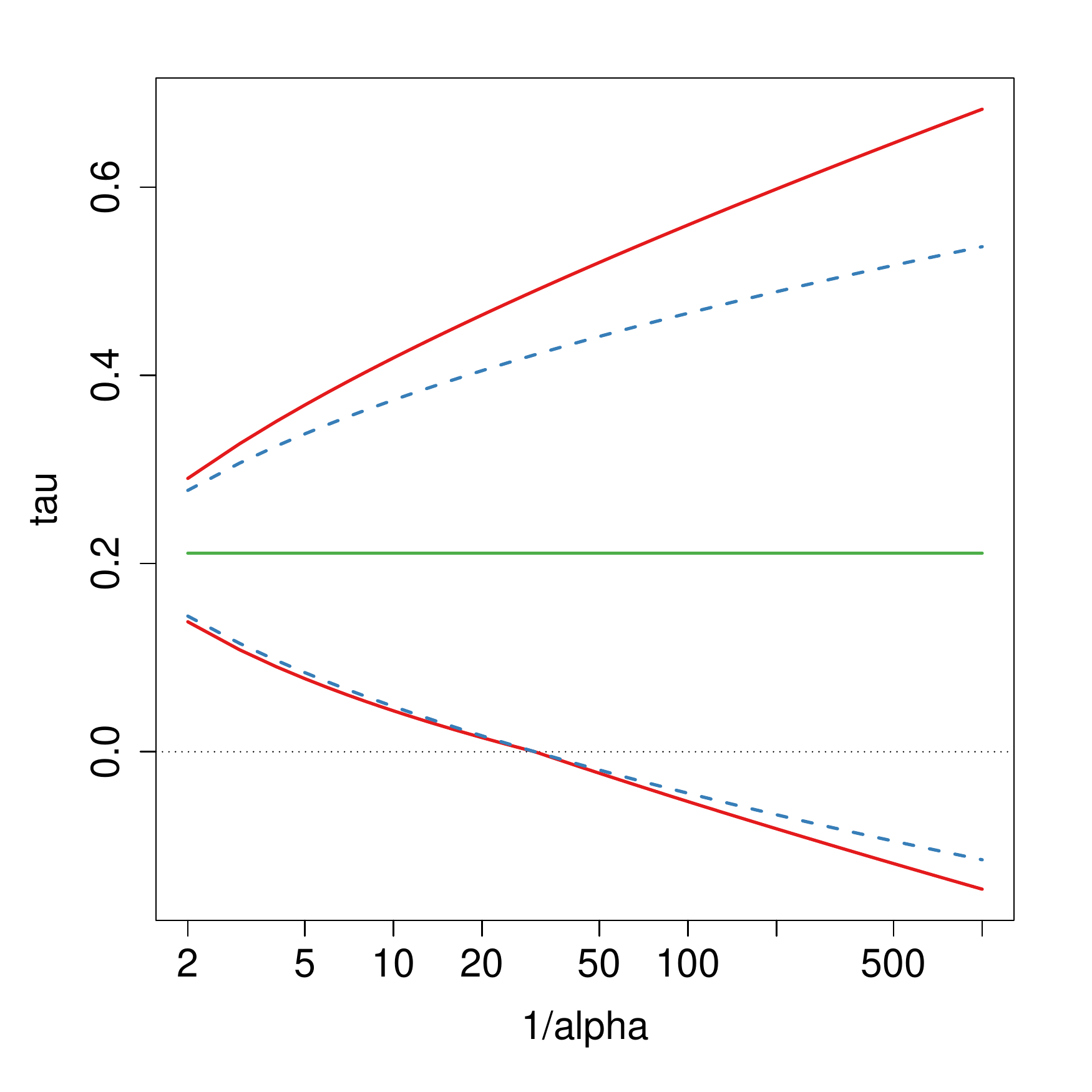}
	\caption{Level-$\alpha$ confidence intervals for the direct effect $\tau_{\DIR}$ in the setting of
	\citet{duflo2013truth}. The dashed blue lines denote upper an lower endpoints of a basic Gaussian
	confidence interval with standard error estimate 0.099, while the solid red curves denote intervals
	endpoints of a confidence interval derived by inverting \eqref{eq:hyp_quant}. The solid line at $\tau = 0.211$
	denotes the point estimate.}
	\label{fig:sensitivity}
\end{figure}

By applying this strategy, we obtain the following interference-robust 95\% confidence interval:
$\tau_{\DIR} \in (0.015, \,  0.464)$. Recall that, in contrast, the unadjusted Gaussian confidence
interval was $(0.017, \, 0.405)$. Interestingly, while our interference adjustment noticeably increased
the upper endpoint of this interval, it barely touched the lower endpoint at all; and, as a consequence of this,
we are still able to reject the null that $\tau_{\DIR} = 0$ at the 95\% level. Figure \ref{fig:sensitivity}
shows the intervals obtained by inverting \eqref{eq:hyp_quant} for different significance levels $\alpha$.

The reason our confidence intervals are less sensitive to interference as we approach 0 is that we assumed above
that indirect effects should be bounded on the order of direct effects; thus, when testing a null hypothesis that
direct effects are very small, the variance inflation due to indirect effects should also be small. If one were to
make different assumptions (e.g., that indirect effects may be large even when direct effects are small), a sensitivity
analysis might lead to different conclusions.

A thorough sensitivity analysis of the robustness of the results of \citet{duflo2013truth} to interference would
also involve examining different assumptions on the graphon, etc., and comparing findings across settings.
However, we hope that our discussing so far has helped highlight the promise of using random graph analysis to
quantitatively and usefully assess robustness of treatment effect estimators to potential variance inflation due to interference.

\section{Estimating the Indirect Effect}
\label{sec:indirect}

We now consider estimation of the indirect effect, i.e., how a typical unit responds to a change in its fraction of treated neighbors.
There is some existing literature on this task; however, it has mostly focused on a setting where one has access to many
independent networks  \citep*{baird2018optimal,basse2018analyzing,hudgens2008toward,tchetgen2012causal}.
This is also referred to as a partial interference assumption, which states that there are disjoint groups of units and
spillover across groups is not allowed. Then, the total effect---and thus also the
indirect effect---can be identified by randomly varying the treatment probability $\pi$ across different groups and regressing the mean
outcome in each group against its treatment probability.

Meanwhile, in the single network setting, we note a recent paper by \citet{leung2020treatment}, who studies estimation of both
direct and indirect effects when the degree of the exposure graph remains bounded as the sample size $n$ gets large.
He then proposes an estimator that is consistent and has a $1/\sqrt{n}$
rate of convergence. At a high level, the motivating insight behind his approach is that, in the case of a bounded-degree interference graph,
we'll be able to see infinitely many (linear in $n$) units for any specific treatment signature consisting of the number of neighbors,
the number of treated neighbors and the treatment allocation. Hence we can take averages of outcomes with a given treatment signature,
and use them to estimate various causal quantities. This strategy, however, does not seem to be extensible to denser graphs.

Our goal here is to develop methods for estimating the indirect effect that can work with a single network that is much denser than
those considered by \citet{leung2020treatment}, i.e., $n\rho_n \gg 1$ following Assumption \ref{assu:graphon}.
We are not aware of any existing results in this setting. Our main contribution is an estimator, the PC balancing estimator, that can be
used to estimate the indirect effect in a setting where the graphon $G$ admits a low-rank representation, i.e.,
$G(u, \, v) = \sum_{k = 1}^r \lambda_k \psi_k(u) \psi_k(v)$ for a small number $r$ of measurable functions $\psi_k : [0, \, 1] \rightarrow \RR$.
We prove that our estimator converges to the indirect effect at rate $\sqrt{\rho_n}$ and satisfies a central limit theorem.
At a high level, the reason we are able to consistently estimate the indirect effect from a single graph is that, even with reasonably
dense graphs, some units will have a higher proportion of treated neighbors than others due to random fluctuations in the treatment
assignment mechanism---and our graphon generative assumptions enable us to carefully exploit this variation for consistent estimation.

\subsection{An Unbiased Estimator}
\label{subsection:unbiased}

We start by discussing a natural unbiased estimator for the indirect effect that starts
from a simple generalization of Horvitz-Thompson weighting. Recall that the total effect is
\begin{equation}
\label{eq:Vbar}
\btau_{\TOT}(\pi) = \frac{d}{d\pi} \bV(\pi), \ \ \ \  \bV(\pi) =  \frac{1}{n} \sum_i  \EE[\pi]{Y_i | Y(\cdot)}.
\end{equation}
For any $\pi' \in (0, \,1)$, the Horvitz-Thompson estimate of $\bV(\pi')$ is
\begin{equation}
\label{eq:Vhat}
\hV(\pi') = \frac{1}{n} \sum_{i = 1}^n Y_i \p{\frac{\pi'}{\pi}}^{M_i + W_i} \p{\frac{1-\pi'}{1-\pi}}^{(N_i - M_i) + (1-W_i)},
\end{equation}
where as usual $M_i$ is number of treated numbers and $N_i$ the number of neighbors.\footnote{One might
also be tempted to study the problem of off-policy evaluation in our setting, i.e., using notation from \eqref{eq:Vbar},
estimating $\bV(\pi')$ for $\pi' \neq \pi$. This, however, appears to be a difficult problem outside of very sparse graphs.
For example, in Proposition 9 of the first arXiv version of this paper, we showed the
estimator \eqref{eq:Vhat} for $\bV(\pi')$ diverges in a random graph model graph whenever its average degree
grows faster than $\log(n)$.}
Thus, as $\btau_{\TOT}(\pi)$ is the derivative of $\bV(\pi)$, one natural idea is to estimate $\btau_{\TOT}(\pi)$ by
taking the derivative of \smash{$\hV(\pi')$}:
\begin{equation}
\label{eqn:tau_tot_unbiased}
\begin{split}
\hat{\tau}^{\U}_{\TOT}(\pi) = \sqb{\frac{d}{d\pi'} \hV(\pi')}_{\pi' = \pi}
= \frac{1}{n} \sum_i Y_i \p{\frac{M_i+W_i}{\pi} - \frac{N_i - M_i + 1 - W_i}{1-\pi}}.
\end{split}
\end{equation}
One can immediately verify that this estimator is unbiased for $\btau_{\TOT}(\pi)$ (hence the superscript $U$)
by noting that \smash{$\hV(\pi')$} is unbiased for $\bV(\pi')$ following the line of argumentation used in \eqref{eq:HT_unb}.
We also note that unbiasedness in \eqref{eqn:tau_tot_unbiased} follows immediately from the argument of
\citet{stein1981estimation} applied to the binomial distribution.
Next, the unbiased estimator of the total effect can be naturally decomposed into two parts:
\begin{equation}
\label{eqn:tau_ind_unbiased_def}
\begin{split}
\hat{\tau}^{\U}_{\TOT} &= \frac{1}{n} \sum_i Y_i \p{\frac{M_i+W_i}{\pi} - \frac{N_i - M_i + 1 - W_i}{1-\pi}}\\
&=  \frac{1}{n} \sum_i Y_i \p{\frac{W_i}{\pi} - \frac{1-W_i}{1-\pi}} +  \frac{1}{n} \sum_i Y_i \p{\frac{M_i}{\pi} - \frac{N_i-M_i}{1-\pi}}
 = \hat{\tau}^{\HT}_{\DIR} + \hat{\tau}^{\U}_{\IND}
\end{split}
\end{equation}
Recalling that $\hat{\tau}^{\HT}_{\DIR}$ is unbiased for $\btau_{\DIR}$, we see that $\hat{\tau}^{\U}_{\IND} $ is also unbiased for $\btau_{\IND}$.

Unfortunately, however, despite its simple intuitive derivation and its unbiasedness, this estimator is not particularly accurate.
More specifically, as shown below, its variance goes to infinity as $n \to \infty$ wherever $\sqrt{n}\rhon \to \infty$;
in other words, this estimator is inconsistent even if most units in the graph only share edges with a fraction $1/\sqrt{n}$ of
other units.

\begin{prop}
\label{prop:unbiased_rate}
Let \smash{$\nu = \EE{\p{\pi f_1(1,\pi) + (1-\pi)f_1(0,\pi)} \gsmlc(U_1) }^2$}. If $\sqrt{n}\rhon \to \infty$, then
under the conditions of Theorem \ref{theo:directCLT}, 
\smash{$\Var{\htau^{\U}_{\IND}} \sim \nu n \rhon^2$} and \smash{$\Var{\htau^{\U}_{\TOT}} \sim \nu n \rhon^2$}.
\end{prop}


\subsection{The PC-Balancing Estimator}
\label{subsection:PC}

In order to develop a new estimator effect for the indirect effect that is robust to the variance explosion phenomenon documented in
Proposition \ref{prop:unbiased_rate}, we focus on a setting where the graphon is low rank with rank $r$, i.e., our graphon can be written in a form of 
\begin{equation}
\label{eqn:low_rank_graphon}
\Gfcnc(U_i, U_j) =  \sum_{k = 1}^r \lambda_k \psi_k(U_i) \psi_k(U_j)
\end{equation}
for some function $\psi_k$.
The low-rank condition \eqref{eqn:low_rank_graphon} quantifies an assumption that each unit
can be characterized using a small number ($r$) of factors, and that the probability of edge formation between
two units is a bilinear function of both of their factors. For example, in a social network, we may assume that
the probability of two people becoming friends is explained by a few factors including their education, experience, and personality.
Such low-rank factor models are a popular way of capturing unobserved heterogeneity; see \citet{athreya2017statistical} for a
recent discussion and references.

Now, in order to develop a consistent estimator, we first need to understand why the unbiased estimator fails.
To this end, consider a simple stochastic block model with $r$ communities, where any two units in the
same community are connected with probability $\rho_n$ while units in different communities are never
connected. Letting $k(i)$ denote the $i$-th unit's community, we can re-write \eqref{eqn:tau_ind_unbiased_def} as
 \begin{align}
\notag
\hat{\tau}^{\U}_{\IND}
  &=  \frac{1}{\pi (1-\pi) n} \sum_i (\mu_{k(i)} + (Y_i - \mu_{k(i)}))(M_i - \pi N_i)\\
   \label{eqn:why_unbiased_fail}
  &= \frac{1}{\pi (1-\pi) n} \p{\sum_{k=1}^r \mu_k\p{\sum_{\cb{i : k(i) = k}} (M_i - \pi N_i)} + \sum_i (Y_i - \mu_{k(i)})(M_i - \pi N_i)},
 \end{align}
where $\mu_k = \EE[\pi]{Y_i \cond k(i) = k}$ is the expected outcome in the $k$-th community
under our sampling model.
In the above expression, the first term is problematic. Specifically
$$\sum_{k(i) = k} \p{M_i - \pi N_i} = \sum_{k(i),k(j) = k} E_{ij} (W_j - \pi) = \sum_{k(j) = k} N_j (W_j - \pi) $$
is mean zero, but has variance of scale $\alpha_k n \rhon^2$, where $\alpha_k$ denotes the fraction of units in community $k$.
In other words, in \eqref{eqn:why_unbiased_fail}, the first term is a major source of noise but contains no information about
the indirect effect. In contrast, all the useful information is contained in the second term. It has a non-zero (non-vanishing)
mean and is of constant scale. Any successful adaptation of this estimator must thus find a way to effectively cancel out
this first term while preserving the second.

Now, given this observation, we can readily mitigate the problematic first term in
the context of the stochastic block model considered in \eqref{eqn:why_unbiased_fail}; for example, we could get rid
of it centering the outcomes $Y_i$ in each community $k = 1, \, \ldots, \, r$ before running \eqref{eqn:tau_ind_unbiased_def}.
The main question is in how to adapt this insight and remedy to more general specifications beyond the stochastic block model.
To this end, recall that the stochastic block model considered above is a special case of our setting as spelled out in
Assumption \ref{assu:graphon}, where the interval $[0, \, 1]$ has been partitioned into $r$ non-overlapping sets
$\ii_k$, the $i$-th unit is in community $k$ whenever $u \in \ii_k$, and the graphon $\Gfcnc$ has a rank-$r$
representation \eqref{eqn:low_rank_graphon} with $\psi_k(u) = 1\p{\cb{u \in \ii_k}} / \sqrt{\PP{U_i \in \ii_k}}$ and
$\lambda_k = \PP{U_i \in \ii_k}$. We also note that the problematic noise term in \eqref{eqn:why_unbiased_fail} shows
up whenever $\EE{Y_i \psi_k(U_i)} \neq 0$.

Given this observation, it's natural to conjecture that
if $\Gfcnc$ is any graphon that admits a low-rank representation as in \eqref{eqn:low_rank_graphon}, then
modifying the unbiased estimator \eqref{eqn:tau_ind_unbiased_def} in a way that projects out signal components that are
correlated with the eigencomponents $\psi_k(u)$ of the graphon will result in a consistent estimator.

Our proposed PC balancing estimator is motivated by this insight. For simplicity, we start by presenting an ``oracle''
version of our estimator that assumes a-priori knowledge of the eigencomponents $\psi_k(u)$ of the graphon.
The unbiased estimator \eqref{eqn:tau_ind_unbiased_def} belongs to a class of weighted estimators $\sum_i \hgamma_i Y_i$.
We would be able to avoid any noise from signal components associated with the $\psi_k(u)$ if we could modify the weights
$\hgamma_i$ such that they balance out the $\psi_k(U_i)$ functions, i.e., if
$\sum_i \hgamma_i \psi_k(U_i) = 0$ for all $k = 1, \, \ldots, \, r$. The oracle PC balancing
estimator achieves this goal by simply projecting out the relevant parts of the weights as follows:
\begin{equation}
\label{eq:PC_oracle}
\begin{split}
&\tilde{\tau}^{\PC}_{\IND} = \frac{1}{n} \sum_i Y_i 
\p{\frac{M_i}{\pi} -  \frac{N_i - M_i}{1-\pi} +  \sum_{k=1}^r \tilde{\beta}_k \psi_k(U_i)}\!, \text{ where  $\tilde{\beta}$ solves} \\
&\sum_i \psi_l(U_i) \p{\frac{M_i}{\pi} -  \frac{N_i - M_i}{1-\pi} +  \sum_{k=1}^r \tilde{\beta}_l \psi_k(U_i)} = 0, \text{ for all } l = 1, \, 2, \, \dots, \, r.
 \end{split}
 \end{equation}
Now, in practice of course the graphon is usually unknown and we don't have access to $\psi_k(U_i)$ directly.
But if the graphon is low rank as in \eqref{eqn:low_rank_graphon}, then the edge probability matrix $\sqb{G_n(U_i, U_j)}$ will
also be a low rank matrix (when $G_n = \rhon G$), and its eigenvectors are approximately $\psi_k(U_i)$. We also note that the
adjacency matrix $E$ is a noisy observation of this low rank edge probability matrix.
Hence, we can estimate $\psi_k(U_i)$ using the eigenvectors \smash{$\hpsi_{ki}$} of $E$, and then use the
data-driven \smash{$\hpsi_{ki}$} to obtain a feasible analogue to \eqref{eq:PC_oracle}.\footnote{Throughout
this paper, we assume that the rank $r$ of the graphon is known. In practice, one could estimate $r$
by thresholding the eigenvalues of the adjacency matrix using, e.g., the approach of \citet{chatterjee2015matrix}.}
We summarize the resulting PC balancing algorithm as Procedure \ref{alg:pc_balance}. Note that unlike the estimators considered in Section \ref{section:dir}, the PC balancing algorithm requires knowledge of the graph $E$.

\begin{algbox}[t]
\myalg{alg:pc_balance}{PC balancing estimator}{
The following algorithm estimates the indirect treatment effect by modifying the weights in the unbiased estimator and balancing the estimated principal components of the graphon. The algorithm requires an input of rank $r$. 

\begin{enumerate}

\item Let $E$ be the adjacency matrix with off-diagonal terms $E_{ij}$ and 0 on the diagonal. 

\item Eigen-decompose $E$: Extract the first $r$ eigenvectors. Let $\lambda_1, \lambda_2, \dots, \lambda_r$ be the first $r$ eigenvalues of $E$ s.t. $|\lambda_1| \geq  |\lambda_2| \geq \dots \geq |\lambda_r|$. Let $\psih_k$ be the eigenvector of $k$ corresponding to the eigenvalue $\lambda_k$. 

\item Compute the PC balancing estimator
\begin{equation}
\htau_{\IND}^{\PC} = \frac{1}{n} \sum_i Y_i \p{\frac{M_i}{\pi} -  \frac{N_i - M_i}{1-\pi} +  \sum_{k=1}^r \hat{\beta}_k {\hat{\psi}_k(U_i)}}, 
 \end{equation}
where $ \hat{\beta}$ is determined by solving the following equations
\begin{equation}
\label{eq:beta_solve}
\sum_i {\hat{\psi}_l(U_i)} \p{\frac{M_i}{\pi} -  \frac{N_i - M_i}{1-\pi} +  \sum_{k=1}^r \hat{\beta}_k {\hat{\psi}_k(U_i)}} = 0, 
 \end{equation}
for all $l = 1,2, \dots, r$.

\end{enumerate}
}
\end{algbox}

Our main formal result about the indirect effect establishes consistency and asymptotic
normality of the PC balancing estimation in the ``sparse'' graph setting (i.e., with $\rho_n \rightarrow 0$).
We state our result in terms of Bernstein's condition:
Given a random variable $X$ with mean $\mu=\mathbb{E}[X]$ and variance $\sigma^{2}=\mathbb{E}\left[X^{2}\right]-\mu^{2},$ we say that Bernstein's condition with parameter $b$
	holds if
	\begin{equation}
	\label{eqn:berstein}
	\left|\mathbb{E}\left[(X-\mu)^{k}\right]\right| \leq \frac{1}{2} k ! \sigma^{2} b^{k-2} \quad \text { for } k=2,3,4, \ldots. 
	\end{equation}
\citet{wainwright2019high} shows that one sufficient condition for Bernstein's condition to hold is that $X$ be bounded. 
The proof of the following result is given in Section \ref{sec:indirect_pf} below.

\begin{theo}
	\label{theo:IND_CLT}
	Under the conditions of Theorem \ref{theo:directCLT}, assume furthermore that we have a sparse graph such that
	\begin{equation}
	\label{eqn:rate_of_rhon}
	\liminf  \frac{ \log\rhon}{\log n} > -\frac{1}{2} \ \ \text{ and } \ \ \limsup \frac{ \log\rhon}{\log n} < 0.
	\end{equation}
Finally suppose that we have a rank-$r$ graphon of the form \eqref{eqn:low_rank_graphon} such that
\begin{equation}
\label{eqn:rankr_assumption}
\abs{\lambda_1} \geq \abs{\lambda_2} \geq  \dots \geq \abs{\lambda_r} > 0, \ \EE{\psi_k(U_i)^2} = 1, \text{ and } \EE{\psi_k(U_i) \psi_l(U_j)} = 0 \text{ for } k\neq l, 
\end{equation}
and that for $U_1, U_2 \overset{\text{i.i.d}}{\sim} \operatorname{Uniform}[0, \, 1]$, 
	\begin{equation}
	\label{eqn:satisfy_bernstein}
	\psi_k(U_i) \text{ satisfies the Bernstein condition (\ref{eqn:berstein}) with parameter } b.
	\end{equation}
Then the PC balancing estimator satisfies
	\begin{equation}
	\label{eq:IND_CLT}
	\frac{\tauc_{\IND} - \tau_{\IND}}{\sqrt{\rhon}\sigma_{\IND}}  \Rightarrow \mathcal{N}(0,1),   \quad  \frac{\tauc_{\IND} - \btau_{\IND}}{\sqrt{\rhon}\sigma_{\IND}}  \Rightarrow \mathcal{N}(0,1),  
	\end{equation}
where $\sigma_{\IND}^2 = \EE{\Gfcnc(U_1,U_2)\p{\alpha_1^2+ \alpha_1\alpha_2}} + \EE{\gsmlc(U_1) \eta_1^2}/(\pi(1-\pi))$, $\alpha_i  = f_i(1,\pi) - f_i(0,\pi)$, $b_i = \pi f_i(1,\pi) + (1-\pi)f_i(0,\pi)$ and $\eta_i = b_i - \sum_{k=1}^r \EE{b_i \psi_k(U_i)}\psi_k(U_i)$. 
\end{theo}

As discussed above, we are not aware of any previous results that allow for consistent estimation of the indirect effect in
generic, moderately sparse graphs. Here, to establish \eqref{eq:IND_CLT}, we need the graph to be ``sparse'' in the sense
that the average fraction of units that are connected decays as $\rho_n \rightarrow 0$; however, we still allow the average
degree of the graph to grow very large. In contrast, existing results \citep[e.g.,][]{leung2020treatment} require the degree
distribution to remain constant, which would amount to setting $\rho_n \sim 1/n$ in our setting.
We also note that the rate of convergence derived for our estimator, namely $\sqrt{\rhon}$, is worse than the
 $1/\sqrt{n}$ we obtained for the direct effect in Theorem \ref{theo:directCLT}; however, this appears to be a
 consequence of the intrinsic difficulty of the task of estimating the indirect effect as opposed to the direct effect.

Finally, given our result for indirect effect, we can naturally get similar results for the total effects. Define
$\tauc_{\TOT} = \tauc_{\IND} + \hat{\tau}^{\HT}_{\DIR}$. Note that by Theorem \ref{theo:directCLT},
$\Var{ \hat{\tau}^{\HT}_{\DIR}} = \oo_p\p{1/n} \ll \rhon$. Hence a central limit theorem for $\tauc_{\TOT}$ can be obtained as well. 

\begin{coro}
Under the conditions of Theorem \ref{theo:IND_CLT}, the PC balancing estimator is asymptotically normally distributed around the total effect
\begin{equation}
\frac{\tauc_{\TOT} - \tau_{\TOT}}{\sqrt{\rhon}\sigma_{\IND}}  \Rightarrow \mathcal{N}(0,1),   \quad  \frac{\tauc_{\TOT} - \btau_{\TOT}}{\sqrt{\rhon}\sigma_{\IND}}  \Rightarrow \mathcal{N}(0,1),  
\end{equation}
where $\sigma_{\IND}$ is defined the same way as in Theorem \ref{theo:IND_CLT}. 
\end{coro}

\subsection{Proof of Theorem \ref{theo:IND_CLT}}
\label{sec:indirect_pf}

As a preliminary to proving our central limit theorem for \smash{$\tauc_{\IND}$}, we need to characterize the behavior
of the eigenvectors \smash{$\hpsi_{ki}$} as estimators of the graphon eigenfunctions $\psi_k(U_i)$, and to show that
if we choose \smash{$\hbeta$} to cancel out noise in the direction of \smash{$\hpsi_{ki}$} using \eqref{eq:beta_solve}, then
we also effectively balance out signal in the direction of \smash{$\psi_k(U_i)$}. Our main tool for doing so is
the following lemma. In order to facilitate the interpretation of \smash{$\hpsi_{ki}$} as an estimate of \smash{$\psi_k(U_i)$}, in
the result below (and throughout this proof), we normalize eigenvectors so that \smash{$\|\psih_k\|_2^2, \, \|\psi_k\|_2^2 = n$}. Let $\smash{\Psih} = [\psih_1, \dots, \psih_r]$ be the $n \times r$ matrix whose $k$-th column is $\psih_k$. 

\begin{lemm}
\label{lemm:a_psi_close}
Under Assumptions \ref{assu:undirected}, \ref{assu:random_graph} and \ref{assu:graphon}, suppose
furthermore that \eqref{eqn:rate_of_rhon}, \eqref{eqn:rankr_assumption} and \eqref{eqn:satisfy_bernstein} hold.
Let $\psi_k$ denote the vector of $\psi_k(U_i)$. There exists an $r \times r$ orthogonal matrix $\hat{R}$,
such that if we write $\PsiRh = \Psih \hat{R}$, and let $\psiRh_k$ be the $k$-th column of $\PsiRh$, then,
for any vector $a$ that is independent of $E$ given $U_i$'s, we have
\begin{equation}
\label{eq:hpsi_consistent}
\abs{a^T \p{\psiRh_k - \psi_k}} \,\big/\, \Norm{a}_2 = \oo_p\p{1}.
\end{equation}
\end{lemm}

Qualitatively, the above result guarantees that stochastic fluctuations in $\hpsi_k$ aren't systematically
aligned with any specific vector $a$; and so, when studying $\tauc_{\IND}$, the fact that we target $\smash{\psih_j}$ in
\eqref{eq:beta_solve} shouldn't induce too much bias. Formally, it is related to the classical result of \citet{davis1970rotation}
on the behavior of eigenvectors of a random matrix (and, in our proof, we rely on a variant of the Davis-Kahan theorem
given in \citet*{yu2015useful}). We also note that, given our normalization of $\|\psih_k\|_2$, the error bound in
\eqref{eq:hpsi_consistent} is fairly strong---and this type of result is needed in our proof. For example, recent work
by \citet*{abbe2020entrywise} provides sup-norm bounds on the fluctuations of $\psih_k$; however, these bounds
do not decay fast enough to be helpful here.

We are now ready to study $\tauc_{\IND}$ itself. To this end, we start by decomposing the estimator into parts
using the Taylor expansion as justified by \eqref{eq:deriv}:
\begin{equation}
\label{eq:ind_decomp}
\begin{split}
\tauc_{\IND} &= \frac{1}{n}\sum_i \p{\frac{M_i}{\pi} -\frac{N_i - M_i}{1 - \pi} + \sum_{k=1}^r \hbeta_k \psih_{k i}} Y_i \\
&= \frac{1}{n}\sum_i \p{\frac{M_i}{\pi} -\frac{N_i - M_i}{1 - \pi} + \sum_{k=1}^r \hbeta_k \psih_{k i}} f_i(W_i, \pi)  \\
&\ \ \ \ \ \ + \frac{1}{n}\sum_i \p{\frac{M_i}{\pi} -\frac{N_i - M_i}{1 - \pi}}  \p{\frac{M_i}{N_i} - \pi} f'_i(W_i, \pi) \\
&\ \ \ \ \ \ + \frac{1}{n}\sum_i \p{\sum_{k=1}^r \hbeta_k \psih_{k i}}  \p{\frac{M_i}{N_i} - \pi} f'_i(W_i, \pi) \\
&\ \ \ \ \ \ + \frac{1}{n}\sum_i \p{\frac{M_i}{\pi} -\frac{N_i - M_i}{1 - \pi} + \sum_{k=1}^r \hbeta_k \psih_{k i}}  \p{\frac{M_i}{N_i} - \pi}^2 f''_i(W_i, \pi_i^{\star}),
\end{split}
\end{equation}
where $\pi_i^{\star}$ is some value between $\pi$ and $M_i/N_i$. 
This decomposition already provides some insight into the behavior of $\tauc_{\IND}$. Here, the second summand is the
one that contains all the signal, while the third and fourth end up being negligible. In particular, we note that the error
terms in all three bounds below are smaller than the leading-order $\sqrt{\rho_n}$ error in \eqref{eq:IND_CLT}.

\begin{prop}
\label{prop:term_2}
Under the conditions of Theorem \ref{theo:IND_CLT},
\begin{equation}
\frac{1}{n}\sum_i \p{\frac{M_i}{\pi} -\frac{N_i - M_i}{1 - \pi}}  \p{\frac{M_i}{N_i} - \pi} f'_i(W_i, \pi) = \tau_{\IND} + \oo_p\p{B\rhon}.
\end{equation}
\end{prop}

\begin{prop}
\label{prop:term_3}
Under the conditions of Theorem \ref{theo:IND_CLT},
\begin{equation}
\frac{1}{n}\sum_i \p{\sum_{k=1}^r \hbeta_k \psih_{k i}}  \p{\frac{M_i}{N_i} - \pi} f'_i(W_i, \pi) = \oo_p\p{B\rhon}.
\end{equation}
\end{prop}

\begin{prop}
\label{prop:term_4}
Under the conditions of Theorem \ref{theo:IND_CLT},
\begin{equation}
 \frac{1}{n}\sum_i \p{\frac{M_i}{\pi} -\frac{N_i - M_i}{1 - \pi} + \sum_{k=1}^r \hbeta_k \psih_{k i}}\p{\frac{M_i}{N_i} - \pi}^2 f''_i(W_i, \pi_i^{\star}) = \oo_p\p{\frac{B}{\sqrt{n \rhon}}}.
\end{equation}
\end{prop}

It now remains to study the first term in \eqref{eq:ind_decomp}. It is perhaps surprising at first glance that this term matters much,
since it has nothing to do with cross-unit interference. However, this term ends up being the dominant source of noise; and, in fact,
is also what causes the variance of the unbiased estimator \smash{$\htau^{\U}_{\IND}$} to explode as seen
in Proposition \ref{prop:unbiased_rate}.

To this end, we introduce some helpful notation. Let $b_i = \pi f_i(1, \, \pi) + (1 - \pi) f_i(0, \, \pi)$,
and let $\mu_k$ be the projection of $b_i$ onto $\psi_k(U_i)$, i.e., $\mu_k = \EE{b_i \psi_k(U_i)}$ (recall that
$\EE{\psi_k^2(U_i)} = 1$). Then, we can express $f_i(W_i , \pi)$ as
\begin{equation}
\label{eq:baseline_split}
\begin{split}
f_i(W_i, \pi)
 &= (W_i - \pi)\sqb{f_i(1 , \pi) - f_i(0 , \pi)} + \p{\pi f_i(1 , \pi) + (1-\pi)f_i(0, \pi)}\\
 &= (W_i - \pi)\sqb{f_i(1 , \pi) - f_i(0 , \pi)} + \sum_{k = 1}^r \mu_k \psi_k(U_i) + \eta_i,
 \end{split}
\end{equation}
where $\eta_i$ is the residual term implied by the above notation. The key property of this decomposition is that,
because $\mu_k$ capture the projection of $b_i$ onto the $\psi_k(U_i)$, then $\EE{\eta_i \psi_k(U_i)} = 0$ for
all $k = 1, \, \ldots, \, r$.

Following the discussion around \eqref{eqn:why_unbiased_fail} if we did not use the PC balancing adjustment,
the problematic term in \eqref{eq:baseline_split} would be the
second one, i.e., the one that's aligned with the $\psi_k(U_i)$. But the PC balancing adjustment helps mitigate the
behavior of this term. Specifically, thanks to \eqref{eq:beta_solve}, we see that in the context of the first summand
of \eqref{eq:ind_decomp},
\begin{equation}
\label{eq:align_remainder}
\begin{split}
&\frac{1}{n}\sum_i \p{\frac{M_i}{\pi} -\frac{N_i - M_i}{1 - \pi} + \sum_{k=1}^r \hbeta_k \psih_{k i}}  \sum_{k = 1}^r \mu_k \psi_k(U_i) \\
&\ \ \ \ \ \ =\frac{1}{n} \sum_{k = 1}^r \mu_k  \sum_i \p{\frac{M_i}{\pi} -\frac{N_i - M_i}{1 - \pi} + \sum_{k=1}^r \hbeta_k \psih_{k i}} \p{\psi_k(U_i) - \psiRh_{ki}},
\end{split}
\end{equation}
i.e., this term gets canceled out to the extent that $\psiRh_{ki}$ acts as a good estimate of $\psi_k(U_i)$. The following
result, which makes heavy use of Lemma \ref{lemm:a_psi_close} given above, validates this intuition.

\begin{prop}
\label{prop:term_5}
Under the conditions of Theorem \ref{theo:IND_CLT}, \eqref{eq:align_remainder} is bounded as $\oo_p\p{B\rhon}$.
\end{prop}

We are now essentially ready to conclude. By combining Propositions \ref{prop:term_2}--\ref{prop:term_5} above and
plugging \eqref{eq:align_remainder} into \eqref{eq:ind_decomp}, we can verify the following using basic concentration
arguments. In doing so, we heavily rely on the fact that $\EE{\eta_i \psi_k(U_i)} = 0$, which implies that terms of
the type $\sum_{i} \eta_i \psi_k(U_i)$ are small.

\begin{prop}
\label{prop:summary}
Under the conditions of Theorem \ref{theo:IND_CLT},
\begin{equation}
\label{eq:ind_summary}
\begin{split}
&\tauc_{\IND} - \tau_{\IND} = \frac{1}{n \pi (1 - \pi)} \sum_{(i,j), i \neq j} (W_i - \pi) E_{ij} \xi_j + o_p\p{\sqrt{\rho_n}} \\
&\xi_j = (W_j - \pi) \p{f_i(1, \, \pi) - f_i(0, \, \pi)} + \eta_j.
\end{split}
\end{equation}
\end{prop}

It now remains to prove a central limit theorem for the asymmetric bilinear statistic appearing in the right-hand
side of \eqref{eq:ind_summary}. To do so, we rely on a central limit theorem for the average of locally dependent
random variables derived in \citet{ross2011fundamentals} via Stein's method for Gaussian approximation.
The following result leads to our desired conclusion regarding convergence around \smash{$\tau_{\IND}$}.

\begin{prop}
\label{prop:CLT_leading}
Under the conditions of Theorem \ref{theo:IND_CLT} and using notation from \eqref{eq:align_remainder}, 
$\epsilon_n = \p{n \pi (1 - \pi)}^{-1} \sum_{i \neq j} (W_i - \pi) E_{ij} \xi_j$ has
a Gaussian limiting distribution:
\begin{equation}
\begin{split}
&\frac{\epsilon_n}{\sqrt{\rhon}} \Rightarrow \mathcal{N}(0,\sigma_{\IND}^2), \quad\quad \alpha_i  = f_i(1,\pi) - f_i(0,\pi), \\
&\sigma_{\IND}^2 = \EE{\Gfcnc(U_1,U_2)\p{\alpha_1^2+ \alpha_1\alpha_2}} + \EE{\gsmlc(U_1) \eta_1^2}/(\pi(1-\pi)).
\end{split}
\end{equation}
\end{prop}

Finally, regarding  \smash{$\btau_{\IND}$}, we note that \smash{$\btau_{\IND} - \tau_{\IND}$} is an average of i.i.d. random
variables bounded by $CB$. Hence \smash{$\btau_{\IND} = \tau_{\IND} + \oo_p({B}/{\sqrt{n}})$}, and so the same central limit theorem
holds if we center our estimator and  \smash{$\btau_{\IND}$} instead.

\subsection{Numerical Evaluation}
\label{subsection:PC_numerical}

We end this section by empirically evaluating the above findings. First, we evaluate the scaling of the
mean-squared error (MSE) of different estimators of the indirect effect. In Figures \ref{fig:mse_vs_n_tilde1}
and \ref{fig:mse_vs_n_tilde2}, we plot the log-MSE of our PC balancing estimator $\tauc_{\IND}$ against the log sample size
$\log(n)$ in a variety of settings described in 
\ifaos
\else
Appendix \ref{sec:simu_spec}.
\fi
In Figure \ref{fig:mse_vs_n_tilde1},
we consider specifications with sparsity level \smash{$\rhon = n^{-{1}/{5}}$}, while in Figure \ref{fig:mse_vs_n_tilde2},
we consider \smash{$\rhon = n^{-{2}/{5}}$}. Theorem \ref{theo:IND_CLT} predicts that the MSE of $\tauc_{\IND}$
should scale as $\rhon$, and here, in line with this prediction, we see that the curves in Figures \ref{fig:mse_vs_n_tilde1}
and \ref{fig:mse_vs_n_tilde2} are roughly linear with slopes $-1/5$ and $-2/5$ respectively.
Next, in Figures \ref{fig:mse_vs_n_hat1} and \ref{fig:mse_vs_n_hat2}, we perform the same exercise with the unbiased
estimator $\htau^{\U}_{\IND}$. By Proposition \ref{prop:unbiased_rate}, we know that the MSE of this estimator scales
as $n\rho^2$, and so we expect to see linear relationships with a slope of $3/5$ when \smash{$\rhon = n^{-{1}/{5}}$} and $1/5$
when \smash{$\rhon =n^{-{2}/{5}}$}. The slope of the realized MSE is again aligned with the prediction from theory. 
Finally, we evaluate the predicted distribution for our PC balancing estimator $\tauc_{\IND}$ on a larger simulation
setting: We consider a rank-3 stochastic block model and a graph with $n = 1,000,000$ nodes. Figure \ref{fig:PC_hist}
shows the distribution of $\tauc_{\IND}$ across $N = 1000$ simulations. We see that the distribution of the estimator closely
matches the limiting Gaussian distribution predicted by Theorem \ref{theo:IND_CLT}.

\begin{figure}
	\centering
	\begin{subfigure}{.45\textwidth}
		\includegraphics[width = \textwidth]{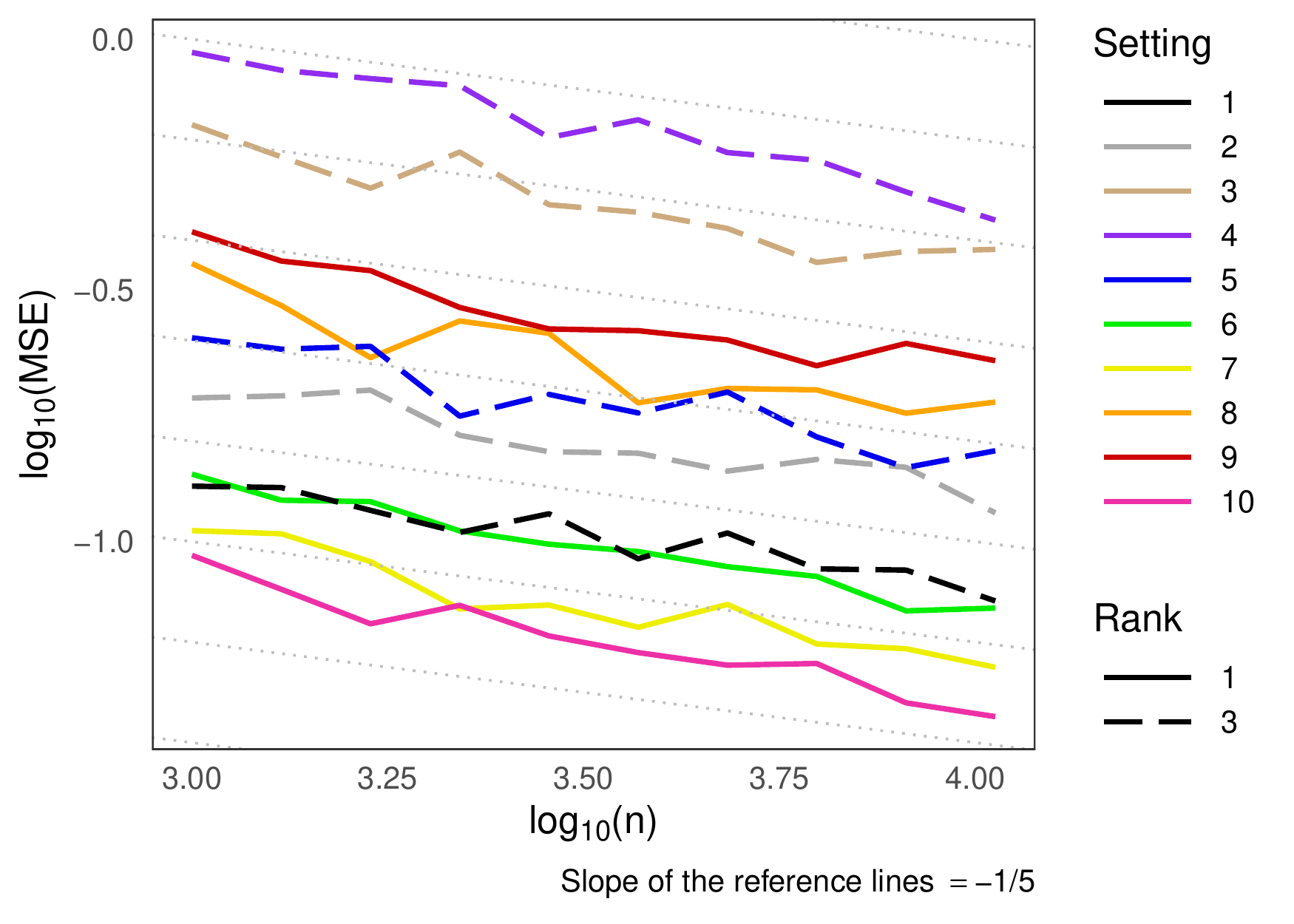}
		\caption{Sparsity level $\rhon = n^{-\frac{1}{5}}$}
		\label{fig:mse_vs_n_tilde1}
	\end{subfigure}%
	\begin{subfigure}{.45\textwidth}
		\centering
		\includegraphics[width = \textwidth]{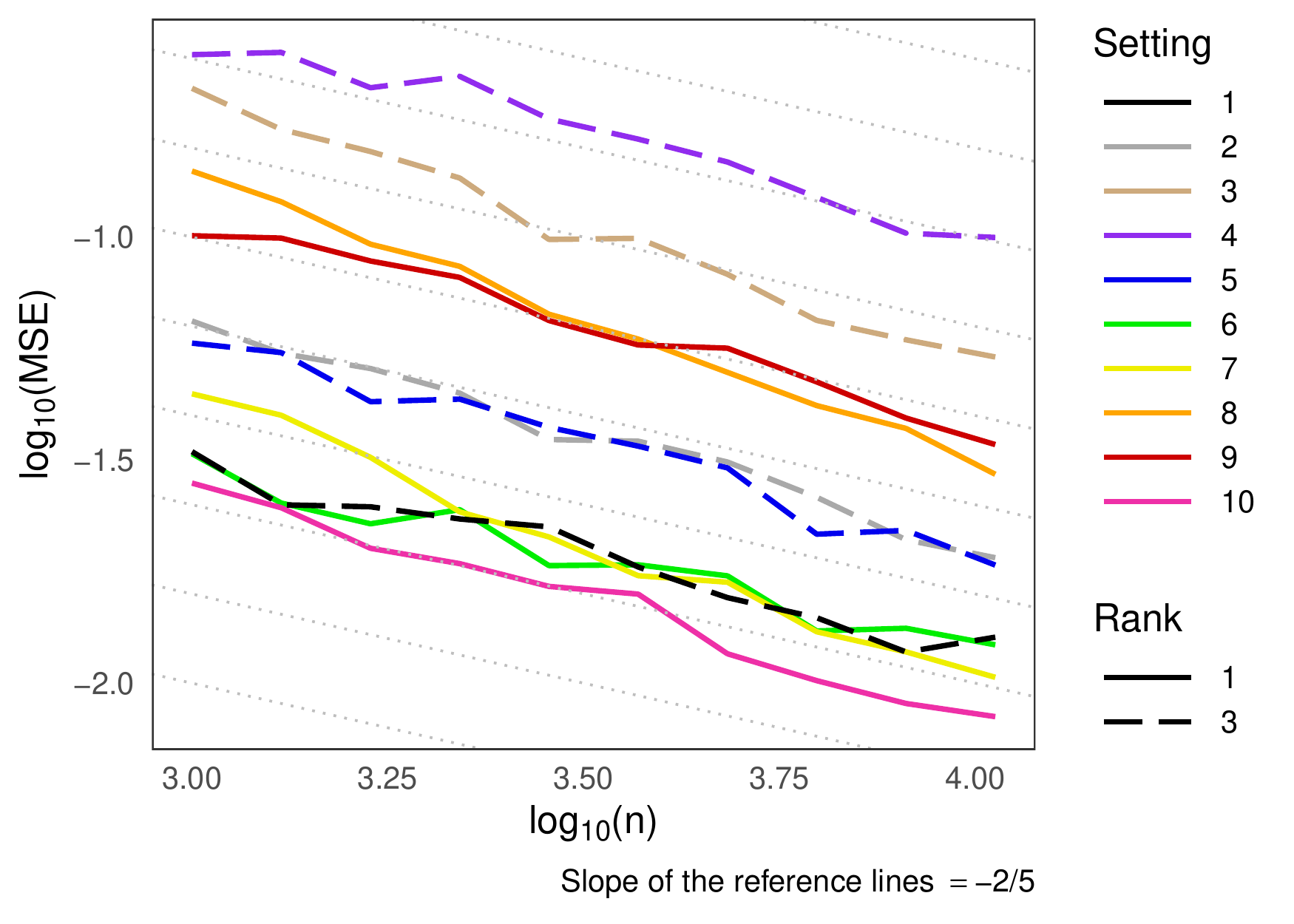}
		\caption{Sparsity level $\rhon = n^{-\frac{2}{5}}$}
		\label{fig:mse_vs_n_tilde2}
	\end{subfigure}
	\caption{MSE of the PC balancing estimator $\tauc_{\IND}$}
	\label{fig:PC_MSE}

	\begin{subfigure}{.45\textwidth}
		\includegraphics[width = \textwidth]{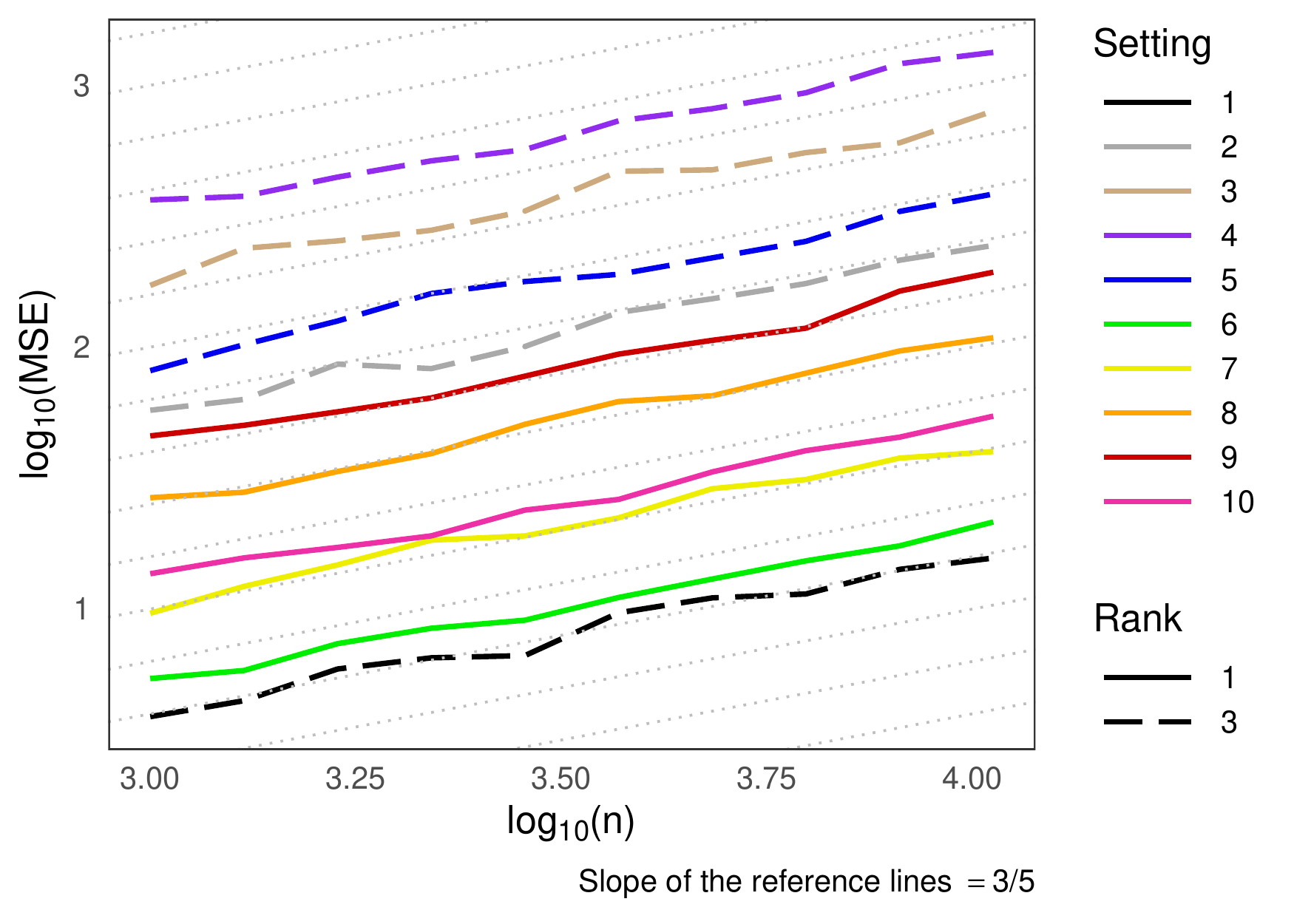}
		\caption{Sparsity level $\rhon = n^{-\frac{1}{5}}$}
		\label{fig:mse_vs_n_hat1}
	\end{subfigure}%
	\begin{subfigure}{.45\textwidth}
		\includegraphics[width = \textwidth]{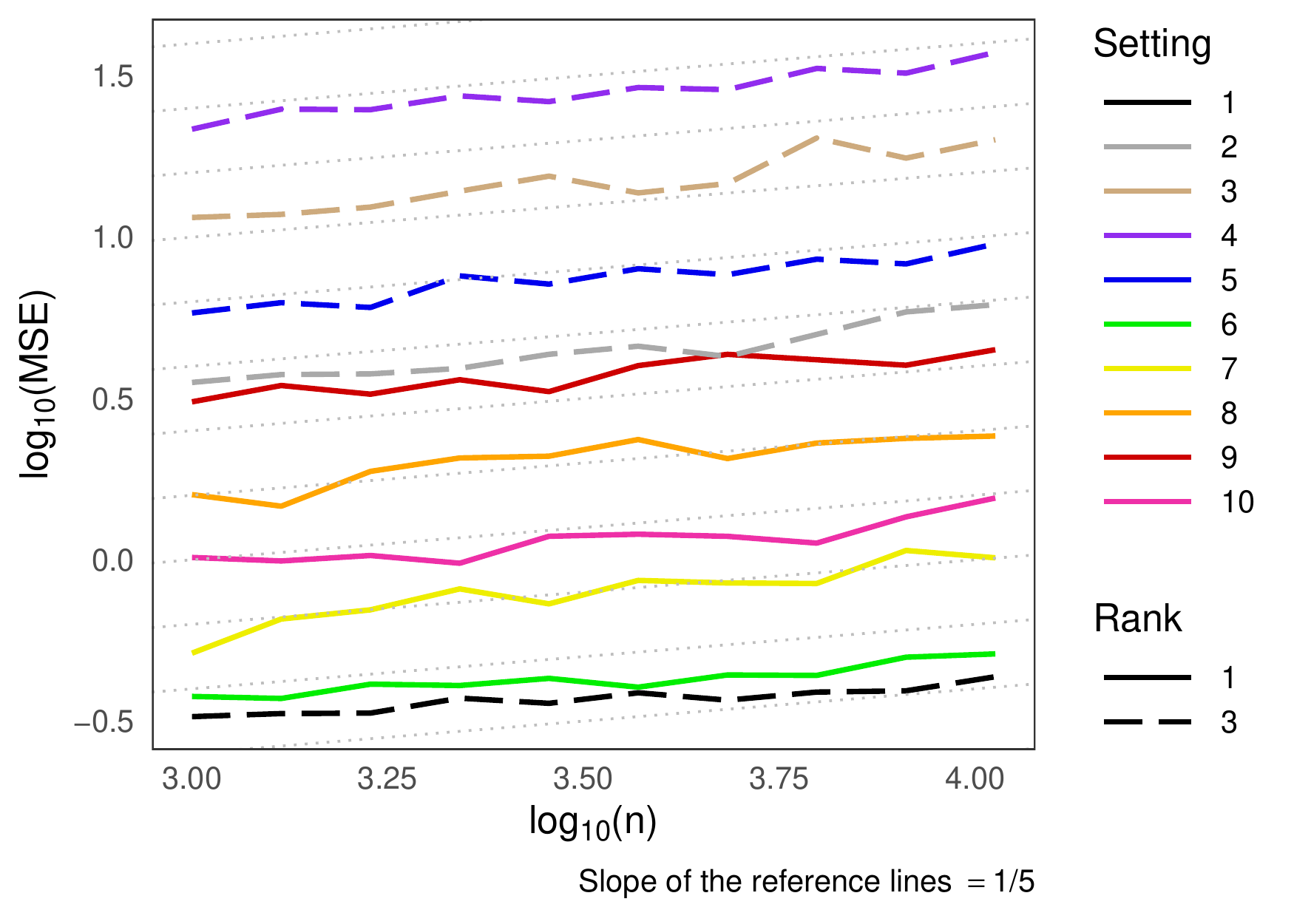}
		\caption{Sparsity level $\rhon = n^{-\frac{2}{5}}$}
		\label{fig:mse_vs_n_hat2}
	\end{subfigure}
	\caption{MSE of the unbiased estimator $\htau_{\IND}$}
	\label{fig:Unbiased_MSE}
\end{figure}

\begin{figure}[t]
	\centering
	\includegraphics[trim=0 20 160 30, clip, width = 0.6\textwidth]{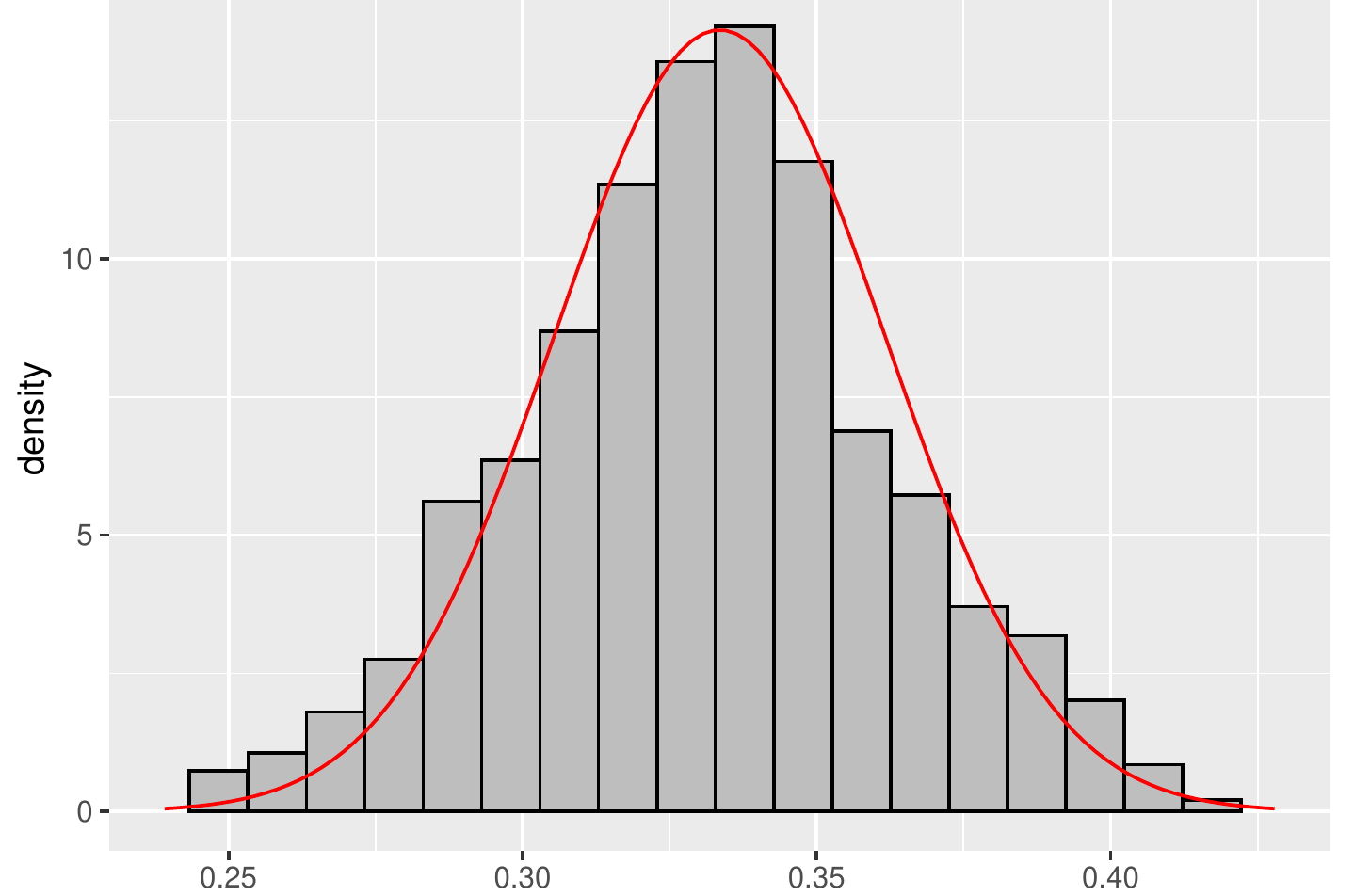}
	\caption{Comparison of a histogram of $\tauc_{\IND}$ across $N = 1000$ simulations, and
	the Gaussian limit predicted by Theorem \ref{theo:IND_CLT}.}
	\label{fig:PC_hist}
\end{figure}

\section{Discussion}

The network interference model is a popular framework for studying treatment effect estimation under cross-unit
interference. In this paper, we studied estimation in the network interference model under random graph assumptions
and showed that---when paired with conditions such as
anonymous interference---these assumptions could be leveraged to provide strong performance guarantees. We considered estimation
of both the direct and indirect effects and, for the former, found that existing estimators can be much more accurate
than previously known while, for the latter, we proposed a new estimator that is consistent in moderately dense settings.
Both sets of results highlight the promise of random graph asymptotics in yielding insights about the nature of treatment
effect estimation under network interference and in providing guidance for new methodological developments.

The finding from Theorem \ref{theo:directCLT} that natural estimators of the direct effect satisfy a $1/\sqrt{n}$-rate
central limit theorem even in dense graphs may prove to be of particular practical interest. This is because, as emphasized in
\citet*{savje2017average}, the considered estimators of the direct effect are algorithmically the same as standard
estimators of the average treatment effect in a randomized study without interference, and so our result for the direct effect
can be used to assess the sensitivity of randomized study inference to the presence of unknown network interference.
Our $1/\sqrt{n}$-rate guarantees are much stronger than the generic bounds given in \citet{savje2017average}, and thus paint a
more optimistic picture of how badly unknown interference may corrupt randomized study inference.

One question left open by this paper is whether the proposed estimators are in any sense optimal. In the case of the
direct effect, the $1/\sqrt{n}$ rate of convergence is clearly optimal; however, it would be interesting to investigate
whether any tractable results on efficiency are available in our setting. Meanwhile, in the case of the indirect effect,
the optimal rate of convergence itself remains open. Our proposed PC balancing estimator achieves a $\sqrt{\rhon}$-rate
of convergence, which intuitively appears to be a reasonably strong rate for this task. For purpose of benchmarking, consider estimation
of treatment effects in a stochastic block model with $K_n = \rhon^{-1}$ non-interacting blocks. Then, any simple
block-level randomized algorithm could at best hope for a $1/\sqrt{K_n} = \sqrt{\rhon}$ rate of convergence, whereas our
PC balancing estimator can achieve this rate using unit-level randomization alone. Developing formal lower bounds for this
problem, however, would of course be of considerable interest.

Another interesting direction for future work is in understand the generality of our results, i.e., under what conditions we can plausibly expect
estimators of the direct effect under network interference to achieve a $1/\sqrt{n}$-rate of converge. Here, we started
with a specific generative model, including anonymous interactions and a graphon model for the exposure graph; however,
it's plausible to us that a similar result would hold under more generality. \citet{lovasz2006limits} show that a graphon limit
arises naturally by considering any sequence of dense graphs \smash{$E^n$} with the property that, for any fixed graph $F$, the
density of copies of $F$ in \smash{$E^n$} tends to a limit. Is it similarly possible to devise regularity assumptions on a sequence
of exposure graphs \smash{$E^n$} and potential outcome functions \smash{$\cb{f_i(\cdot)}_{i=1}^n$} under which
the behavior of estimators for the direct effect is accurately predicted by graphon modeling?

\section*{Acknowledgment}

We are grateful for helpful discussions with Guillaume Basse, Emmanuel Cand\`es, Peng Ding, Trevor Hastie, Avi Feller, Betsy Ogburn,
Fredrick S\"avje and seminar participants at a number of venues. This work was partially supported by
NSF grant DMS--1916163.

\ifaos
\else
\bibliographystyle{plainnat}
\bibliography{references}
\fi

\ifaos
\else

\newpage

\begin{appendix}

\section{Simulation Details}
\label{sec:simu_spec}
Code to reproduce the experiments is available from
\url{https://github.com/lsn235711/random-graph-interference}.

\subsection{Mean square errors of $\tauc_{\IND}$ and $\htau^{\U}_{\IND}$}
In Section \ref{subsection:PC_numerical}, we compare 10 different setups (different graphons and potential outcome models). For all 10 settings, we consider the number of units $n$ vary from $1000$ to $10000$. We consider two different levels of sparsity, $\rhon = n^{-\frac{1}{5}}$ and $\rhon = n^{-\frac{2}{5}}$. We compute 500 replicates of  $\tauc_{\IND}$ and $\htau^{\U}_{\IND}$ in each of the above setting, use Proposition \ref{prop:estimands} to compute $\tau_{\IND}$, and find the mean square errors of $\tauc_{\IND}$ and $\htau_{\IND}$. 

\paragraph{Settings we considered}
\quad \\

\quad	  \textbf{Rank-3 graphons}
\begin{enumerate}

	\item  A stochastic block model: Rank-3 graphon $\Gfcnc (U_i,U_j) = \frac{3}{5} \Big(\mathbf{1}\cb{U_i \in \sqb{0,\frac{1}{3}}, U_j \in \sqb{0,\frac{1}{3}}} + \mathbf{1}\cb{U_i \in \sqb{\frac{1}{3},\frac{2}{3}}, U_j \in \sqb{\frac{1}{3},\frac{2}{3}}} + \mathbf{1}\cb{U_i \in \sqb{\frac{2}{3},1}, U_j \in \sqb{\frac{2}{3},1}} \Big) + \frac{1}{5}$. Potential outcome
	$Y_i = \frac{1}{2}\p{W_i + U_i\frac{M_i}{N_i}}^2 + \frac{1}{5}\epsilon_i$, where $\epsilon_i \sim \mathcal{N}(0,1)$. 

	\item Rank-3 graphon $\Gfcnc(U_i,U_j) = \frac{27}{4}(U_i U_j - 2U_i^2 U_j^2 + U_i^3 U_j^3)$. Potential outcome $Y_i = \cos\p{3W_i\frac{M_i}{N_i}} + \frac{1}{5}\epsilon_i$, where $\epsilon_i \sim \mathcal{N}(0,1)$.

	\item Rank-3 graphon $\Gfcnc(U_i,U_j) = \frac{27}{4}(U_i U_j - 2U_i^2 U_j^2 + U_i^3 U_j^3)$. Potential outcome $Y_i = -e^{U_i} \cos\p{3 W_i \frac{M_i}{N_i}} + \frac{1}{5}\epsilon_i$, where $\epsilon_i \sim \mathcal{N}(0,1)$.

	\item Rank-3 graphon $\Gfcnc(U_i,U_j) = \frac{1}{4} + \frac{\lfloor3 \min(U_i, U_j)\rfloor}{4}$. Potential outcome $Y_i = (1 + W_i) e^{\frac{M_i}{N_i}} + \frac{1}{5}\epsilon_i$, where $\epsilon_i \sim \mathcal{N}(0,1)$.

	\item Rank-3 graphon $\Gfcnc(U_i,U_j) = \frac{1}{4} + \frac{\lfloor3 \min(U_i, U_j)\rfloor}{4}$. Potential outcome $Y_i = \frac{1}{5}(1+U_i)^2(1 + W_i) e^{\frac{M_i}{N_i}} + \frac{1}{5}\epsilon_i$, where $\epsilon_i \sim \mathcal{N}(0,1)$.

	\textbf{Rank-1 graphons}

	\item  Rank-1 graphon $\Gfcnc(U_i,U_j) = \p{\frac{3}{10} + \frac{3}{5}\mathbf{1}\cb{U_i>\frac{1}{2}}}\p{\frac{3}{10} + \frac{3}{5}\mathbf{1}\cb{U_j>{\frac{1}{2}}}}$. Potential outcome
	$Y_i = \frac{1}{2}\p{W_i + U_i\frac{M_i}{N_i}}^2 + \frac{1}{5}\epsilon_i$, where $\epsilon_i \sim \mathcal{N}(0,1)$. 

	\item Rank-1 graphon $\Gfcnc(U_i,U_j) = \p{\frac{3}{10}\sin(2\pi U_i) + \frac{1}{2}}\p{\frac{3}{10}\sin(2\pi U_j) + \frac{1}{2}}$. Potential outcome $Y_i = \cos\p{3W_i\frac{M_i}{N_i}} + \frac{1}{5}\epsilon_i$, where $\epsilon_i \sim \mathcal{N}(0,1)$.

	\item Rank-1 graphon $\Gfcnc(U_i,U_j) = \p{\frac{3}{10}\sin(2\pi U_i) + \frac{1}{2}}\p{\frac{3}{10}\sin(2\pi U_j) + \frac{1}{2}}$. Potential outcome $Y_i = -e^{U_i} \cos\p{3 W_i \frac{M_i}{N_i}} + \frac{1}{5}\epsilon_i$, where $\epsilon_i \sim \mathcal{N}(0,1)$.

	\item Rank-1 graphon $\Gfcnc(U_i,U_j) = \p{\frac{1}{20}(U_i+1)^4 + \frac{1}{10}}\p{\frac{1}{20}(U_j+1)^4 + \frac{1}{10}}$. Potential outcome $Y_i = (1 + W_i) e^{\frac{M_i}{N_i}} + \frac{1}{5}\epsilon_i$, where $\epsilon_i \sim \mathcal{N}(0,1)$.

	\item Rank-1 graphon $\Gfcnc(U_i,U_j) = \p{\frac{1}{20}(U_i+1)^4 + \frac{1}{10}}\p{\frac{1}{20}(U_j+1)^4 + \frac{1}{10}}$. Potential outcome $Y_i = \frac{1}{5}(1+U_i)^2(1 + W_i) e^{\frac{M_i}{N_i}} + \frac{1}{5}\epsilon_i$, where $\epsilon_i \sim \mathcal{N}(0,1)$.
\end{enumerate}

We plot the rank-3 graphons considered above in Figure \ref{fig:graphon01}, \ref{fig:graphon02} and \ref{fig:graphon03}. 	

\begin{figure}[t]
	\includegraphics[width = 0.8\textwidth]{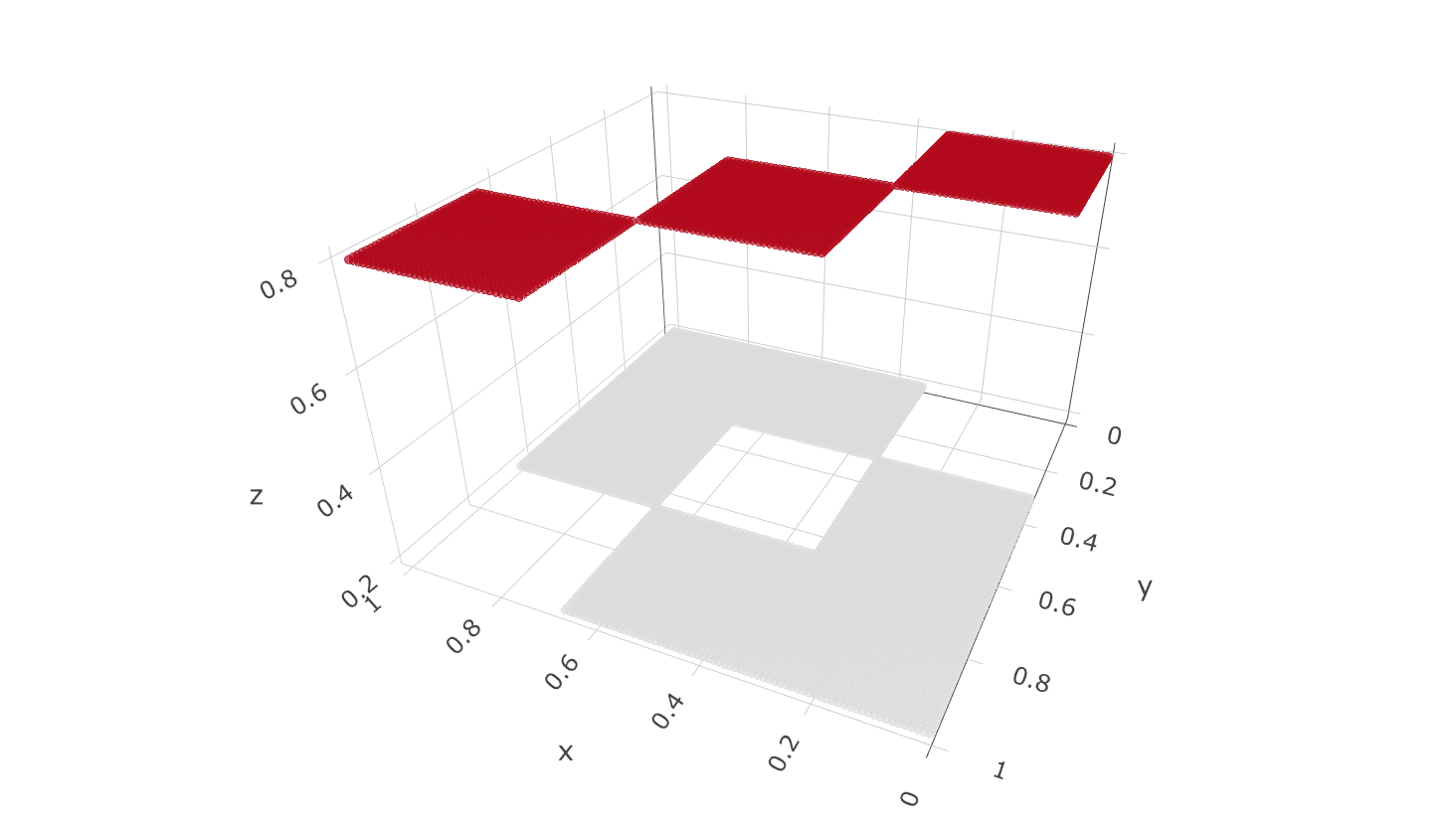}
	\caption{Rank-3 Graphon: $\Gfcnc(U_i,U_j) = \frac{3}{5} \Big(\mathbf{1}\cb{U_i \in \sqb{0,\frac{1}{3}}, U_j \in \sqb{0,\frac{1}{3}}} + \mathbf{1}\cb{U_i \in \sqb{\frac{1}{3},\frac{2}{3}}, U_j \in \sqb{\frac{1}{3},\frac{2}{3}}} + \mathbf{1}\cb{U_i \in \sqb{\frac{2}{3},1}, U_j \in \sqb{\frac{2}{3},1}} \Big) + \frac{1}{5}$.}
	\label{fig:graphon01}
\end{figure}

\begin{figure}[t]
	\includegraphics[width = 0.8\textwidth]{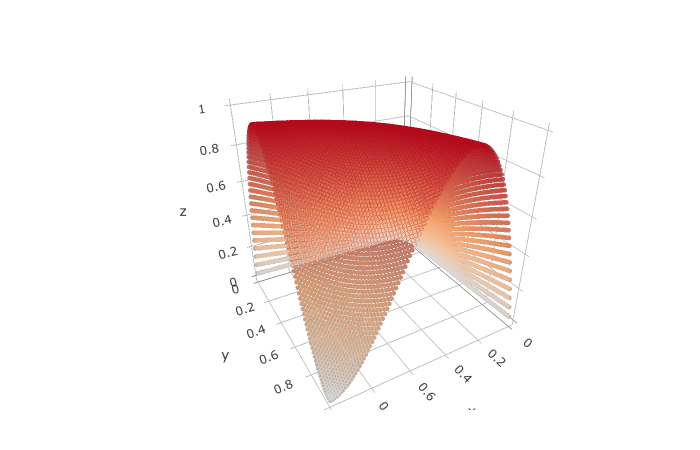}
	\caption{Rank-3 Graphon: $\Gfcnc(U_i,U_j) = \frac{27}{4}(U_i U_j - 2U_i^2 U_j^2 + U_i^3 U_j^3)$.}
	\label{fig:graphon02}
\end{figure}

\begin{figure}[t]
	\includegraphics[width = 0.8\textwidth]{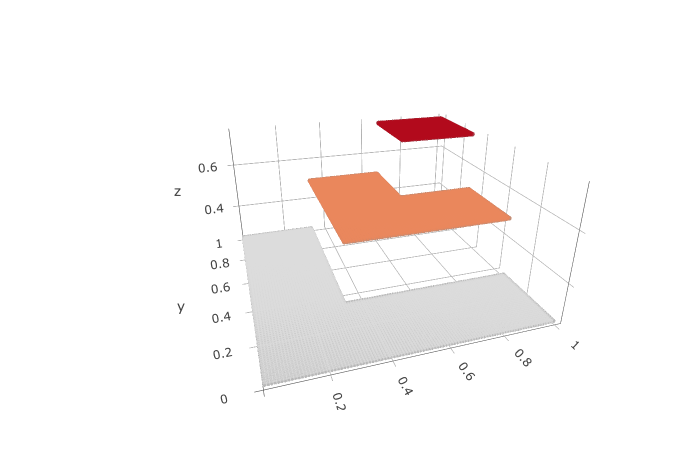}
	\caption{Rank-3 Graphon: $\Gfcnc(U_i,U_j) = \frac{1}{4} + \frac{\lfloor3 \min(U_i, U_j)\rfloor}{4}$.}
	\label{fig:graphon03}
\end{figure}

\subsection{Distribution of $\tauc_{\IND}$}
For the larger simulation, we consider setting 1 as in the above subsection. Specifically, we consider a stochastic block model: Rank-3 graphon $\Gfcnc (U_i,U_j) = \frac{3}{5} \Big(\mathbf{1}\cb{U_i \in \sqb{0,\frac{1}{3}}, U_j \in \sqb{0,\frac{1}{3}}} + \mathbf{1}\big\{U_i \in \sqb{\frac{1}{3},\frac{2}{3}}, U_j \in \sqb{\frac{1}{3},\frac{2}{3}}\big\} + \mathbf{1}\cb{U_i \in \sqb{\frac{2}{3},1}, U_j \in \sqb{\frac{2}{3},1}} \Big) + \frac{1}{5}$. Potential outcome
$Y_i = \frac{1}{2}\Big(W_i + U_i\frac{M_i}{N_i}\Big)^2 + \frac{1}{5}\epsilon_i$, where $\epsilon_i \sim \mathcal{N}(0,1)$. We take the number of units $n$ to be 1,000,000. We compute $1000$ PC balancing estimators $\tauc_{\IND}$ and plot their histogram. We also plot the predicted normal distribution from theory in red.

\newpage
\section{Proofs}
\label{sec:proofs}

\subsection{Some more notations}
We'll introduce some notations here.

For $i,j,k$ all different, let
\begin{equation}
\begin{split}
&\Hfcnc(U_i, U_j) = \EE{ \Gfcnc(U_i, U_k)\Gfcnc(U_j, U_k) | U_i,U_j }  \\
&\hsmlc(U_i) = \EE{ \Gfcnc(U_i, U_k)\Gfcnc(U_j, U_k) | U_i } = \EE{ \Gfcnc(U_i, U_j) \gsmlc(U_j) | U_i },\\
&\hbsmc = \EE{\hsmlc(U_i)} = \EE{ \Gfcnc(U_i, U_j)\Gfcnc(U_j, U_k)}  = \EE{ \gsmlc(U_i)^2}.
\end{split}
\end{equation}
Here, $\Hfcnc(U_i, U_j)$ captures to the expectation of number of common neighbors of unit $i$ and $j$, while $\hsmlc$ and $\hbsmc$ are again marginalized versions of it. Write related quantities of the graphon $G_n$ with an $n$-superscript.

Let $E$ be the adjacency (edge) matrix with entries $E_{ij}$ and 0 on diagonal. Let $\Es = \EE{E|U}$ is the probability matrix with entries $\Gfcn(U_i, U_j)$ and 0 on diagonal. Let $\Gm = \sum_{k = 1}^r \lambda_k \psi_k \psi_k^T$, where $\psi$ is a vector with $\psi_i = \psi(U_i)$ as defined in (\ref{eqn:low_rank_graphon}). Hence $\Gm$ has entries of the form $\Gm_{ij} = \Gfcnc(U_i, U_j)$. Note that we make the diagonal of $\Gm$ non-zero, but putting $\Gm_{ii} = \sum_{k = 1}^r \lambda_k \psi(U_i)^2$. Thus $\Gm$ is low rank with rank $r$. 

Recall that as in Procedure \ref{alg:pc_balance}, $\psih_{k}$ is the $k$-th (scaled) eigenvector of $E$. Specifically we scale it so that $\Norm{\smash{\psih_{k}}} = \sqrt{n}$. We scale it this way so that it's easier to compare $\smash{\psih_{ki}}$ with $\psi_k(U_i)$. Let $\smash{\psis_{k}}$ and $\smash{\psit_{k}}$ be the $k$-th (scaled) eigenvector of $\Es$ and $\Gm$ respectively, again with  $\Norm{\psis_{k}} = \sqrt{n}$ and $\Norm{\smash{\psit_{k}}} = \sqrt{n}$. Let $\lambdah_k$, $\lambdas_k$, $\lambdat_k$ be the $k$-th eigenvalue of $E$, $\Es$ and $\Gm$, i.e. $E \smash{\psih_k} = \lambdah_k \smash{\psih_k}$, $\Es \psis_k = \lambdas_k \psis_k$, and $\Gm \smash{\psit_k} = \lambdat_k \smash{\psit_k}$. Let $\Psi = \sqb{\psi_1, \dots, \psi_r}$ be the $n \times r$ matrix whose $k$-th column is $\psi_k$. Define $\Psit$, $\Psis$ and $\Psih$ in the same fashion.  

We summarize the above notations in Table \ref{tab:notation_clt}. 

\begin{table}
\centering
	\begin{tabular}{|c || c | c | c |} 
		\hline
		Matrix & $E$ & $\Es$ & $G$  \\ 
		\hline
		Eigenvector & $\psih_k$ & $\psis_k$ & $\psit_k$  \\ 
		\hline
		Eigenvalue & $\lambdah_k$ & $\lambdas_k$ & $\lambdat_k$\\
		\hline
		Entries & $E_{ij}$ & $\Gfcn(U_i, U_j)$  & $\Gfcnc(U_i, U_j)$ \\
		\hline
	\end{tabular}
\caption{Summary of notation}
\label{tab:notation_clt}
\end{table}

Let 
$\kappa_1 = -\limsup \frac{ \log\rhon}{\log n} > 0$ and $\kappa_2 = -\liminf  \frac{ \log\rhon}{\log n} < \frac{1}{2}$.

\subsection{Some lemmas}

\begin{lemm}
	\label{lemma:boundMN}
	Consider a randomized trial under network interference satisfying Assumptions
\ref{assu:undirected}, \ref{assu:random_graph} and \ref{assu:graphon}, with treatment assigned independently as $W_i \sim \text{Bernoulli}(\pi)$ for some $0 < \pi < 1$.
	
	\begin{enumerate}
		\item 
		Suppose furthermore that if we define $g_1(u) = \int_0^1\min(1,\Gfcnc(u,t)) dt$, then the function $g_1$ is bounded away from 0, i.e. 
	\begin{equation}
	\label{eqn:G_lower_boundk}
	g_1(u_1) \geq \clower \text{ for any } u_1.
	\end{equation}

 Then for any $k\in \mathbb{N}, k \geq 1$, there exist some constant $C_k$ depending on $k$, s.t.
	\begin{enumerate}
		\item 	\[\EE{\p{\frac{1}{N_i}}^{k}  1_{\cb{N_i >0}} \Bigg  | U_i } \leq \frac{C_k}{(n \rhon \clower)^k},\]
		\item 	\[\EE{\p{\frac{1}{N_i}}^{k} 1_{\cb{N_i >0}} } \leq \frac{C_k}{(n \rhon \clower)^k},\]
		\item 	\[\EE{\p{\frac{M_i - \pi N_i}{N_i}}^{2k}} \leq \frac{C_k}{(n \rhon \clower)^k},\]
		\item   \[ \PP{N_i < \frac{\clower}{2}\rhon n} \leq e^{-Cn \rhon}.\]
		\item   \[ \PP{ \min_i \cb{N_i} < \frac{\clower}{2}\rhon n} \leq n e^{-Cn \rhon }.\]
	\end{enumerate}

	\item Assume the graphon has a finite $K^{\text{th}}$ moment, i.e. 
	\begin{equation}
	\label{eqn:G_upper_boundk}
	\EE{\Gfcnc(U_1, U_2)^k} \leq \cupper^k, \text{ for } k = 1, 2, \dots K.
	\end{equation}
	Then 
		\begin{enumerate}
		\item 	\[\EE{N_i ^ k} \leq C_k(n \rhon \cupper)^k,\]
		\item 	\[\EE{\p{M_i - \pi N_i}^{2k}} \leq C_k(n \rhon \cupper)^k,\]
	\end{enumerate}
	for $k = 1, \dots, K$, where $C_k$ is some constant depending on $k$. 
\end{enumerate}	
\end{lemm}

\begin{proof}
	For 1, Given $U_i$, $N_i$ follows a Binomial distribution, $N_i \sim \text{Bin}(n-1, \gsml(U_i))$. Note that $\gsml(U_i) \geq \rhon\clower$, by property of Binomial distribution,  $\EE{\frac{1}{N_i^k}1_{\cb{N_i >0}} | U_i } \leq \frac{C_k}{(n \rhon \clower)^k}$, and hence $\EE{\frac{1}{N_i^k}1_{\cb{N_i >0}}} \leq \frac{C_k}{(n \rhon \clower)^k}$. Now given $N_i$, $M_i$ also follows a binomial distribution, $M_i \sim \text{Bin}(N_i, \pi)$. By property of Binomial distribution,
	\[\EE{\p{\frac{M_i - \pi N_i}{N_i}}^{2k} \bigg| N_i} \leq \frac{C_k}{N_i^k}1_{\cb{N_i >0}},\]
	for some constant $C_k$ depending on $k$. Hence $\EE{\p{\frac{M_i - \pi N_i}{N_i}}^{2k}} \leq \EE{\frac{C_k}{N_i^k}1_{\cb{N_i >0}}} \leq \frac{C_k}{(n \rhon \clower)^k}$. (d) follows easily from Chernoff bound. (e) is a direct consequence of (d) by applying the union bound. 

	For 2, again by the fact that given $U_i$, $N_i$ follows a Binomial distribution, $N_i \sim \text{Bin}(n-1, \gsml(U_i))$ and that $\gsml(U_i) \leq \rhon\cupper$, we have $\EE{N_i ^ k|U_i} \leq C_k(n \rhon \cupper)^k$, hence $\EE{N_i ^ k} \leq C_k(n \rhon \cupper)^k$. Given $N_i$, $M_i$ follows a binomial distribution, $M_i \sim \operatorname{Bin}(N_i, \pi)$. Hence $\EE{(M_i - \pi N_i)^{2k}|N_i}  \leq C_k N_i^k$. Hence $\EE{(M_i - \pi N_i)^{2k}}  \leq C_k \EE{N_i^k} \leq C_k(n \rhon \cupper)^k$.
\end{proof}

\begin{lemm}
	\label{lemma:Eij_expectaions}
	
	Under the conditions of Lemma \ref{lemma:boundMN} (2), then
	for $i, j, k$ distinct, let
	\[X_n = \Gfcn(U_i,U_j)^{a_1} \Gfcn(U_j,U_k)^{a_2} \gsml(U_i)^{a_3}\gsml(U_j)^{a_4} \gbsm^{a_5},\]
	\[Y_n = \Hfcn(U_i,U_j)^{b_1} \Hfcn(U_j,U_k)^{b_2} \hsml(U_i)^{b_3}\hsml(U_j)^{b_4} \hbsm^{b_5}, \]
	\[Z_n = E_{ij}^{c_1} E_{jk}^{c_2},\]
	\[X = \Gfcnc(U_i,U_j)^{a_1} \Gfcnc(U_j,U_k)^{a_2} \gsmlc(U_i)^{a_3}\gsmlc(U_j)^{a_4} \gbsmc^{a_5},\]
	\[Y = \Hfcnc(U_i,U_j)^{b_1} \Hfcnc(U_j,U_k)^{b_2} \hsmlc(U_i)^{b_3}\hsmlc(U_j)^{b_4} \hbsmc^{b_5}, \]
	\[Z =\Gfcnc(U_i,U_j)^{c_1} \Gfcnc(U_j,U_k)^{c_2}.\]

	If $k  = \sum a_i + 2 \sum b_i$, then $X_n Y_n/\rhon^k \to X Y$. 

	If $c_1 \leq 1, c_2 \leq 1, k  = \sum a_i + 2 \sum b_i + \sum c_i \leq K$, then
	\[\EE{X_n Y_n Z_n} \leq C \rhon^k \cupper^k, \text{ and } \EE{X_n Y_n Z_n}/\rhon^k \to \EE{X Y Z}. \]
\end{lemm}
\begin{proof}
	The lemma follows easily from the assumption that $\EE{\Gfcnc(U_1, U_2)^k} \leq \cupper^k$ for $1 \leq k \leq K$, Cauchy–Schwarz inequality and dominated convergence theorem. 
\end{proof}

\begin{lemm}
\label{lemma:gamma_small}
Under the conditions of Lemma \ref{lemma:boundMN}, assume furthermore that \eqref{eqn:G_upper_bound} holds, then the number of common neighbors $\gamma_{i,j} = \sum_{k \neq i,j} E_{ik} E_{ij}$ satisfy $\sum_{i,j} \gamma_{i,j} = \oo_p\p{n^3 \rhon^2}$. 
\end{lemm}
\begin{proof}
We start by decomposing $\sum_{i,j} \gamma_{i,j}$. 
\begin{align*}
\sum_{i,j} \gamma_{i,j} = \sum_{i,j} \sum_{k \neq i,j} E_{ik}E_{jk}
= \sum_i \sum_{k \neq i} E_{ik}+ \sum_{i,j,k \text{ all distinct}}  E_{ik}E_{jk} 
\end{align*}
Thus if we take expectation on both hand sides, we get
\begin{align*}
\EE{\sum_{i,j} \gamma_{i,j}} = n(n-1) \gbsm + n(n-1)(n-2) \hbsm. 
\end{align*}
Therefore $\EE{\sum_{i,j} \gamma_{i,j} }/\p{n^3 \rhon^2} \to \hbsmc$ by Lemma~\ref{lemma:Eij_expectaions}. Thus  $\sum_{i,j} \gamma_{i,j} = \oo_p\p{n^3 \rhon^2}$. 

\end{proof}

\begin{lemm}
	\label{lemma:cov_ineq}
	For random variables $X \geq 0$, $Z_1 \geq Y_1 \geq 0$ and $Z_2 \geq Y_2 \geq 0$, if $Y_1 \independent (Y_2, Z_2)$, $Y_2 \independent (Y_1, Z_1)$, and $Z_i <  X$ only when $Z_i = 0$ for $i \in \{0,1\}$, then 
	\begin{align*}
	\Cov{\frac{Y_1}{X+Z_1} , \frac{Y_2}{X+Z_2}} &\leq \EE{\frac{Y_1}{Z_1}}\EE{\frac{Y_2 X}{Z_2^2}} + \EE{\frac{Y_2}{Z_2}}\EE{\frac{Y_1 X}{Z_1^2}}\\
	&\quad \quad + \frac{1}{\EE{Z_1}\EE{Z_2}}\sqrt{\EE{\p{Y_1\frac{Z_1 - \EE{Z_1}}{Z_1}}^2}\EE{\p{Y_2\frac{Z_2 - \EE{Z_2}}{Z_2}}^2}},
	\end{align*}
	with the convention that $0/0 = 0$. In particular if $(X_1, Y_1, X) \stackrel{\text{d}}{=} (X_2, Y_2, X)$, then 
	\[\Cov{\frac{Y_1}{X+Z_1} , \frac{Y_2}{X+Z_2}} \leq
	2\EE{\frac{Y_1}{Z_1}}\EE{\frac{Y_2 X}{Z_2^2}} + \frac{1}{\EE{Z_1}^2}\EE{\p{Y_1\frac{Z_1 - \EE{Z_1}}{Z_1}}^2}.\]

\end{lemm}
\begin{proof}
	\begin{align*}
	\Cov{\frac{Y_1}{X+Z_1},\frac{Y_2}{X+Z_2}} &= 
	\EE{\frac{Y_1Y_2}{(X + Z_1)(X+Z_2)}} - \EE{\frac{Y_1}{X + Z_1}}\EE{\frac{Y_2}{X + Z_2}}\\
	& \leq \EE{\frac{Y_1Y_2}{Z_1 Z_2}} - \EE{\frac{Y_1}{X + Z_1}}\EE{\frac{Y_2}{X + Z_2}}\\
	&= \EE{\frac{Y_1}{Z_1}}\EE{\frac{Y_2}{Z_2}} - \EE{\frac{Y_1}{X + Z_1}}\EE{\frac{Y_2}{X + Z_2}} + \Cov{\frac{Y_1}{Z_1}, \frac{Y_2}{Z_2}}.
	\end{align*}
	Note that 
	\[\EE{\frac{Y_i}{X+Z_i}} \geq \EE{\frac{Y_i(Z_i - X)}{Z_i^2}} = \EE{\frac{Y_i}{Z_i}} - \EE{\frac{X Y_i}{Z_i^2}} \geq 0. \]
	Hence 
	\begin{align*}
	& \quad \quad \Cov{\frac{Y_1}{X+Z_1},\frac{Y_2}{X+Z_2}} \\
	& \leq \EE{\frac{Y_1}{Z_1}}\EE{\frac{Y_2}{Z_2}} - \p{\EE{\frac{Y_1}{Z_1}} - \EE{\frac{X Y_1}{Z_1^2}}} \p{\EE{\frac{Y_2}{Z_2}} - \EE{\frac{X Y_2}{Z_2^2}}} + \Cov{\frac{Y_1}{Z_1}, \frac{Y_2}{Z_2}}\\
	& \leq \EE{\frac{Y_1}{Z_1}}\EE{\frac{Y_2 X}{Z_2^2}} + \EE{\frac{Y_2}{Z_2}}\EE{\frac{Y_1 X}{Z_1^2}} + \Cov{\frac{Y_1}{Z_1}, \frac{Y_2}{Z_2}}.
	\end{align*}
Here the term $\Cov{\frac{Y_1}{Z_1}, \frac{Y_2}{Z_2}}$ can be bounded in the following way
	\begin{align*}
	\Cov{\frac{Y_1}{Z_1}, \frac{Y_2}{Z_2}} & = \Cov{\frac{Y_1}{Z_1} - \frac{Y_1}{\EE{Z_1}}, \frac{Y_2}{Z_2} - \frac{Y_2}{\EE{Z_2}}}\\
	& \leq \sqrt{\Var{\frac{Y_1}{Z_1} - \frac{Y_1}{\EE{Z_1}}}\Var{\frac{Y_2}{Z_2} - \frac{Y_2}{\EE{Z_2}}}}\\
	& \leq \frac{1}{\EE{Z_1}\EE{Z_2}}\sqrt{\EE{\p{Y_1\frac{Z_1 - \EE{Z_1}}{Z_1}}^2}\EE{\p{Y_2\frac{Z_2 - \EE{Z_2}}{Z_2}}^2}}.  
	\end{align*}

\end{proof}

\begin{lemm}
\label{lemm:finite_8_moment}
Under Assumptions \ref{assu:undirected}, \ref{assu:random_graph} and \ref{assu:graphon}, assume (\ref{eqn:G_lower_bound}), (\ref{eqn:rankr_assumption}) and (\ref{eqn:satisfy_bernstein}). Then $G(U_i, U_j)$ has a finite $4^{\text{th}}$ moment, and Lemma \ref{lemma:boundMN} holds for $K = 4$. 
\end{lemm}
\begin{proof}
Note that $\EE{G(U_i, U_j)^4} \leq C \sum_{k=1}^r \lambda_k^4 \EE{\psi_k(U_i)^4 \psi_k(U_j)^4} = C \sum_{k=1}^r \lambda_k^4 \EE{\psi_k(U_i)^4}$ $\EE{\psi_k(U_j)^4} \leq C$, as $\psi_k(U_i)$ satisfies the Bernstein condition (\ref{eqn:berstein}). 
\end{proof}

\begin{lemm}
\label{lemma:summation_decompose}
For indexes in $\cb{1, \dots, n}$, for random variables $X$ indexed by $(i,j)$,
\begin{enumerate}
\item 
\begin{align*}
\p{\sum_{(i,j), i \neq j} X_{ij}}^2 &= \sum_{\substack{(i,j)\\	i \neq j}} X_{ij}(X_{ij} + X_{ji})
+ \sum_{\substack{(i,j,k) \\\text{all distinct}}} X_{ij}(X_{ik} + X_{jk} + X_{ki} + X_{kj} )\\
& \quad \quad + \sum_{\substack{(i,j,k,l) \\\text{all distinct}}} X_{ij}X_{kl}
\end{align*}

\item 
\begin{align*}
\Var{\sum_{(i,j), i \neq j} X_{ij}} &= \sum_{\substack{(i,j)\\	i \neq j}} \Cov{X_{ij}, X_{ij} + X_{ji}}
+ \sum_{\substack{(i,j,k) \\\text{all distinct}}} \Cov{X_{ij},X_{ik} + X_{jk} + X_{ki} + X_{kj}}\\
& \quad \quad + \sum_{\substack{(i,j,k,l) \\\text{all distinct}}} \Cov{X_{ij}, X_{kl}}
\end{align*}

\item 
If $\Var{X_{ij}} \leq a$ for any $i,j$, then  
\begin{align*}
\Var{\sum_{(i,j), i \neq j} X_{ij}} \leq  4n^3 a + \sum_{\substack{(i,j,k,l) \\\text{all distinct}}} \Cov{X_{ij}, X_{kl}}
\end{align*}

%
%

\end{enumerate}
\end{lemm}

\begin{proof}
1 and 2 follow from breaking brackets. 3 is a direct corollary of 2. 
\end{proof}

\begin{lemm}
	\label{lemma:cross_vanish}
	Let $X_i$ be i.i.d. random variables in $\mathbb{R}^{d_1}$.  Let $Y_{ij} \in \mathbb{R}^{d_2}$ be some other i.i.d random variables satisfying $Y_{ij} = Y_{ji}$. Assume $Y$'s are independent of $X$'s. Consider a function $\phi$. If $\phi$ satisfies $\EE{\phi(X_i, X_j, Y_{ij})|X_i} = 0$ and $\EE{\phi(X_i, X_j, Y_{ij})|X_j} = 0$, then $\sum_{(i,j), i \neq j}\phi(X_i, X_j, Y_{ij})$ has zero mean. Furthermore, $\Var{\sum_{(i,j), i \neq j}\phi(X_i, X_j, Y_{ij})} = \sum_{(i,j), i \neq j} \Big( \EE{ \phi(X_i, X_j, Y_{ij})^2} + \EE{\phi(X_i, X_j, Y_{ij}) \phi(X_j, X_i, Y_{ij})}\Big)$. Hence an upper bound of the variance is $\Var{\sum_{(i,j), i \neq j}\phi(X_i, X_j, Y_{ij})} \leq  2\sum_{(i,j), i \neq j} \EE{ \phi(X_i, X_j, Y_{ij})^2} $. 
\end{lemm}

\begin{proof}
	For expectation, $\EE{\phi(X_i, X_j, Y_{ij})}  = \EE{\EE{\phi(X_i, X_j, Y_{ij})|X_i} }= 0$.

Define $\tilde{\phi}(X_i, X_j, Y_{ij}) = \phi(X_i , X_j, Y_{ij}) + \phi(X_j , X_i, Y_{ij}) $. Then $ \tilde{\phi}(X_i, X_j, Y_{ij})$ is symmetric in $i$ and $j$. Furthermore, it still enjoys the property that $\EE{\tilde{\phi}(X_i, X_j, Y_{ij})\Big|X_i} = 0$ and $\EE{\tilde{\phi}(X_i, X_j, Y_{ij})\Big|X_j} = 0$. We can rewrite the target quantity $\sum_{(i,j), i \neq j}\phi(X_i, X_j, Y_{ij})$ = 1/2 $\sum_{(i,j), i \neq j} \tilde{\phi}(X_i, X_j, Y_{ij})$.

	For variance, note that for $i,j,k,l$ district, $\Cov{ \tilde{\phi}(X_i, X_j, Y_{ij}),  \tilde{\phi}(X_k, X_l, Y_{kl})} = 0$ as the two terms are independent. Hence $\EE{ \tilde{\phi}(X_i, X_j, Y_{ij})  \tilde{\phi}(X_k, X_l, Y_{kl})} = 0$. For $i,j,k$ district, 
	\begin{align*}
	\EE{ \tilde{\phi}(X_i, X_j, Y_{ij}),  \tilde{\phi}(X_i, X_k, Y_{ik})} 
	&= \EE{\EE{ \tilde{\phi}(X_i, X_j, Y_{ij})  \tilde{\phi}(X_i, X_k, Y_{ij})|X_i}} = 0.
	\end{align*}
	Hence
	\begin{align*}
	\Var{\sum_{\substack{(i,j)\\ i \neq j}} \tilde{\phi}(X_i, X_j, Y_{ij})} & = \EE{\p{\sum_{\substack{(i,j)\\ i \neq j}} \tilde{\phi}(X_i, X_j, Y_{ij})}^2}\\
	&= 2 \sum_{\substack{(i,j)\\ i \neq j}}\EE{ \tilde{\phi}(X_i, X_j, Y_{ij})^2} + 4 \sum_{\substack{(i,j, k)\\\text{all distinct}}} \EE{ \tilde{\phi}(X_i, X_j, Y_{ij})  \tilde{\phi}(X_i, X_k, Y_{ik})} \\
	&\quad \quad +  \sum_{\substack{(i, j, k, l)\\\text{all distinct}}}\EE{ \tilde{\phi}(X_i, X_j, Y_{ij})  \tilde{\phi}(X_k, X_l, Y_{kl})} \\
	& = 2 \sum_{\substack{(i,j)\\ i \neq j}}\EE{ \tilde{\phi}(X_i, X_j, Y_{ij})^2}. 
	\end{align*}

Note that $\tilde{\phi}(X_i, X_j, Y_{ij})^2 = \phi(X_i, X_j, Y_{ij})^2 + \phi(X_j, X_i, Y_{ij})^2 + 2\phi(X_i, X_j, Y_{ij}) \phi(X_j, X_x, Y_{ij})$. Hence 
\begin{align*}
\sum_{\substack{(i,j)\\ i \neq j}}\EE{ \tilde{\phi}(X_i, X_j, Y_{ij})^2} 
&= 2\sum_{\substack{(i,j)\\ i \neq j}} \Big( \EE{ \phi(X_i, X_j, Y_{ij})^2} + \EE{\phi(X_i, X_j, Y_{ij}) \phi(X_j, X_i, Y_{ij})}\Big). 
\end{align*}
Combining the above results, we get 
\[\Var{\sum_{(i,j), i \neq j}\phi(X_i, X_j, Y_{ij})} = \sum_{(i,j), i \neq j} \Big( \EE{ \phi(X_i, X_j, Y_{ij})^2} + \EE{\phi(X_i, X_j, Y_{ij}) \phi(X_j, X_i, Y_{ij})}\Big).\] 

The upper bound follows trivially from Cauchy-Schwartz inequality. 
\end{proof}

	\begin{lemm}[Bernstein-type bound]
	\label{lemma:berstein_bound}
	For any random variable satisfying the Bernstein condition (\ref{eqn:berstein}) we have
	\[
	\EE{e^{\lambda(X-\mu)}} \leq e^{\frac{\lambda^{2} \sigma^{2} / 2}{1-b|\lambda |}} \quad \text { for all }|\lambda|<\frac{1}{b},
	\]
	and, moreover, the concentration inequality
	\[
	\PP{|X-\mu| \geq t} \leq 2 e^{-\frac{t^{2}}{2\left(\sigma^{2}+b t\right)}} \quad \text { for all } t \geq 0.
	\]
\end{lemm}
\begin{proof}
	It follows directly from \cite{wainwright2019high} Proposition 2.10. 
\end{proof}

\begin{lemm}
	\label{lemm:G_Gn_bound}
	Under Assumption \ref{assu:graphon}, assume (\ref{eqn:rate_of_rhon}), (\ref{eqn:rankr_assumption}) and (\ref{eqn:satisfy_bernstein}). 
\begin{enumerate}
\item $\EE{\Gfcnc(U_1, U_2)^{2m}} \leq C_m$, where $C_m$ is a constant depending on $m$.
\item For a sequence $a_n \to \infty$, for $n$ large, 
\[\PP{\Gfcnc(U_1, U_2) \geq a_n} \leq C e^{ - C\sqrt{a_n}}. \]
\item For $n$ large,
	\[\EE{\abs{\Gfcnc(U_1, U_2)	- \frac{\Gfcn(U_1, U_2)}{\rhon}}^m} \leq C_m e^{- C n^{\kappa_1}},  \]
	where $C_m$ is a constant depending on $m$.
\end{enumerate}

\end{lemm}
\begin{proof}
As we can write $G(U_1, U_2) = \sum_{k=1}^r \lambda_k \psi_k(U_1) \psi_k(U_2)$, we have for $\EE{\Gfcnc(U_1, U_2)^{2m}}$, 
\begin{align*}
\EE{\Gfcnc(U_1, U_2)^{2m}}
\leq C \sum_{k=1}^r \lambda_k^{2m} \EE{\psi_k(U_1)^{2m} \psi_k(U_2)^{2m}}
=  C \sum_{k=1}^r \lambda_k^{2m} \EE{\psi_k(U_1)^{2m}} \EE{\psi_k(U_2)^{2m}}
\leq C_m,
\end{align*}
where the last inequality follows from the Bernstein condition on $\psi_k(U_i)$. 

Now for $\PP{\Gfcnc(U_1, U_2) \geq a_n}$,
\begin{align*}
\PP{\Gfcnc(U_1, U_2) \geq a_n}
&= \PP{ \sum_{k=1}^r \lambda_k \psi_k(U_1) \psi_k(U_2) \geq a_n}
\leq  \sum_{k=1}^r \PP{\lambda_k \psi_k(U_1) \psi_k(U_2) \geq a_n/r}\\
& \leq \sum_{k=1}^r\p{ \PP{ \psi_k(U_1) \geq \sqrt{\frac{a_n}{r \lambda_k}}} + \PP{ \psi_k(U_2) \geq \sqrt{\frac{a_n}{r  \lambda_k}}} }\\
& \leq \sum_{k=1}^r 2 \PP{ \psi_k(U_1) \geq \sqrt{\frac{a_n}{r \lambda_r}}}. 
\end{align*}
For $\psi_k$, if $\EE{\psi_k(U_1)} < \sqrt{\frac{a_n}{r \lambda_r}}/2$, then $ \PP{ \psi_k(U_1) \geq \sqrt{\frac{a_n}{r \lambda_r}}} \leq \PP{ \abs{ \psi_k(U_1) -  \EE{\psi_k(U_1)} } \geq \sqrt{\frac{a_n}{r \lambda_r}}/2}$. By Lemma \ref{lemma:berstein_bound}, this is further bounded above by $2 \operatorname{exp} \sqb{\frac{-\frac{ a_n}{4 r \lambda_r}}{2\p{\sigma^2 + \frac{b}{2} \sqrt{\frac{ a_n}{2r  \lambda_r}}}}} \leq C e^{-C n^{\kappa_1/2}}$. Thus, for $n$ large enough,
\[\PP{\Gfcnc(U_1, U_2) \geq a_n} \leq C e^{-C \sqrt{a_n}}. \]
As a special case, we have 
\[\PP{\Gfcnc(U_1, U_2) \geq \frac{1}{\rhon}} \leq C e^{-C \sqrt{a_n}} \leq C e^{-C n^{\kappa_1/2}}. \]

Note that by definition $\Gfcn(U_1, U_2) = \min(1,\rhon \Gfcnc(U_1, U_2))$, hence we can write
$\Gfcnc(U_1, U_2)	- \frac{\Gfcn(U_1, U_2)}{\rhon} = \p{\Gfcnc(U_1, U_2) - \frac{1}{\rhon}} \mathbf{1}\cb{\Gfcnc(U_1, U_2) \geq \frac{1}{\rhon}} \leq \Gfcnc(U_1, U_2)\mathbf{1}\cb{\Gfcnc(U_1, U_2) \geq \frac{1}{\rhon}} $.  By Cauchy-Schwartz inequality, this implies that
\[ \EE{\abs{\Gfcnc(U_1, U_2)	- \frac{\Gfcn(U_1, U_2)}{\rhon}}^m} \leq \EE{\Gfcnc(U_1, U_2)^{2m}} \PP{\Gfcnc(U_1, U_2) \geq \frac{1}{\rhon}}. \]
Together with the bounds on $\EE{\Gfcnc(U_1, U_2)^{2m}}$ and $\PP{\Gfcnc(U_1, U_2) \geq \frac{1}{\rhon}}$, we have for $n$ large enough,
\[ \EE{\abs{\Gfcnc(U_1, U_2)	- \frac{\Gfcn(U_1, U_2)}{\rhon}}^m} \leq C_m e^{-C n^{-\kappa_1/2}}. \]
\end{proof}

\begin{lemm}
\label{lemm:M_pi_N_ai}
For any vector $a$ that are independent of the treatment assignment $W$,
\[\sum_i (M_i - \pi N_i) a_i =  \oo_p \p{\Norm{E a}}.\]
\end{lemm}
\begin{proof}
For a non-random vector $a$, 
we can write the expression as
\begin{align*}
	\sum_i (M_i - \pi N_i) a_i
=  \sum_{i, j \text{ distinct}} E_{ij} (W_j - \pi) a_i
=  \sum_{j}  (W_j - \pi) \sum_{i \neq j} E_{ij} a_i.
\end{align*}
The second moment of it is 
\begin{align*}
&	 \quad \quad	\EE{ \p{\sum_{j}  (W_j - \pi) \sum_{i \neq j} E_{ij} a_i}^2}
   =  \sum_{j}  \EE{ \p{ (W_j - \pi) \sum_{i \neq j} E_{ij} a_i}^2}\\
& =  \sum_{j}  \EE{ (W_j - \pi)^2} \p{ \sum_{i \neq j} E_{ij} a_i}^2
   = \pi(1-\pi) \sum_j \p{ \sum_{i \neq j} E_{ij} a_i}^2\\
& = \pi(1-\pi) \Norm{E a}^2. 
\end{align*}
As $\sum_i (M_i - \pi N_i) a_i$ is mean zero, the above implies that
\[ 	\sum_i (M_i - \pi N_i) a_i =  \oo_p \p{\Norm{E a}}.\]

The above result can be easily generalized to the case where $a$ is random yet independent of $W$ by only marginalizing over $W$ when taking expectations. 
\end{proof}

\subsection{Proof of Proposition \ref{prop:estimands}}
We start with analyzing the direct effect. 
By \eqref{eq:deriv}, we can Taylor expand $f_i\p{w_i, \frac{M_i}{N_i}}$ into three terms,
\begin{equation}
\label{eq:taylor2}
f_i\p{w_i, \frac{M_i}{N_i}} = f_i(w_i, \pi) + f'_i(w_i, \pi) \p{\frac{M_i}{N_i} - \pi} + f_i''(w_i, \pi_i^{\star}) \p{\frac{M_i}{N_i} - \pi}^2
\end{equation}
for some $\pi_i^{\star}$ between $\pi$ and $\frac{M_i}{N_i}$. Now, $\EE{(M_i/N_i - \pi) \cond G, f(\cdot)} = 0$ and $\EE{(M_i/N_i - \pi)^2 \cond G} = \pi(1 - \pi)/N_i$. 
Therefore, the conditional expectation given the graph and the potential outcome functions is
\[\EE{f_i\p{w_i, M_i/N_i} \mid G, f(\cdot)} = f_i(w_i, \pi)  + \EE{f_i''(W_i, \pi_i^{\star}) \mid G, f(\cdot)}\pi(1 - \pi)/N_i \]
By \eqref{eq:deriv}, we know that $|f_i''(W_i, \pi_i^{\star})| \leq B$, thus $\abs{\EE{f_i''(W_i, \pi_i^{\star}) \mid G, f(\cdot)}\pi(1 - \pi)/N_i} \leq B\pi(1-\pi)/N_i$. Therefore,
\begin{align*}
\btau_{\DIR} &= \frac{1}{n} \sum_i \EE{f_i\p{1, \frac{M_i}{N_i}} - f_i\p{0, \frac{M_i}{N_i}} \mid G, f(\cdot)}\\
&= \frac{1}{n} \sum_i \p{f_i\p{1, \pi} - f_i\p{0, \pi} }
+ \oo \p{\frac{B}{\min_i N_i}}. 
\end{align*}

For the indirect effect, we can use the arguments from Section \ref{subsection:unbiased}, to verify that
	\begin{align*}
\btau_{\IND} =  \frac{1}{n\pi(1-\pi)} \sum_i \EE{Y_i \p{M_i - \pi N_i} \mid G, f(\cdot)}.
\end{align*}
Then Taylor expanding $f_i$ as in \eqref{eq:taylor2},  we can further rewrite $\btau_{\IND} $ as 
\begin{align*}
&\quad \quad \btau_{\IND} \\
	&= \frac{1}{n\pi(1-\pi)} \sum_i \EE{f_i (W_i, \pi)  \p{M_i - \pi N_i} \mid G, f(\cdot)}\\
	& \quad \quad + \frac{1}{n\pi(1-\pi)} \sum_i \EE{ \p{f'_i \p{W_i, \pi}\p{\frac{M_i}{N_i} - \pi} + \frac{1}{2} f''_i \p{W_i, \pi_i^*}\p{\frac{M_i}{N_i} - \pi}^2}   \p{M_i - \pi N_i} \mid G, f(\cdot)}\\
	& = \frac{1}{n\pi(1-\pi)} \sum_i \EE{ f'_i \p{W_i, \pi} \frac{\p{M_i - \pi N_i}^2}{N_i} + \frac{1}{2} f''_i (W_i, \pi_i^*) \frac{\p{M_i - \pi N_i}^3}{N_i^2} \mid G, f(\cdot) }\\
	& = \frac{1}{n\pi(1-\pi)}  \sum_i \EE{f'_i \p{W_i, \pi}\mid G, f(\cdot)} \EE{\frac{\p{M_i - \pi N_i}^2}{N_i} \mid G, f(\cdot)}\\
	& \quad \quad  + \frac{1}{2n\pi(1-\pi)}\sum_i \EE{f''_i (W_i, \pi_i^*) \frac{\p{M_i - \pi N_i}^3}{N_i^2} \mid G, f(\cdot)}\\
	& = D_1 + D_2,
	\end{align*}
where the third line follows from the fact that
$M_i - \pi N_i$ is mean zero and independent of $f_i (W_i, \pi)$ given the graph $G$. 
	
	For $D_1$, since
	$\EE{\p{M_i - \pi N_i}^2/N_i \mid G, f(\cdot)} = \pi(1-\pi)$, we have 
	$D_1 = \frac{1}{n}\sum_i \big(\pi f'_i(1,\pi) + (1-\pi) f'_i(0,\pi) \big).$
	For $D_2$, note that
	\[\EE{\p{f''_i (W_i, \pi_i^*) \frac{\p{M_i - \pi N_i}^3}{N_i^2}}^2 \mid G, f(\cdot)} \leq \frac{B^2 \EE{(M_i - \pi N_i)^6 \mid N_i}}{N_i^4} \leq \frac{C B^2}{N_i}.\]
	Hence 
	$D_2 = \oo\p{\frac{B}{\sqrt{n \rhon}}}.$
	Therefore 
	\[\btau_{\IND} =  \frac{1}{n}\sum_i \p{\pi f'_i(1,\pi) + (1-\pi) f'_i(0,\pi)} + \oo\p{\frac{B}{\sqrt{\min_i N_i}}}. \]

	Putting things together, we get
	\[\btau_{\DIR} = \frac{1}{n}\sum_{i = 1}^n \p{f_i(1,\pi) - f_i(0,\pi)} + \oo\p{\frac{B}{\min_i N_i}}, \text{ and }\]
	\[\btau_{\IND} = \frac{1}{n}\sum_i \p{\pi f'_i(1,\pi) + (1-\pi) f'_i(0,\pi) } + \oo\p{\frac{B}{\sqrt{\min_i N_i}}}. \]

If $1/(\min_i N_i) = o_p \p{1}$, then by taking expectation and limit of the above equation, we get 
	\[\tau_{\DIR} = \frac{1}{n}\sum_{i = 1}^n \EE{f_i(1,\pi) - f_i(0,\pi)}, \text{ and }\]
	\[\tau_{\IND} = \frac{1}{n}\sum_i \EE{\pi f'_i(1,\pi) + (1-\pi) f'_i(0,\pi) }. \]

\subsection{Proof of Proposition \ref{prop:function_mi_ni}}
We will look at the conditional second moment of the term given $G$ and $f_i$'s:
\begin{align*}
& \EE{\p{\frac{1}{n} \sum_{i} (W_i - \pi) \alpha_i \p{\frac{M_i}{N_i}}}^2 \mid G, f(\cdot)}
 = \frac{1}{n^2} \sum_i \EE{(W_i - \pi)^2} \EE{\alpha_i \p{\frac{M_i}{N_i}}\mid G, f(\cdot)}^2 + \\
& \qquad \qquad \qquad \qquad \frac{1}{n^2} \sum_{i,j \text{ distinct}} \EE{(W_i - \pi)(W_j - \pi) \alpha_i \p{\frac{M_i}{N_i}} \alpha_j \p{\frac{M_j}{N_j}}\mid G, f(\cdot)}. 
\end{align*}
For a pair $(i,j)$ such that $E_{ij} = 0$, the term  $\EE{(W_i - \pi)(W_j - \pi) \alpha_i \p{\frac{M_i}{N_i}} \alpha_j \p{\frac{M_j}{N_j}} \mid G, f(\cdot)}$ is zero since $W_i$ is  independent of $M_j/N_j$. Otherwise we can bound the conditional expectation by $CB/(N_i N_j)$ by Cauchy-Schwartz inequality. 
Therefore, 
\begin{align*}
 \EE{\p{\frac{1}{n} \sum_{i} (W_i - \pi) \alpha_i \p{ \frac{M_i}{N_i}}}^2 \mid G, f(\cdot)}
& \leq \frac{1}{n^2} \sum_i  CB^2 / N_i^2 +  \frac{1}{n^2} \sum_{i,j \text{ distinct}} \frac{ C B^2 E_{ij}}{N_i N_j}\\
& \leq \frac{C B^2}{n \min_i N_i^2} + \frac{C B^2}{n \min_i N_i}. 
\end{align*}
Therefore the term $\frac{1}{n} \sum_{i} (W_i - \pi) \alpha_i \p{ M_i/N_i}$ satisfies
\[\frac{1}{n} \sum_{i} (W_i - \pi) \alpha_i \p{ \frac{M_i}{N_i}} = 
\oo_p\p{\frac{B}{\sqrt{n} \min_i N_i} + \frac{B}{\sqrt{n \min_i N_i}} }
= \oo_p\p{ \frac{B}{\sqrt{n \min_i N_i}} }. 
\]

\subsection{Proof of Lemma \ref{lemm:direct}}
We start by rewriting the Horvitz-Thompson estimator. 
\begin{equation}
\label{eq:ht_decompose}
\begin{split}
\htau_{\DIR}^{\HT}
&= \frac{1}{n} \sum_{i} \p{\frac{W_i}{\pi} f_i\p{1, \frac{M_i}{N_i}} -  \frac{1 - W_i}{1 - \pi} f_i\p{0, \frac{M_i}{N_i}}} \\
&= \frac{1}{n} \sum_{i} \p{f_i\p{1, \frac{M_i}{N_i}} - f_i\p{0, \frac{M_i}{N_i}} } \\
& \qquad \qquad \qquad \qquad + \frac{1}{n} \sum_{i} (W_i - \pi) \p{\frac{f_i\p{1, \frac{M_i}{N_i}}}{\pi} + \frac{f_i\p{0, \frac{M_i}{N_i}}}{1 - \pi}}. 
\end{split}
\end{equation}
The estimand $\btau_{\DIR}$ is the conditional expectation of $\tau_{\DIR}^{\HT}$ given the graph $G$ and the potential outcome functions $f_i$'s. For the second term in the above expression, conditional on $G$ and $f_i$'s, its expectation is 0 since $W_i$ is independent of $f_i\p{1, M_i/N_i}$ and $f_i\p{0, M_i/N_i}$. Thus
\begin{equation}
\label{eq:estimand_decompose}
\begin{split}
\btau_{\DIR} & = \EE{\tau_{\DIR}^{\HT} \mid G, f(\cdot)} = \frac{1}{n} \sum_{i} \EE{f_i\p{1, \frac{M_i}{N_i}} - f_i\p{0, \frac{M_i}{N_i}} \mid G, f(\cdot)}.
\end{split}
\end{equation}

By \eqref{eq:deriv}, for $w \in \cb{0,1}$, we can Taylor expand $f_i\p{w, \frac{M_i}{N_i}}$ into four terms,
\begin{equation}
\begin{split}
\label{eq:taylor}
 f_i\p{w, \frac{M_i}{N_i}} &= f_i(w, \pi) 
+ f'_i(w, \pi) \p{\frac{M_i}{N_i} - \pi} \\
& \qquad \qquad + \frac{1}{2} f_i''(w, \pi) \p{\frac{M_i}{N_i} - \pi}^2 + 
r_i(w, M_i/N_i), 
\end{split}
\end{equation}
where $r_i(w, M_i/N_i) = \frac{1}{6} f_i'''(w_i, \pi_i^{\star}) \p{\frac{M_i}{N_i} - \pi}^3$ for some $\pi_i^{\star}$ between $\pi$ and $\frac{M_i}{N_i}$. 
Now, $\EE{(M_i/N_i - \pi)^6 \mid G, f(\cdot)} \leq C/N_i^3$, while again by \eqref{eq:deriv}, we know that $|f_i'''(W_i, \pi_i^{\star})| \leq B$. Thus, we have $\EE{r_i(w, M_i/N_i)^2 \mid G, f(\cdot)} \leq CB/N_i^3$. 
Taking conditional expectation on both hand sides of \eqref{eq:taylor} yields 
\begin{equation}
\label{eq:taylor_expect}
\begin{split}
\EE{f_i\p{w, \frac{M_i}{N_i}} \mid G, f(\cdot)}  & = \EE{f_i(w, \pi) \mid G, f(\cdot)}  
 + \frac{1}{2} f_i''(w, \pi) \EE{\p{\frac{M_i}{N_i} - \pi }^2 \mid G, f(\cdot)} 
 \\
&\qquad \qquad+ 
\EE{r_i(w, M_i/N_i)  \mid G, f(\cdot)}\\
&= f_i(w, \pi)
 + \frac{1}{2} f_i''(w, \pi) \frac{\pi(1-\pi)}{N_i} 
+ \EE{r_i(w,  \frac{M_i}{N_i})  \mid G, f(\cdot)}.
\end{split}
\end{equation}
For notation simplicity, let $r^c_i(w,  M_i/N_i) = r_i(w, M_i/N_i) - \EE{r_i(w, M_i/N_i)  \mid G, f(\cdot)}$. The above bound on $r$ holds for $r^c$ as well, i.e., $\EE{r^c_i(w, M_i/N_i)^2 \mid G, f(\cdot)} \leq CB^2/N_i^3$. 

Plugging  \eqref{eq:taylor} into \eqref{eq:ht_decompose} and \eqref{eq:taylor_expect} into \eqref{eq:estimand_decompose} gives
\begin{equation}
\label{eq:lemma2_major_decompose}
\begin{split}
\htau^{\HT}_{\DIR} - \btau_{\DIR} &= \frac{1}{n} \sum_{i}(W_i - \pi) \p{\frac{f_i\p{1, \pi}}{\pi} + \frac{f_i\p{0, \pi}}{1 - \pi}}\\
& \qquad +  \frac{1}{n} \sum_{i} \p{ \frac{M_i}{N_i} - \pi} \p{f'_i\p{1, \pi} - f'_i\p{0,\pi} }\\
& \qquad + \frac{1}{n} \sum_{i}(W_i - \pi)\p{\frac{M_i}{N_i} - \pi} \p{\frac{f'_i\p{1, \pi}}{\pi} + \frac{f'_i\p{0, \pi}}{1 - \pi}}\\
& \qquad + \frac{1}{2n} \sum_{i}(W_i - \pi) \p{\frac{M_i}{N_i} - \pi}^2 \p{\frac{f''_i\p{1, \pi}}{\pi} + \frac{f''_i\p{0, \pi}}{1 - \pi}}\\
& \qquad + \frac{1}{n} \sum_{i}(W_i - \pi)  \p{\frac{r_i(1, M_i/N_i)}{\pi} + \frac{r_i(0, M_i/N_i)}{1 - \pi}}\\
& \qquad +  \frac{1}{2n} \sum_{i} \sqb{\p{ \frac{M_i}{N_i} - \pi}^2 - \frac{\pi(1 - \pi)}{N_i}}\p{f''_i\p{1, \pi} - f''_i\p{0,\pi} }\\
& \qquad +  \frac{1}{n} \sum_{i}  \p{r^c_i(1,M_i / N_i) - r^c_i(0,M_i / N_i)}.
\end{split}
\end{equation}
For the second summand, we can rearrange terms while preemptively relabeling the summation index as $j$:
\begin{align*}
&\frac{1}{n} \sum_{j = 1}^n \p{\frac{M_j}{N_j} - \pi} \p{f'_j\p{1, \, \pi} - f'_j\p{0, \, \pi}} = \frac{1}{n} \sum_{j= 1}^n \frac{\sum_{i \neq j} E_{ij}\p{W_i - \pi}}{\sum_{k \neq j} E_{jk}} \p{f'_j\p{1, \, \pi} - f'_j\p{0, \, \pi}}. 
\end{align*}
Thus the first two summands in \eqref{eq:lemma2_major_decompose} complete our target expression. We will work on the rest of the terms and show that conditional on the graph $G$ and the potential outcome functions $f_i$'s, they are all negligible. 

For the third summand in \eqref{eq:lemma2_major_decompose}, $ \frac{1}{n} \sum_{i}(W_i - \pi)\p{M_i/N_i - \pi} \p{f'_i\p{1, \pi}/\pi + f'_i\p{0, \pi}/(1 - \pi)}$, we can rewrite it into
\begin{align*}
\frac{1}{n} \sum_{i}(W_i - \pi)\p{\frac{M_i}{N_i} - \pi} \p{\frac{f'_i\p{1, \pi}}{\pi} + \frac{f'_i\p{0, \pi}}{1 - \pi}}
= \frac{1}{n} \sum_{i, j \text{ distinct}} (W_i - \pi) M_{ij} (W_j - \pi), 
\end{align*}
where $M_{ij} =  \p{f'_j (1,\pi)/\pi + f'_j (0,\pi)/(1 - \pi)}E_{ij}/N_j$. Since the vector $W_i - \pi$ has independent and mean-zero entries, we can use the Hanson-Wright inequality
as stated in \citet{rudelson2013hanson} to verify that the above term is bounded in probability to order
$\Norm{M}_{\operatorname{F}}  /n$. 
We will then move to bound the conditional second moment of the Frobenius norm of $M$ given $G$ and $f_i$'s. 
\begin{align*}
\EE{\Norm{M}^2_{\operatorname{F}} \mid G, f(\cdot)} 
&= \sum_{i, j \text{ distinct}}\EE{M_{ij}^2 \mid G, f(\cdot)}
&\leq C B^2 \sum_{i, j \text{ distinct} } E_{ij}/N_j^2
= C B^2 \sum_{j}  1/N_j. 
\end{align*}
Therefore, the third summand in \eqref{eq:lemma2_major_decompose} satisfies
\begin{align*}
 \frac{1}{n} \sum_{i}(W_i - \pi)\p{\frac{M_i}{N_i} - \pi} \p{\frac{f'_i\p{1, \pi}}{\pi} + \frac{f'_i\p{0, \pi}}{1 - \pi}} & = \oo_p\p{\frac{B}{\sqrt{n \min_i N_i}}}. 
\end{align*}

For the fourth summand and the fifth summand in \eqref{eq:lemma2_major_decompose}, they both have the form of $\sum_{i} (W_i - \pi) \alpha_i(M_i/N_i)/n$, where $\alpha_i$ only depends on $f_i$ and $M_i/N_i$, and it satisfies $\mathbb{E} [\alpha_i(M_i/N_i)^2 \mid G, f(\cdot)] \leq C B/N_i^2$ for some constant $C$. Thus Proposition \ref{prop:function_mi_ni} can be applied here to show that both terms are $\oo_p\p{B/\sqrt{n \min_i N_i}}$ conditional on the graph $G$ and the potential outcome functions $f_i$'s. 

For the sixth summand in \eqref{eq:lemma2_major_decompose},  $\frac{1}{2n} \sum_{i} \sqb{\p{ \frac{M_i}{N_i} - \pi}^2 - \frac{\pi(1 - \pi)}{N_i}}\p{f''_i\p{1, \pi} - f''_i\p{0,\pi} }$, we will analyze its conditional second moment separately for diagonal terms and cross terms. We will start with the diagonal terms. Since $M_i$ is a binomial distribution conditional on the graph $G$, we have that 
\[\EE{  \p{\p{\frac{M_i}{N_i} - \pi}^2 - \frac{\pi(1 - \pi)}{N_i}}^2 \mid G, f(\cdot)} \leq \frac{C }{N_i^2}.\]
Thus the sum of the diagonal terms can be bounded by
\begin{align*}
  \frac{1}{n^2} \sum_{i} \EE{\p{\p{ \frac{M_i}{N_i} - \pi}^2 - \frac{\pi(1 - \pi)}{N_i}}^2\p{f''_i\p{1, \pi} - f''_i\p{0,\pi} }^2 \mid G, f(\cdot)}
 \leq \frac{1}{n^2} \sum_i \frac{C_2 B^2}{N_i^2}
\end{align*}
Now for the cross terms, consider $i \neq j$. We rewrite $\p{ M_i - \pi N_i}^2 - \pi(1 - \pi)N_i = \sum_k E_{ik}((W_k - \pi)^2 - \pi (1 - \pi)) + \sum_{k_1, k_2 \text{ distinct}} E_{i k_1} E_{i k_2} (W_{k_1} - \pi) (W_{k_2} - \pi)$. Thus
\begin{align*}
& \qquad \EE{\p{\p{ M_i - \pi N_i}^2 - \pi(1 - \pi)N_i } \p{\p{ M_j - \pi N_j}^2 - \pi(1 - \pi)N_j } \mid G, f(\cdot)}\\
&= \sum_{k \neq i, l \neq j} E_{ik} E_{jk} \EE{((W_k - \pi)^2 - \pi (1 - \pi))((W_l - \pi)^2 - \pi (1 - \pi))} + \\
& \qquad \qquad \sum_{\substack{k_1, k_2: k_1, k_2, i \text{ distinct} \\ l_1, l_2: l_1, l_2, j \text{ distinct}}} E_{i k_1} E_{ i k_2} E_{j l_1} E_{j l_2} \EE{(W_{k_1} - \pi)(W_{k_2} - \pi) (W_{l_1} - \pi)(W_{l_2} - \pi)}\\
& = \sum_{k \neq i,j} E_{ik} E_{jk}  \EE{((W_k - \pi)^2 - \pi (1 - \pi))^2} + \\
& \qquad \qquad  2 \sum_{k_1, k_2: k_1, k_2, i, j\text{ distinct}} E_{i k_1} E_{ i k_2} E_{j k_1} E_{j k_2} \EE{(W_{k_1} - \pi)^2(W_{k_2} - \pi)^2}\\
& \leq 2\p{\sum_{k \neq i,j} E_{ik} E_{jk}}^2 = 2 \gamma_{i,j}^2.
\end{align*}
Therefore if we define $A_i = \sqb{\p{ \frac{M_i}{N_i} - \pi}^2 - \frac{\pi(1 - \pi)}{N_i}}\p{f''_i\p{1, \pi} - f''_i\p{0,\pi} }$, then
$\mathbb{E}[A_i A_j \mid G, f(\cdot)] \leq C B^2  \gamma_{i,j}^2/(N_i^2 N_j^2). $
Hence the sum of the cross terms can be bounded by
\begin{align*}
\frac{1}{n^2} \sum_{i, j \text{ distinct}} \EE{A_i A_j \mid G, f(\cdot)} \leq  \frac{CB^2 \sum_{i,j} \gamma_{i,j}^2}{ n^2\min_i N_i^4} \leq  \frac{CB^2 \sum_{i,j} \gamma_{i,j}}{ n^2\min_i N_i^3}. 
\end{align*}
Combining the results on the diagonal terms, we have that the sixth summand in \eqref{eq:lemma2_major_decompose} satisfies
\begin{align*}
& \frac{1}{2n} \sum_{i} \sqb{\p{ \frac{M_i}{N_i} - \pi}^2 - \frac{\pi(1 - \pi)}{N_i}}\p{f''_i\p{1, \pi} - f''_i\p{0,\pi} }\\
&\qquad \qquad  = \oo_p\p{ \frac{B}{\sqrt{n \min_i N_i}} + \frac{B \sqrt{\sum_{i,j} \gamma_{i,j}}}{n \min_i N_i^{3/2}}}.
\end{align*}

For the seventh summand in \eqref{eq:lemma2_major_decompose}, note that we have 
$\EE{r^c_i(w, M_i/N_i)^2 \mid G, f(\cdot)} \leq CB^2/N_i^3$ and that $\EE{r^c_i(w, M_i/N_i) \mid G, f(\cdot)} = 0$. Since $r^c_i(w, M_i/N_i)$ only depends on $f_i$ and $M_i/N_i$, $r^c_i(w, M_i/N_i)$ will be independent of $r^c_j(w, M_j/N_j)$ conditional on the graph if $i$ and $j$ has no common neighbors. In this case $\mathbb{E} [r^c_i(w, M_i/N_i) r^c_j(w, M_j/N_j) \mid G, f(\cdot)]  = 0$. 
Therefore,
\begin{align*}
&\qquad \EE{ \p{\frac{1}{n} \sum_{i}  r^c_i(w, M_i / N_i) }^2 \mid G, f(\cdot)}\\
&\leq \frac{1}{n^2}\sum_{i,j} \sum_{k \neq i,j} E_{ik} E_{jk} \EE{ r^c_i(w, M_i / N_i) r^c_j(w, M_j / N_j) \mid G, f(\cdot)}\\
&\leq  \frac{C B^2 }{n^2} \sum_{i,j}\frac{\sum_{k \neq i,j} E_{ik} E_{jk}}{N_i^{3/2} N_j^{3/2}}
\leq \frac{CB^2\sum_{i,j} \gamma_{i,j}}{ n^2 \min_i N_i^3}. 
\end{align*}
Hence the seventh summand in \eqref{eq:lemma2_major_decompose} satisfies that 
$\frac{1}{n} \sum_{i}  \p{r^c_i (1, M_i / N_i) - r^c_i (1, M_i / N_i)} = \oo_p\p{B \sqrt{\sum_{i,j} \gamma_{i,j}/(n^2\min_i N_i^3)}}$. 

Combining the above analyses, we have 
\begin{align}
&\htau^{\HT}_{\DIR} - \btau_{\DIR} = \frac{1}{n} \sum_{i = 1}^n \p{\frac{f_i(1,\pi)}{\pi} + \frac{f_i(0,\pi)}{1-\pi}} (W_i - \pi) \\
& \quad\quad\quad\quad+ \frac{1}{n} \sum_{i} \p{\sum_{j \neq i} \frac{E_{ij}}{\sum_{k \neq j} E_{jk}}\p{f'_j(1,\pi) - f'_j(0,\pi)}} (W_i - \pi)+ \oo_p\p{\delta},
\end{align}
where conditional on the graph and the potential outcome functions, 
\[\delta = \oo_p\p{ \frac{B}{\sqrt{n} \min_i N_i} + \frac{B}{\sqrt{n \min_i N_i}} + \frac{B \sqrt{\sum_{i,j} \gamma_{i,j}}}{n\min_i N_i^{3/2}}}. \]
Since $\min_i N_i \geq \sqrt{\min_i N_i}$, we can simplify the above expression into
\[\delta =  \oo_p\p{\frac{B}{\sqrt{n \min_i N_i}} + \frac{B \sqrt{\sum_{i,j} \gamma_{i,j}}}{n\min_i N_i^{3/2}}}\]
as desired. 

We can also characterize $\htau^{\HAJ}_{\DIR}$ via a similar argument. Here, we obtain an analogue to \eqref{eq:lemma2_major_decompose} , except that all terms of the form $(W_i - \pi)$, $1/\pi$ and $1/(1 - \pi)$ are replaced by $(W_i - \hpi)$, $1/\hpi$ and $1/(1 - \hpi)$ respectively, where $\hpi = \sum_{i = 1}^n W_i / n$. Note that all instances
of $\pi$ arising from \eqref{eq:taylor}, including the term $(M_i/N_i - \pi)$, remain unchanged. Now, the second, sixth and seventh summands of this analogue to \eqref{eq:lemma2_major_decompose} are unchanged. For the first summand, note that
\begin{equation*}
\begin{split}
&\frac{1}{n} \sum_{i} \p{\frac{f_i\p{1, \pi}}{\hpi} + \frac{f_i\p{0, \pi}}{1 - \hpi}}(W_i - \hpi)\\
&\qquad   = \frac{1}{n} \sum_{i = 1}^n \Bigg(\frac{f_i(1,\pi)}{\pi} + \frac{f_i(0,\pi)}{1-\pi} - \frac{1}{n}\sum_j \p{{\frac{f_j(1,\pi)}{\pi} + \frac{f_j(0,\pi)}{1-\pi}}} \Bigg)  (W_i - \pi)\\
&\qquad  = \frac{1}{n} \sum_{i = 1}^n \p{\frac{f_i(1,\pi)}{\pi} + \frac{f_i(0,\pi)}{1-\pi} - \EE{\frac{f_i(1,\pi)}{\pi} + \frac{f_i(0,\pi)}{1-\pi}} } (W_i - \pi) + \oo_p\p{\frac{B}{n}}, 
\end{split}
\end{equation*}
where the last line follows from central limit theorem. The third, fourth and fifth summands of this Haj\`ek analogue to \eqref{eq:lemma2_major_decompose} are all of the form $\frac{1}{n} \sum_i (a_i/\hpi -  b_i/ (1 - \hpi))(W_i - \hpi)$, where $\EE{a_i^2 | G , f(\cdot)} \leq CB/N_i$ and $\EE{b_i^2 | G , f(\cdot)} \leq CB/N_i$.  We will then analyze the term $\frac{1}{n} \sum_i a_i(W_i - \hpi)$. Note that $\frac{1}{n} \sum_i a_i(W_i - \hpi) = \frac{1}{n} \sum_i (a_i - \bar{a})(W_i - \pi)$, where $\bar{a} = \sum_i {a_i}/n$. Hence $\frac{1}{n} \sum_i a_i(W_i - \hpi) =  \frac{1}{n} \sum_i a_i (W_i - \pi) +  \frac{\bar{a}}{n} \sum_i (W_i - \pi)$. Since $\EE{a_i^2 | G , f(\cdot)} \leq CB/N_i$, we have that $ \frac{\bar{a}}{n} \sum_i (W_i - \pi) = \oo_p\p{B/\sqrt{n \min_i N_i}}$. Therefore $\frac{1}{n} \sum_i a_i(W_i - \hpi) - \frac{1}{n} \sum_i a_i (W_i - \pi) = \oo_p\p{B/\sqrt{n \min_i N_i}}$. The above analysis holds for $b_i$ as well. Then together with the fact that $\hpi - \pi = \oo_p(1/\sqrt{n})$, we have $\frac{1}{n} \sum_i (a_i/\hpi -  b_i/ (1 - \hpi))(W_i - \hpi) - \frac{1}{n} \sum_i (a_i/\pi -  b_i/ (1 - \pi))(W_i - \pi) = \oo_p(B/\sqrt{n \min_i N_i})$, i.e. the third, fourth and fifth summands of this Haj\`ek analogue to \eqref{eq:lemma2_major_decompose} differ from those of \eqref{eq:lemma2_major_decompose} by an error of size  $\oo_p(B/\sqrt{n \min_i N_i})$, thus completing the proof. 

\subsection{Proof of Theorem \ref{theo:directCLT}}
Note first that by Lemma~\ref{lemma:boundMN} and the assumption that
$\liminf  { \log\rhon}\,/\,{\log n} > - 1$, we have $\min_i {N_i} = \Omega_p \p{n \rhon}$, and by Lemma \ref{lemma:gamma_small}, $\sum_{i,j} \gamma_{i,j} = \oo_p\p{n^3\rhon^2}$. 
Thus by Lemma \ref{lemm:direct} and the follow-up analysis (sufficient set of conditions for $\delta$ to be negligible) we know
	\begin{align*}
	&\tauht_{\DIR} - \bar{\tau}_{\DIR}\\
 &= \frac{1}{n} \sum_{i} \p{\frac{f_i(1,\pi)}{\pi} + \frac{f_i(0,\pi)}{1-\pi} 
	+ \sum_{j \neq i} \frac{E_{ij}}{N_j}\p{f'_j(1,\pi) - f'_j(0,\pi)}} (W_i - \pi) 
+ o_p\p{\frac{B}{\sqrt{n}}}.
	\end{align*}
Define $R_i = \frac{f_i(1,\pi)}{\pi} + \frac{f_i(0,\pi)}{1-\pi} $. Hence we have
	\[\tauht_{\DIR} = \bar{\tau}_{\DIR} + \frac{1}{n} \sum_{i}\p{ R_i + \sum_{j \neq i} \frac{E_{ij}}{N_j}\p{f'_j(1,\pi) - f'_j(0,\pi)}} (W_i - \pi)+ o_p\p{\frac{B}{\sqrt{n}}}.
	\]

	We will analyze the term $\sum_{j \neq i} \frac{E_{ij}}{N_j}\p{f'_j(1,\pi) - f'_j(0,\pi)}$ first. For simplicity, define $\phi_j = f'_j(1,\pi) - f'_j(0,\pi)$. Note that given $U_j$, $\phi_j$ is independent of the set of edges, and that $|\phi_j|\leq 2B$. For fixed $j$, given $U_j$, $E_{jk}$'s are i.i.d Bernoulli's. We can then decompose the term as 
	\begin{align*}
	&\sum_{j \neq i} \frac{E_{ij}}{N_j}\p{f'_j(1,\pi) - f'_j(0,\pi)} = \sum_{j \neq i} \frac{E_{ij}\phi_j}{N_j}\\
	&\qquad \qquad  = \sum_{j \neq i} \frac{E_{ij}\phi_j}{(n-1)\gsml(U_j)} - \sum_{j \neq i} \frac{E_{ij}\phi_j \p{N_j - (n-1)\gsml(U_j)}}{(n-1)\gsml(U_j) N_j}.
	\end{align*}

For the first term $\frac{1}{n-1}\sum_{j \neq i} \frac{E_{ij}\phi_j}{\gsml(U_j)}$, note that for fixed $i$, given $U_i$, $\frac{E_{ij}\phi_j}{\gsml(U_j)}$ are i.i.d. Define $\EE{\frac{E_{ij}\phi_j}{\gsml(U_j)} \Big| U_i} = Q_{n,i}$. Then we have 
	\begin{align*}
&\quad \quad  \EE{\p{\frac{1}{n-1}\sum_{j \neq i}  \frac{E_{ij}\phi_j}{\gsml(U_j)} - Q_{n,i}}^2} =
	\frac{1}{n-1}\EE{\p{ \frac{E_{ij}\phi_j}{\gsml(U_j)} - Q_{n,i}} ^2}\\
& \leq \frac{1}{n-1} \EE{\frac{E_{ij}^2 \phi_i^2}{\gsml(U_j)^2}}
  \leq  \frac{C B^2}{(n-1)\rhon}.
	\end{align*}
In words, the first term $\frac{1}{n-1}\sum_{j \neq i} \frac{E_{ij}\phi_j}{\gsml(U_j)}$ can be well approximated by $Q_{n,i}$ with a small error. 

For the second term $\sum_{j \neq i} \frac{E_{ij} \phi_j(N_j - (n-1)\gsml(U_j))}{(n-1)\gsml(U_j)N_j}$, we start by noting that
\begin{align*}
\frac{E_{ij} \phi_j(N_j - (n-1)\gsml(U_j))}{\gsml(U_j)N_j}
&= \frac{E_{ij} \phi_j((N_j - E_{ij} + 1)- (n-1)\gsml(U_j))}{\gsml(U_j)(N_j - E_{ij} + 1)}
\end{align*}
Conditional on $U$ and $f(\cdot)$, $(N_j - E_{ij})$ is distributed as a Binomial$(n-1, \gsml(U_j))$. Thus 
\begin{align*}
\EE{\frac{(N_j - E_{ij} + 1)- (n-1)\gsml(U_j)}{\gsml(U_j)(N_j - E_{ij} + 1)} \mid U, f(\cdot)}
&= \frac{1}{\gsml(U_j)} - \EE{\frac{n-1}{N_j - E_{ij} + 1} \mid U, f(\cdot)}\\
&= \frac{(1-\gsml(U_j))^n}{\gsml(U_j)} \leq \frac{(1 - c_l \rhon)^n}{c_l \rhon} \leq e^{-C n\rhon},
\end{align*}
for some constant $C$. 
Let $A_{ij} = \frac{E_{ij} \phi_j(N_j - (n-1)\gsml(U_j))}{\gsml(U_j)N_j}$. 
We will now move on to bound the second moment of $\sum_{j \neq i} \frac{E_{ij} \phi_j(N_j - (n-1)\gsml(U_j))}{(n-1)\gsml(U_j)N_j} = \sum_{j \neq i} A_{ij}$. 
We decompose the sum into cross terms and diagonal terms:
$\EE{\p{\sum_{j \neq i} A_{ij}}^2} = \sum_{j \neq i} \EE{A_{ij}^2} + \sum_{j,k: i,j,k\text{ all distinct} } \EE{A_{ij} A_{ik}}$. For the diagonal terms,
\begin{align*}
\EE{A_{ij}^2} &= \EE{\frac{E_{ij} \phi_j^2((N_j - E_{ij} + 1)- (n-1)\gsml(U_j))^2}{\gsml(U_j)^2(N_j - E_{ij} + 1)^2}}\\
&= \EE{\Gfcn(U_i, U_j) \phi_j^2 \EE{\frac{((N_j - E_{ij} + 1)- (n-1)\gsml(U_j))^2}{\gsml(U_j)^2(N_j - E_{ij} + 1)^2} \mid U}}
\leq \frac{CB^2}{n^3 \rhon^2}, 
\end{align*}
by Lemma \ref{lemma:boundMN} and Lemma \ref{lemma:Eij_expectaions}. 
For the cross terms, conditional on $U$ and $f(\cdot)$, $E_{ij}$, $E_{jk}$, $\frac{N_j - (n-1) \gsml(U_j)}{N_j}$ and $\frac{N_k - (n-1) \gsml(U_k)}{N_k}$ are all independent. Thus
\begin{align*}
&\EE{A_{ij} A_{ik}} \\
& = \EE{\frac{G_n(U_i,U_j) G_n(U_i, U_k) \phi_j \phi_k}{(n-1)^2 \gsml(U_j) \gsml(U_k)} \EE{ \frac{N_j - (n-1) \gsml(U_j)}{N_j} \mid U, f(\cdot)}\EE{\frac{N_k - (n-1) \gsml(U_k)}{N_k} \mid U, f(\cdot)}}\\
& = \EE{\frac{G_n(U_i,U_j) G_n(U_i, U_k) \phi_j \phi_k}{(n-1)^2 \gsml(U_j) \gsml(U_k)} \frac{ (1-\gsml(U_j))^n}{\gsml(U_j)} \frac{ (1-\gsml(U_k))^n}{\gsml(U_k)}}\\
& \leq e^{-2 C n\rhon} \EE{\frac{G_n(U_i,U_j) G_n(U_i, U_k) \phi_j \phi_k}{(n-1)^2 \gsml(U_j)^2 \gsml(U_k)^2}} \leq C_1 B^2 e^{-2 C n\rhon} /(n^2\rhon^2),
\end{align*}
where the last inequality follows from Lemma \ref{lemma:Eij_expectaions}. Combing the cross terms and the diagonal terms, we have
\[\EE{\p{\sum_{j \neq i} \frac{E_{ij} \phi_j(N_j - (n-1)\gsml(U_j))}{(n-1)\gsml(U_j)N_j}}^2} \leq \frac{CB^2}{n^2 \rhon^2}.   \]

Combining the results on the first and second terms, we get 
	\[\EE{ \p{\sum_{j \neq i} \frac{E_{ij}}{N_j}\p{f'_j(1,\pi) - f'_j(0,\pi)} - Q_{n,i}}^2} \leq \frac{C B^2}{ n \rhon}.\]
In particular, this implies that  
\begin{align*}
\frac{1}{n} \sum_{i}\p{ \sum_{j \neq i} \frac{E_{ij}}{N_j}\p{f'_j(1,\pi) - f'_j(0,\pi)} - Q_{n,i}} (W_i - \pi) = \oo_p\p{\frac{1}{\sqrt{n}\sqrt{n \rhon}}} = o_p\p{\frac{B}{\sqrt{n}}}. 
\end{align*}

	Thus the Horvitz-Thompson estimator can be written in a form of
	\[
	\tauht_{\DIR}  = \bar{\tau}_{\DIR} + \frac{1}{n} \sum_{i} \p{R_i +Q_{n,i}} (W_i - \pi)+ o_p\p{\frac{B}{\sqrt{n}}},
	\]
	where
	\begin{align*}
	R_i &= \frac{f_i(1,\pi)}{\pi} + \frac{f_i(0,\pi)}{1-\pi}, \text{ and }\\
	Q_{n,i} &= \EE{\frac{E_{ij}\phi_j}{\gsml(U_j)} \Bigg| U_i}\\
	&= \EE{\frac{\Gfcn(U_i, U_j)(f'_j(1,\pi) - f'_j(0,\pi))}{\gsml(U_j)} \Bigg|U_i}.
	\end{align*}
	By the same analysis, we get for H\'ajek estimator,
	\[\tauha_{\DIR}  = \bar{\tau}_{\DIR} + \frac{1}{n} \sum_{i} \p{R_i - \EE{R_i} +Q_{n,i}} (W_i - \pi)+ o_p\p{\frac{B}{\sqrt{n}}}. \]

	Define 
	\begin{align*}
	Q_{i} &= \EE{\frac{\Gfcnc(U_i, U_j)(f'_j(1,\pi) - f'_j(0,\pi))}{\gsmlc(U_j)} \Bigg|U_i}.
	\end{align*}

	Note that as $\Gfcn/\rhon \to \Gfcnc$ and $\gsml/\rhon \to \gsmlc$, together with \eqref{eqn:G_upper_bound}, by dominated convergence theorem, we have $\EE{Q_{n,1}^2} \to \EE{Q_{1}^2}$, $\EE{Q_{n,1}^2} \to \EE{Q_{1}^2}$ and $\EE{Q_{n,1}^2} \to \EE{Q_{1}^2}$. Hence the asymptotic behavior of $Q_{n,i}$ will basically be the same as that of $Q_{i}$. 

Since $\tauht_{\DIR} - \btau_{\DIR}$ and $\tauha_{\DIR} - \btau_{\DIR}$ are averages of i.i.d random variables with a small noise, the central limit theorems follow easily as long as their variance converges. To compute their variance, we will compute the second moment of the terms:
	\begin{align*}
	\EE{(R_i + Q_{n,i})(W_i - \pi)}^2 &= \EE{\p{R_i + Q_{n,i}}^2} \EE{\p{W_i - \pi}^2} = \pi(1-\pi)\EE{\p{R_i + Q_{n,i}}^2}\\
	& \to  \pi(1-\pi)\EE{\p{R_i + Q_i}^2},
	\end{align*}
\begin{align*}
	\EE{(R_i -\EE{R_i} +  Q_{n,i})(W_i - \pi)}^2 &= \EE{R_i -\EE{R_i} +  Q_{n,i}}^2 \EE{W_i - \pi}^2\\
	& = \pi(1-\pi)\p{\Var{R_i+ Q_{n,i}} + (\EE{Q_{n,i})^2}}\\
& \to \pi(1-\pi)\p{\Var{R_i+Q_i} + (\EE{Q_i})^2},
	\end{align*}
Then the central limit theorems follow easily from the variance calculations,
	\begin{align*}
	& \sqrt{n}(\tauht_{\DIR} -\bar{\tau}_{\DIR} ) \overset{d}{\rightarrow} \mathcal{N}\p{0, \pi(1-\pi)\EE{(R_i+Q_i)^2}},\\
	& \sqrt{n}(\tauha_{\DIR} -\bar{\tau}_{\DIR} ) \overset{d}{\rightarrow} \mathcal{N}\p{0, \pi(1-\pi)\p{\Var{R_i+Q_i} + \EE{Q_i^2}}}.
	\end{align*}

For the population-level estimands, if $\sqrt{n} \rhon \to \infty$, then by Proposition \ref{prop:estimands}, we have 
\begin{align*}
	\tauht_{\DIR} - \tau_{\DIR} &=   \frac{1}{n} \sum_{i} \p{f_i(1,\pi) - f_i(0,\pi) +  \p{R_i +Q_{n,i}} (W_i - \pi)}+ o_p\p{\frac{B}{\sqrt{n}}},\\
 \tauha_{\DIR} - \tau_{\DIR}&=   \frac{1}{n} \sum_{i} \p{f_i(1,\pi) - f_i(0,\pi) + \p{R_i - \EE{R_i} +Q_{n,i}} (W_i - \pi)} + o_p\p{\frac{B}{\sqrt{n}}}. 
\end{align*}
Then the central limit theorems follow from similar analysis as above,
	\begin{align*}
	& \sqrt{n}(\tauht_{\DIR} -\tau_{\DIR} ) \overset{d}{\rightarrow} \mathcal{N}\p{0, \sigma_0^2 + \pi(1-\pi)\EE{(R_i+Q_i)^2}},\\
	& \sqrt{n}(\tauha_{\DIR} -\tau_{\DIR} ) \overset{d}{\rightarrow} \mathcal{N}\p{0, \sigma_0^2 + \pi(1-\pi)\p{\Var{R_i+Q_i} + \EE{Q_i^2}}}.
	\end{align*}

\subsection{Proof of Proposition \ref{prop:unbiased_rate}}

	Note that $\tauuu_{\TOT} =\tauuu_{\IND} + \tauht_{\DIR} = \tauuu_{\IND} + \oo_p(B)$. Now for $\tauuu_{\IND}$,
	\begin{align*}
	\tauuu_{\IND} &= \frac{1}{n\pi(1-\pi)} \sum_i \p{f_i (W_i, \pi) + f'_i \p{W_i, \pi_i^*}\p{\frac{M_i}{N_i} - \pi}} \p{M_i - \pi N_i}\\
	& = \frac{1}{n\pi(1-\pi)}\Bigg[\sum_i f_i (W_i, \pi) \p{M_i - \pi N_i} + \sum_i f'_i \p{W_i, \pi_i^*}\p{\frac{M_i}{N_i} - \pi}\p{M_i - \pi N_i} \Bigg]\\
	&= \frac{1}{n\pi(1-\pi)}\Big[ B_1 + B_2 \Big],
	\end{align*}
	where $B_1$, $B_2$ are the two summations in square bracket  respectively. 

	For $B_2$,
	\begin{align*}
	B_2 &= \sum_i f'_i \p{W_i, \pi_i^*}\p{\frac{M_i}{N_i} - \pi}\p{M_i - \pi N_i}= \sum_i f'_i \p{W_i, \pi}\p{M_i - \pi N_i}^2/N_i.
	\end{align*}
	By Lemma \ref{lemma:boundMN},
	\[\EE{|B_2|} \leq \sum_i B \pi(1-\pi) = n C B \pi(1-\pi).\]
	By Markov inequality, $B_2 = \oo_p(B n)$.

	For $B_1$,
	\begin{align*}
	B_1 &= \sum_i f_i (W_i, \pi) \p{M_i - \pi N_i}\\
	&= \sum_i \p{W_i f_i (1,\pi) + (1-W_i)f_i (0,\pi)} \p{M_i - \pi N_i}\\
	& = \sum_i \p{\pi f_i(1,\pi) + (1-\pi)f_i (0,\pi)} \sum_j E_{ij} (W_j-\pi)  \\
	&\quad \quad  + \sum_i (W_i - \pi) (f_i(1,\pi) - f_i (0,\pi)) \sum_j E_{ij} (W_j-\pi)
	\end{align*}
	The second term can be written as 
	$\sum_{i,j ,i\neq j} (W_i - \pi)(W_j-\pi) A_{ij}$, where $A_{ij} = E_{ij}  (f_i(1,\pi) - f_i (0,\pi))$. Here $A_{ij}$'s are independent of $W$ and are bounded. Hence its second moment is bounded by $\sum_{i,j ,i\neq j} E(A_{ij}^2) = \oo(B^2 n^2)$. This shows that the second term is $\oo_p(B n)$, again by Markov inequality. 

	Combining the above, we get that
	\begin{align*}
	\tauuu_{\IND} &= \frac{1}{n\pi(1-\pi)} \sum_i \p{\pi f_i(1,\pi) + (1-\pi)f_i (0,\pi)} \sum_{j \neq i} E_{ij} (W_j-\pi) + \oo_p(B)\\
	& = \frac{1}{n\pi(1-\pi)} \sum_j \sqb{\sum_{i \neq j} \p{\pi f_i(1,\pi) + (1-\pi)f_i (0,\pi)}  E_{ij}} (W_j-\pi) + \oo_p(B)\\
	& = \frac{1}{n\pi(1-\pi)}\sum_i \sqb{\sum_{j \neq i}\p{\pi f_i(1,\pi) + (1-\pi)f_i(0,\pi)}E_{ij}  }(W_i - \pi) + \oo_p(B), 
	\end{align*}
	and hence the same holds for estimator of the total effect,
	\[ \tauuu_{\TOT} = \frac{1}{n\pi(1-\pi)}\sum_i \sqb{\sum_{j \neq i}\p{\pi f_i(1,\pi) + (1-\pi)f_i(0,\pi)}E_{ij}  }(W_i - \pi) + \oo_p(B).  \]

	Now we are interested in studying the variance of 
	\[\frac{1}{n\pi(1-\pi)}\sum_i \sqb{\sum_{j \neq i}\p{\pi f_i(1,\pi) + (1-\pi)f_i(0,\pi)}E_{ij}  }(W_i - \pi) .\] 
	Note that the expression has a form of $\frac{1}{n}\sum_i A_i (W_i - \pi)$, where $A_i$ is the term in the square bracket in the previous line.  $A_i$'s are independent of $W_i$'s, and the $A_i$'s are identically distributed. Hence for $i \neq j$, $\EE{A_i (W_i - \pi)A_j(W_j - \pi)} = \EE{A_i A_j} \EE{W_i - \pi}\EE{W_j - \pi} = 0$. Hence $\EE{\p{\frac{1}{n}\sum_i A_i (W_i - \pi)}^2} = \frac{1}{n^2} \sum_i \EE{A_i^2  (W_i - \pi)^2 } =   \frac{1}{n^2} \sum_i \EE{A_i^2}  \EE{(W_i - \pi)^2} = \frac{ \pi(1-\pi)}{n} \EE{A_1^2}$. Therefore
	\begin{align*}
	& \quad \quad \frac{1}{n\rhon^2} \Var{\frac{1}{n\pi(1-\pi)}\sum_i \sqb{\sum_{j \neq i}\p{\pi f_i(1,\pi) + (1-\pi)f_i(0,\pi)}E_{ij}  }(W_i - \pi) } \\
	&= \frac{1}{n^2\rhon^2}  \EE{\p{\sum_{j\neq 1} \p{\pi f_1(1,\pi) + (1-\pi)f_1(0,\pi)} E_{1j} }^2}\\
	&= \frac{1}{n^2\rhon^2}  \EE{\p{\p{\pi f_1(1,\pi) + (1-\pi)f_1(0,\pi)} N_1 }^2} \\ 
	& = \frac{1}{n^2\rhon^2}  \EE{ \EE{\p{\pi f_1(1,\pi) + (1-\pi)f_1(0,\pi)}^2 N_1^2 \Big|U_1  } }\\
	&= \frac{1}{n^2\rhon^2}  \EE{ \EE{\p{\pi f_1(1,\pi) + (1-\pi)f_1(0,\pi)}^2 \Big|U_1  }  \EE{N_1^2 \Big|U_1  } }\\
	&= \frac{1}{n^2\rhon^2}  \EE{ \p{\pi f_1(1,\pi) + (1-\pi)f_1(0,\pi)}^2 \p{ (n-1)^2\gsml(U_1) ^2 + (n-1)\gsml(U_1) (1-\gsml(U_1) )} }\\
	& \to  \EE{\p{\pi f_1(1,\pi) + (1-\pi)f_1(0,\pi)}^2 \gsmlc(U_1)^2 }\text{ by Lemma \ref{lemma:Eij_expectaions}}.
	\end{align*}

	Therefore $\Var{\tauuu_{\IND}} \sim \nu n \rhon^2$ and $\Var{\tauuu_{\TOT}} \sim \nu n \rhon^2$, where 
	\[\nu = \EE{\p{\pi f_1(1,\pi) + (1-\pi)f_1(0,\pi)} \gsmlc(U_1) }^2.\] 

\subsection{Random matrix related lemmas}

We'll present a few lemmas related to the $\psih$. In rough words, we show that $\psih$ is close $\psi$.

Without further specification, all the lemmas in this section will be under assumptions \ref{assu:undirected}, \ref{assu:random_graph} and \ref{assu:graphon}, and assuming (\ref{eqn:rate_of_rhon}), (\ref{eqn:rankr_assumption}) and (\ref{eqn:satisfy_bernstein}). 

\begin{lemm}
\label{lemm:E_op_norm}
\begin{enumerate}
\item 
\[\Norm{\Es - \rhon G}_{op} = \oo_p\p{\sqrt{n} \rhon}\]
\item 
\[\Norm{E - \Es}_{op} = \oo_p\p{ \p{\log n}^2 \sqrt{n \rhon} }\]
\item 
\[\Norm{E - \rhon G}_{op} = \oo_p\p{ \p{\log n}^2 \sqrt{n \rhon} }.\]
\end{enumerate}

\end{lemm}
\begin{proof}
To show 1, note first that $\Norm{\Es - \rhon G}_{op} \leq \Norm{\Es - \rhon G}_F$. But the Frobenius norm is bounded by
\begin{align*}
\EE{\Norm{\Es - \rhon G}_F^2} 
&= \sum_{(i,j), i \neq j } \EE{\p{G_n(U_i, U_j) - \rhon G(U_i,U_j)}^2} + \sum_i \rhon^2 \EE{G(U_i,U_i)}^2\\ &\leq  n^2 \expb + Cn \rhon^2,
\end{align*}
by Lemma \ref{lemm:G_Gn_bound}. We'll then move on to bound $\Norm{E - \Es}_{op}$. Consider the event of 
\[I = \cb{\max_{i,j} G(U_i, U_j) \leq \p{\log n}^4 }.\] By Lemma \ref{lemm:G_Gn_bound}, the event happens with probability at least
\begin{align*}
\PP{I} \geq 1 - n^2 \PP{G(U_i, U_j) \geq \p{\log n}^4} \geq 1 - C n^2 e^{- C \p{\log n}^2} \to 1.  
\end{align*}
On the other hand, on the event $I$, we have every entry in $\Es$ bounded above by $\rhon \p{\log n}^4$: $\Es_{ij} = G_n(U_i, U_j) \leq \rhon G(U_i, U_j) \leq \rhon \p{\log n}^4$. Then we can apply standard spectral bounds on random matrices to $\Norm{E - \Es}_{op}$. Specifically, we make use of Theorem 5.2 in \cite{lei2015consistency}. \cite{lei2015consistency} show that with a probability converging to 1, $\Norm{E - \Es}_{op} \leq C \sqrt{d}$ if $n \max \Es_{ij} \leq d$. Hence we have on the event $I$, $\Norm{E - \Es}_{op} \leq C \sqrt{ n \rhon  \p{\log n}^4}$ with a probability converging to 1. Together with the lower bound on $\PP{I}$, we have
\[ \Norm{E - \Es}_{op} = \oo_p\p{ \p{\log n}^2 \sqrt{n \rhon} }. \]
Therefore
\begin{align*}
\Norm{E - \rhon G}_{op} 
&\leq \Norm{\Es - \rhon G}_{op} +  \Norm{E - \Es}_{op}\\
& = \oo_p\p{ \sqrt{n^2 \expb} +  C \sqrt{n} \rhon + \p{\log n}^2 \sqrt{n \rhon} }\\
&=  \oo_p\p{ \p{\log n}^2 \sqrt{n \rhon} }.
\end{align*}
\end{proof}

\begin{lemm}
\label{lemm:tilde_close}
\begin{enumerate}
\item $\abs{ \lambdat_k - n \lambda_k} = \oo_p\p{\sqrt{n}}$. 
\item There exists an $r \times r$ orthogonal matrix $\tilde{R}$, such that $\Norm{\Psit \tilde{R} - \Psi}_F =  \oo_p\p{1}$. If we write $\PsiRt = \Psit \tilde{R}$, and let $\psiRt_k$ be the $k$-th column of $\PsiRt$, then $\Norm{\psiRt_k - \psi_k} = \oo_p\p{1}$.
\end{enumerate}
\end{lemm}
\begin{proof}

Note that $\Gm = \sum_{k=1}^r \lambda_k \psi_k \psi_k^T$ is a rank-$r$ matrix, but as $\psi_k$'s are not exactly norm $\sqrt{n}$ and not exactly orthogonal, it's eigenvectors of $\Gm$ are not $\psi_k$. But we'll show in this lemma that they are close enough. Let $A = \cb{a_{ij}}$ be an $r \times r$ matrix such that $\Psi = \Psit A$. Let $\Lambda = \operatorname{diag}(\lambda_i)$ and $\tilde{\Lambda} = \operatorname{diag}(\tilde{\lambda}_i)$. By construction, we know $G = \frac{1}{n} \tilde{\Psi} \tilde{\Lambda} \tilde{\Psi}^T = \Psi \Lambda \Psi^T = \frac{1}{n} \tilde{\Psi} A \tilde{\Lambda} A^T \tilde{\Psi}^T$. Thus $A\tilde{\Lambda} A^T/n = \Lambda$. We'll show that the matrix $A$ is close to orthogonal. Note first that, by law of large numbers, $\Norm{\Psi^T \Psi - nI}_F = \oo_p\p{\sqrt{n}}$. We also have $\Psi^T \Psi = \smash{A^T \Psit^T \Psit A = n A^TA}$. Hence $\Norm{A^T A - I}_F = \oo_p\p{1/\sqrt{n}}$.

We start by looking at $\psi_1$. We write $\psi_1$ as a linear combination of $\psit$'s. 
\[ \psi_1 = a_{11} \psit_1 + a_{12} \psit_2 + \dots + a_{1r} \psit_r. \]
The fact that $\Norm{A^T A - I}_F = \oo_p\p{1/\sqrt{n}}$ implies that  $\sum_k a_{1k}^2 = 1 + \oo_p\p{1/\sqrt{n}}$. As $\psit$'s are the eigenvectors, i.e. $\Gm \psit_k = \lambdat \psit_k$, we have
\[ \Gm \psi_1 = \lambdat_1 a_{11} \psit_1 + \lambdat_2 a_{12} \psit_2 + \dots + \lambdat_r  a_{1r} \psit_r. \]
Thus for $\lambda_1$, 
\[ \Gm \psi_1 - n\lambda_1 \psi_1 = \p{\lambdat_1  - n \lambda_1} a_{11} \psit_1 + \p{\lambdat_2  - n \lambda_1}a_{12} \psit_2 + \dots + \p{\lambdat_r  - n \lambda_1} a_{1r} \psit_r, \]
hence
\[ \Norm{\Gm \psi_1 - n\lambda_1 \psi_1}^2 = n \sum_{k=1}^{r} a_{1k}^2 \p{\lambdat_k  - n \lambda_1}^2. \]
On the other hand, note that as $\Gm = \sum_{k=1}^r \lambda_k \psi_k \psi_k^T$,
\[ \Gm \psi_1 - n \lambda_1 \psi_1 =  \lambda_1 \p{\psi_1^T\psi_1 - n} \psi_k + \sum_{k=2}^r \lambda_k \p{\psi_k^T\psi_1} \psi_k.\]
Since $\p{\psi_1^T\psi_1 - n} = \oo_p\p{\sqrt{n}}$ and $\psi_k^T\psi_1 = \oo_p\p{\sqrt{n}}$, we have $ \Norm{\Gm \psi_1 - n \lambda_1 \psi_1} = \oo_p\p{n}$. Therefore,
\[n \sum_{k=1}^{r} a_{1k}^2 \p{\lambdat_k  - n \lambda_1}^2 = \oo_p\p{n^2}, \text{ i.e. } \sum_{k=1}^{r} a_{1k}^2 \p{\lambdat_k  - n \lambda_1}^2 = \oo_p\p{n}.\]
Let $l = \argmax_k a_{1k}^2$. Then $a_{1l}^2 \p{\lambdat_l - n \lambda_1}^2 \leq \sum_{k=1}^{r} a_{1k}^2 \p{\lambdat_k  - n \lambda_1}^2 = \oo_p\p{n}$. Yet on the other hand $a_{1l}^2 \geq \sum_k a_{1k}^2/r \geq 1/r + \oo_p\p{1/\sqrt{n}}$. This implies that $\p{\lambdat_l - n \lambda_1}^2 = \oo_p\p{n}$, i.e. 
$\lambdat_l - n \lambda_1 = \oo_p\p{\sqrt{n}}$.

The result above is not specific to $\psi_1$. In fact it works for any $\psi_k$. With similar arguments, we are able to show that there exists $l_k$ (that might be data dependent), s.t. 
$\lambdat_{l_k} - n \lambda_k = \oo_p\p{\sqrt{n}}$. 

Were the eigenvalues all different, i.e., $\lambda_1 > \lambda_2 > \dots > \lambda_r$, then the above arguments would imply $\lambdat_{k} - n \lambda_k = \oo_p\p{\sqrt{n}}$. A bit more work is needed if some eigenvalues are the same. Assume $\lambda_1 = \dots = \lambda_{r_0} > \lambda_{r_0 + 1}$. Then for $k > r_0$,  $(\lambdat_{k} - n \lambda_1)^2 = \Omega_p\p{n^2}$. Since $a_{1k}^2 (\lambdat_k - n \lambda_1)^2 \leq \sum_{i=1}^{r} a_{1i}^2 (\lambdat_i  - n \lambda_1)^2 = \oo_p\p{n}$, we have $a_{1k} = \oo_p\p{\frac{1}{n}}$. Again, the above arguments work for indices other than $1$. Specifically, we have $a_{ij} = \oo_p\p{\frac{1}{n}}$ if $\lambda_i \neq \lambda_j$. We partition the matrices $A$,  $\Lambda$, $\tilde{\Lambda}$ into blocks:
\[ A=
\left[
\begin{array}{c c }
A_{11} & A_{12} \\
A_{21} & A_{22}
\end{array}
\right], 
\quad 
\Lambda = 
\left[
\begin{array}{c c }
\lambda_1 I_{r_0 \times r_0} & 0 \\
0 & \Lambda_{22}
\end{array}
\right], 
\quad 
\tilde{\Lambda} = 
\left[
\begin{array}{c c }
\tilde{\Lambda}_{11} & 0 \\
0 & \tilde{\Lambda}_{22}
\end{array}
\right], 
\]
where $A_{11}$ and $\tilde{\Lambda}_{11}$ are $r_0 \times r_0$ matrices. Since $A \tilde{\Lambda} A^T/n = \Lambda$, we have $A_{11} \tilde{\Lambda}_{11} A_{11}^T + A_{12} \tilde{\Lambda}_{22} A_{12}^T = n \lambda_1 I$. But each entry in $A_{12}$ is $\oo_p\p{\frac{1}{n}}$, thus $\Norm{\smash{ A_{11}\tilde{\Lambda}_{11}A_{11}^T/n - \lambda_1 I }}_F = \oo_p\p{\frac{1}{n^2}}$. We also note that since $\Norm{A^T A - I}_F = \oo_p\p{1/\sqrt{n}}$, we have $\Norm{A_{11}^T A_{11} - I}_F = \oo_p\p{1/\sqrt{n}}$. Thus
\begin{align*}
\Norm{ A_{11}\tilde{\Lambda}_{11}A_{11}^T/n - \lambda_1 I }_F 
&\geq \Norm{\tilde{\Lambda}_{11}/n - \lambda_1 I}_F - \Norm{\tilde{\Lambda}_{11}(A_{11}^T A_{11} - I)/n}_F\\
&\geq \Norm{\tilde{\Lambda}_{11}/n - \lambda_1 I}_F - \Norm{\tilde{\Lambda}_{11}/n}_F \Norm{A_{11}^T A_{11} - I}_F\\
& = \Norm{\tilde{\Lambda}_{11}/n - \lambda_1 I}_F  + \oo_p\p{1/\sqrt{n}}. 
\end{align*}
This implies that $ \Norm{\smash{\tilde{\Lambda}_{11}/n - \lambda_1 I}}_F =  \oo_p\p{1/\sqrt{n}}$. Hence $\lambdat_{k} - n \lambda_k = \oo_p\p{\sqrt{n}}$ for $k \leq r_0$. One can apply the same arguments to $k \geq r_0$ and obtain $\lambdat_{k} - n \lambda_k = \oo_p\p{\sqrt{n}}$ for any $k \leq r$. 

We will move on to study the eigenvectors. Let $A^T A = U D U^T$ be the eigen decomposition of $A^T A$. Then $\Norm{A^T A - I}^2_F = \operatorname{tr}\p{(D - I)^2}$.  Since $\Norm{A^T A - I}_F =  \oo_p\p{1/\sqrt{n}}$, we have $D_{ii} - 1 = \oo_p\p{1/\sqrt{n}}$. Take $\tilde{R} = A U D^{-\frac{1}{2}} U^T$. We note that the $\tilde{R} $ defined this way is indeed orthogonal. In fact, $\tilde{R}^T \tilde{R} = U D^{-\frac{1}{2}} U^T A^T A U D^{-\frac{1}{2}} U^T = U D^{-\frac{1}{2}} U^T U D U^T U D^{-\frac{1}{2}} U^T = I$. This $\tilde{R}$ satisfies that $\Psit \tilde{R} = \Psit A U D^{-\frac{1}{2}} U^T = \Psi U D^{-\frac{1}{2}} U^T$. Thus
\begin{align*}
\Norm{\Psi - \Psit \tilde{R}}_F^2
=  \Norm{\Psi - \Psi U D^{-\frac{1}{2}} U^T}_F^2
\leq \Norm{\Psi}_F^2 \Norm{I - U D^{-\frac{1}{2}} U^T}_F^2
= r n \Norm{I - D^{-\frac{1}{2}}}_F^2
= \oo_p\p{1}. 
\end{align*}
\end{proof}

\begin{lemm}
\label{lemm:eigen_value}
For $k, l \in \cb{1, \dots, r}$, if $\lambda_k \neq \lambda_l$, then 
\begin{enumerate}
\item 
$\abs{\lambdah_k - n \rhon \lambda_k} = \oo_p\p{\sqrt{n \rhon } \p{\log n}^2}$, 
$\abs{\lambdas_k - n \rhon \lambda_k} = \oo_p\p{\sqrt{n} \rhon}$,
$\abs{\lambdat_k - n \lambda_k} = \oo_p\p{\sqrt{n}}$, 
\item 
$\abs{\lambdah_k - \lambdah_l} = \Omega_p\p{n \rhon}$, 
$\abs{\lambdas_k - \lambdas_l} = \Omega_p\p{n \rhon}$,
$\abs{\lambdat_k - \lambdat_l} = \Omega_p\p{n}$. 
\end{enumerate}
\end{lemm}
\begin{proof}
For 1, by triangle inequality $\abs{\lambdah_k - n \rhon \lambda_k} \leq \abs{\lambdah_k - n \rhon \lambdat_k}
+ n \rhon \abs{\lambdat_k - \lambda_k} \leq \Norm{E - \rhon \Gm}_{op} + n\rhon \abs{\lambdat_k - \lambda_k} = \oo_p\p{\sqrt{n \rhon } \p{\log n}^2}$ by Lemma \ref{lemm:E_op_norm} and \ref{lemm:tilde_close}. Similarly, $\abs{\lambdas_k - n \rhon \lambda_k} \leq \abs{\lambdas_k - n \rhon \lambdat_k}
+ n \rhon \abs{\lambdat_k - \lambda_k} \leq \Norm{\Es - \rhon \Gm}_{op} + n\rhon \abs{\lambdat_k - \lambda_k} = \oo_p\p{\sqrt{n} \rhon}$.  

Statement 2 is a direct consequence of 1.
\end{proof}
\subsection{Proof of Lemma \ref{lemm:a_psi_close}}

As a preliminary to our proof, we recall the statement of the Davis-Kahan theorem as given in \citet*{yu2015useful}:
Let $\Sigma, \hat{\Sigma} \in \mathbb{R}^{p \times p}$ be symmetric, with eigenvalues $\lambda_{1} \geq \ldots \geq \lambda_{p}$ and $\hat{\lambda}_{1} \geq$
$\ldots \geq \hat{\lambda}_{p}$ respectively. Fix $1 \leq r \leq s \leq p$ and assume that $\min \left(\lambda_{r-1}-\lambda_{r}, \lambda_{s}-\lambda_{s+1}\right)>0$
where $\lambda_{0}:=\infty$ and $\lambda_{p+1}:=-\infty .$ Let $d:=s-r+1,$ and let $V=\left(v_{r}, v_{r+1}, \ldots, v_{s}\right) \in \mathbb{R}^{p \times d}$
and $\hat{V}=\left(\hat{v}_{r}, \hat{v}_{r+1}, \ldots, \hat{v}_{s}\right) \in \mathbb{R}^{p \times d}$ have orthonormal columns satisfying $\Sigma v_{j}=\lambda_{j} v_{j}$ and
$\hat{\Sigma} \hat{v}_{j}=\hat{\lambda}_{j} \hat{v}_{j}$ for $j=r, r+1, \ldots,$ s. Then there exists an orthogonal matrix $\hat{O} \in \mathbb{R}^{d \times d}$ such that
\begin{equation}
\label{eq:davis_kahan}
\|\hat{V} \hat{O}-V\|_F\leq \frac{2^{3 / 2} \min \left(d^{1 / 2}\|\hat{\Sigma}-\Sigma\|_{op},\|\hat{\Sigma}-\Sigma\|_F\right)}{\min \left(\lambda_{r-1}-\lambda_{r}, \lambda_{s}-\lambda_{s+1}\right)}.
\end{equation}
Specifically, let $\hat{V} ^T V = O_1 D O_2^T$ be the singular value decomposition of $\hat{V} ^T V$, then $\hat{O}$ is constructed by taking $\hat{O} = O_1 O_2^T$. 

Together with Lemma \ref{lemm:E_op_norm} and Lemma \ref{lemm:eigen_value}, if we apply \eqref{eq:davis_kahan} to $\Psih$ and $\Psis$, we get that there exists an $r \times r$ orthogonal matrix $\tilde{R}^*$ such that
$\Norm{ \smash{\Psih \hat{R}^* - \Psis}}_F = \oo_p \smash{ \p{\frac{\p{\log n }^2}{ \sqrt{\rhon}}}}.$
And if we instead apply \eqref{eq:davis_kahan} to $\smash{\Psis}$ and $\smash{\Psit}$, we get that there exists an $r \times r$ orthogonal matrix $\smash{\hat{R}^*}$ such that
$\Norm{\smash{\Psis \tilde{R}^*- \Psit}}_F = \oo_p \smash{\p{ e^{-C n^{\kappa/2}} /\sqrt{\rhon}}}.$
Combining the above two results, we have that there exists $r \times r$ orthogonal matrices $\smash{\hat{R}}$ and $\smash{R^*}$ such that
\begin{equation}
\label{eq:psi_close00}
\begin{split}
\Norm{\Psih \hat{R} - \Psis R^*}_F &= \oo_p \p{\frac{\p{\log n }^2}{ \sqrt{\rhon}}},\\
\Norm{\Psis R^*- \Psit \tilde{R}}_F &= \oo_p \p{ e^{-C n^{\kappa/2}} /\sqrt{\rhon}}.
\end{split}
\end{equation}
For notation simplicity, we write $\PsiRh = \Psih \hat{R}$, $\PsiRs = \Psis R^*$ and $\PsiRt = \Psit \tilde{R}$ (as in Lemma \ref{lemm:tilde_close}). Let $\psiRh_k$ be the $k$-th column of $\PsiRh$, $\psiRs_k$ be the $k$-th column of $\PsiRs$, and $\psiRt_k$ be the $k$-th column of $\PsiRt$ (as in Lemma \ref{lemm:tilde_close}). Then we have 
\begin{equation}
\label{eq:psi_close0}
\Norm{\psiRh_k - \psiRs_k} =  \oo_p \p{\frac{\p{\log n }^2}{ \sqrt{\rhon}}} \text{ and }
\Norm{\psiRs_k - \psiRt_k} =  \oo_p \p{ e^{-C n^{\kappa/2}} /\sqrt{\rhon}}.
\end{equation}
In particular, together with Lemma \ref{lemm:tilde_close}, the above implies that
\begin{equation}
\label{eq:psi_close}
\Norm{\psiRh_k - \psi_k} = \oo_p \p{\frac{\p{\log n }^2}{ \sqrt{\rhon}}}. 
\end{equation}

Note also that the construction of $\hat{O}$ in \eqref{eq:davis_kahan} ensures that $(\hat{V} \hat{O})^T \hat{V}$ is symmetric. Specifically, $(\hat{V} \hat{O})^T V = O_2 O_1^T O_1 D O_2^T = O_2 D O_2^T$. Thus $\Norm{\smash{\hat{V} \hat{O} - V}}^2_F = 2\operatorname{tr}(I - D)$. This observations implies that $(\PsiRh)^T \PsiRt$ is symmetric, and if we write $(\PsiRh)^T \PsiRt = n O_{\Psi} D_{\Psi} O_{\Psi}^T$, then $\operatorname{tr}(I - D_{\Psi}) =  \oo_p \p{\frac{\p{\log n }^4}{n \rhon}}$. 

We are now ready to move towards proving Lemma \ref{lemm:a_psi_close} itself.
We'll start by looking at $a^T (\psiRh_k - \psiRt_k)$. 
Without loss of generality, assume that $\Norm{a} = 1$. Let $a = \alpha_0 \psiRt_0 +  \alpha_1 \psiRt_1 + \alpha_1 \psiRt_2 + \dots + \alpha_r \psiRt_r$, where $\Norm{\smash{\psiRt_0}} = \sqrt{n}$ and $\psiRt_0$ is orthogonal to $\psiRt_k$ for any $k \in \cb{1,2 \dots, r}$. Thus $\psiRt_0$ is an eigenvector of $G$ with its corresponding eigenvalue 0. Then $1 = \Norm{a}^2 = n \p{ \alpha_0^2 +  \alpha_1^2 + \dots +  \alpha_r^2}$. We'll study $(\psiRt_l)^T(\psiRh_k - \psiRt_k)$ now. 

For $l \neq 0$, we will show that $\Norm{(\PsiRt)^T(\PsiRh - \PsiRt)}_F$ is small. Note that $(\PsiRt)^T(\PsiRh - \PsiRt) = (\PsiRt)^T\PsiRh - nI = n (O_{\Psi} D_{\Psi} O_{\Psi}^T - I)$. Thus
\begin{align*}
\Norm{(\PsiRt)^T(\PsiRh - \PsiRt) }^2_F
= n^2 \operatorname{tr} \p{\p{(O_{\Psi} D_{\Psi} O_{\Psi}^T - I}^2}
= n^2 \operatorname{tr} \p{(D_{\Psi}  - I)^2}. 
\end{align*}
But we know $\operatorname{tr}(I - D_{\Psi}) =  \oo_p \p{\frac{\p{\log n }^4}{n \rhon}}$, thus $ \operatorname{tr} \p{(D_{\Psi}  - I)^2} = \oo_p \p{\frac{\p{\log n }^8}{n^2 \rhon^2}}$. Therefore,
\[\Norm{(\PsiRt)^T(\PsiRh - \PsiRt) }_F = \oo_p \p{\frac{\p{\log n }^4}{ \rhon}}.\]
The result can also be written in the vector form:
\[(\psiRt_l)^T\p{\psiRh_k - \psiRt_k} = \oo_p \p{\frac{\p{\log n }^4}{ \rhon}},\]
for any $k ,l \in \cb{1, \dots, r}$. 

For $l = 0$, we have $(\psiRt_0)^T \p{\PsiRh - \PsiRt} = (\psiRt_0)^T \PsiRh$. Note that
\begin{equation}
\label{eqn:key_E_G}
\begin{split}
(\psiRt_0)^T (E - \rhon G) \PsiRh &= (\psiRt_0)^T E \PsiRh - (\psiRt_0)^T \rhon G \PsiRh = (\psiRt_0)^T E \PsiRh\\
 & = (\psiRt_0)^T E \Psih \hat{R} 
 = (\psiRt_0)^T  \Psih \hat{\Lambda} \hat{R} ,
\end{split}
\end{equation}
where $\hat{\Lambda} $ is the $r \times r$ diagonal matrix with $\hat{\lambda}_1, \dots,  \hat{\lambda}_k$ on its diagonal. The left hand side of \eqref{eqn:key_E_G} can be decomposed into 
\[ (\psiRt_0)^T (E - \rhon G) \PsiRh = (\psiRt_0)^T (E - \rhon G) \p{\PsiRh - \PsiRt } + (\psiRt_0)^T (E - \rhon G) \PsiRt. \]
The first term can be easily bounded by
\[ \Norm{(\psiRt_0)^T (E - \rhon G) \p{\PsiRh - \PsiRt }} \leq \Norm{\psiRt_0} \Norm{E - \rhon G}_{op} \Norm{\PsiRh - \PsiRt} = \oo_p\p{n \p{\log n}^4}, \]
where the last inequality follows from Lemma \ref{lemm:E_op_norm} and \eqref{eq:psi_close00}. For the second term, consider $(\psiRt_0)^T (E - \rhon G) \psiRt_k$. It can be bounded by $\abs{\smash{(\psiRt_0)^T (E - \rhon G) \psiRt_k}} \leq \abs{\smash{(\psiRt_0)^T (E - G_n) \psiRt_k}} + \abs{\smash{(\psiRt_0)^T (G_n - \rhon G) \psiRt_k}}$. Note that $ \abs{\smash{(\psiRt_0)^T (G_n - \rhon G) \psiRt_k}} = C n^4 e^{- C \p{\log n}^2}$ by Lemma \ref{lemm:G_Gn_bound}. 
We then rewrite the term $(\psiRt_0)^T (E - G_n) \psiRt_k$ in a different form: $(\psiRt_0)^T (E - G_n) \psiRt_k = \sum_{(i,j), i \neq j } \psiRt_{0i}\psiRt_{kj} (E_{ij} - \Gfcn(U_i,U_j))$. As $E_{ij}$'s are independent given $U$, we have
\begin{align*}
& \quad \quad \EE{\p{\sum_{(i,j), i \neq j } \psiRt_{0i}\psiRt_{kj} (E_{ij} -  G_n(U_i,U_j))}^2}\\
&\leq 2 \sum_{(i,j), i \neq j }  \EE{(\psiRt_{0i})^2(\psiRt_{kj})^2 (E_{ij} - G_n(U_i,U_j))}
 = 2 \sum_{(i,j), i \neq j }  \EE{(\psiRt_{0i})^2(\psiRt_{ki})^2 \Gfcn(U_i,U_j)}\\
& \leq 2 \rhon (\log n)^4 \sum_{(i,j), i \neq j }  \EE{(\psiRt_{0i})^2(\psiRt_{kj})^2} + 2n^4 \PP{G_n(U_1,U_2) \geq \rhon (\log n)^4}\\
& \leq 2 \rhon n^2 (\log n)^4  + C n^4 e^{- C \p{\log n}^2} \leq C \rhon (\log n)^4, 
\end{align*}
where the last inequality comes from Lemma \ref{lemm:G_Gn_bound}. Combining the two terms, we get that the left hand side of \eqref{eqn:key_E_G} satisfy
\[ \Norm{(\psiRt_0)^T (E - \rhon G) \PsiRh} = \oo_p\p{n \p{\log n}^4}.\]
Therefore the right hand side of \eqref{eqn:key_E_G} satisfy $\Norm{(\psiRt_0)^T  \Psih \hat{\Lambda} \hat{R}} = \oo_p\p{n \p{\log n}^4}$. But Lemma \ref{lemm:eigen_value} shows that $\lambda_k = \Omega_p\p{n \rhon}$ for $k \leq r$. Thus
\[\Norm{(\psiRt_0)^T  \PsiRh} =  \Norm{(\psiRt_0)^T  \Psih R} = \Norm{(\psiRt_0)^T  \Psih \hat{\Lambda} \hat{R} \hat{R}^T  \hat{\Lambda}^{-1}  \hat{R}} = \oo_p\p{\p{\log n}^4/ \rhon}. \]

Back to the vector $a$, 
\[ a^T \p{\psiRh_k  - \psiRt_k} = \sum_{k = 0}^r \alpha_k  \psit_k^T \p{\psiRh_k  - \psiRt_k} = \oo_p\p{ \p{\log n}^4 /\p{\sqrt{n} \rhon} }, \]
as $n \alpha_i^2 \leq 1$. 
Note also that $\abs{a^T \p{\psiRt_k  - \psi_k}} \leq \Norm{a} \Norm{\psiRt_k  - \psi_k} = \oo_p\p{1}$ by Lemma \ref{lemm:tilde_close}. Hence
\begin{align*}
\abs{a^T \p{\psiRh_k  - \psi_k}} &
\leq  \abs{a^T \p{\psiRh_k  - \psiRt_k}} + \abs{a^T \p{\psiRt_k  - \psi_k}}\\
& = \oo_p\p{\p{\log n}^4 /\p{\sqrt{n} \rhon} + 1}
= \oo_p\p{1}. 
\end{align*}

\subsection{Consequences of Lemma \ref{lemm:a_psi_close}}
\begin{lemm}
\label{lemm:E_psi_close}
Under assumptions \ref{assu:undirected}, \ref{assu:random_graph} and \ref{assu:graphon}, assume (\ref{eqn:rate_of_rhon}), (\ref{eqn:rankr_assumption}) and (\ref{eqn:satisfy_bernstein}). If we define $\psiRh$ as in Lemma \ref{lemm:a_psi_close}, then
\[ \Norm{E \p{\psiRh_k - \psi_k}} = \oo_p \p{n \rhon}.  \]
\end{lemm}
\begin{proof}

We start by decomposing the target expression. 
\begin{align*}
\Norm{E \p{\psih_k - \psi_k}} 
& \leq \Norm{\rhon \Gm \p{\psiRh_k - \psi_k}} + \Norm{\p{E - \rhon \Gm} \p{\psiRh_k - \psi_k}} \\
& \leq \frac{\rhon}{n} \Norm{\sum_{k=1}^r \lambdat_k \psit_k \psit_k^{T} \p{\psiRh_k - \psi_k}} + \Norm{E - \rhon \Gm}_{op} \Norm{\psiRh_k - \psi_k} \\
&= \oo_p\p{\sqrt{n}\rhon \sqrt{n}} + \oo_p\p{ \p{\log n}^2 \sqrt{n \rhon} \frac{\p{\log n}^2}{\sqrt{\rhon}} }\\
& =  \oo_p\p{n \rhon},
\end{align*}
where the $\oo_p\p{}$ terms come from Lemma \ref{lemm:a_psi_close}, Lemma \ref{lemm:E_op_norm} and \eqref{eq:psi_close}. 
\end{proof}

\begin{lemm}
\label{lemm:beta_hat_k}
Under assumptions \ref{assu:undirected}, \ref{assu:random_graph} and \ref{assu:graphon}, assume (\ref{eqn:rate_of_rhon}), (\ref{eqn:rankr_assumption}) and (\ref{eqn:satisfy_bernstein}). Define $\hat{R}$ and $\psiRh$ as in Lemma \ref{lemm:a_psi_close}. Let $\hbetaR = \hat{R}^T \hbeta$, then
\[\sum_{k=1}^r  \hbeta_k \psih_{k i} = \sum_{k=1}^r  \hbetaR_k \psiRh_{k i}, \quad   \hbetaR_k =\oo_p\p{\sqrt{n} \rhon}\text{ and } \hbeta_k =\oo_p\p{\sqrt{n} \rhon} .\]
\end{lemm}
\begin{proof}
The first statement follows directly from the fact that $\PsiRh \hbetaR = \Psih \hat{R} \hat{R}^T \hbeta = \Psih \hbeta$. To show the remaining two, 
recall that $\hbeta$ is constructed so that the following equations hold
\[\sum_i \psih_{l i} \p{\frac{M_i}{\pi} - \frac{N_i - M_i}{1 - \pi} + \sum_{k=1}^r  \hbeta_k \psih_{k i} } =0, \text{ for }l = 1, \dots, r. \]
As $\psih$'s are the eigenvectors of $E$ and are hence orthogonal to each other, we can easily solve the above and get
\[\hbeta_k = - \frac{1}{n} \sum_i \psih_{ki}\p{\frac{M_i}{\pi} - \frac{N_i - M_i}{1 - \pi}}.  \]
Writing the above in matrix form and multiply $\hat{R}$ on both hand side, one can obtain
\[\hbetaR_k = - \frac{1}{n} \sum_i \psiRh_{ki}\p{\frac{M_i}{\pi} - \frac{N_i - M_i}{1 - \pi}}.  \]
We then break the above into two terms 
\begin{align*}
\hbeta_k = - \frac{1}{n} \sum_i \psi_{k}(U_i)\p{\frac{M_i}{\pi} - \frac{N_i - M_i}{1 - \pi}}  - \frac{1}{n} \sum_i \p{\psiRh_{ki}  - \psi_{k} (U_i)}\p{\frac{M_i}{\pi} - \frac{N_i - M_i}{1 - \pi}},
\end{align*}
and will analyze the two terms one by one. 

For the first term,
\begin{align*}
	\sum_i \psi_{k}(U_i)\p{\frac{M_i}{\pi} - \frac{N_i - M_i}{1 - \pi}}
=  \frac{1}{\pi(1-\pi)} \sum_{(i,j), i \neq j} \psi_k(U_i)(W_j - \pi) E_{ij}
\end{align*}
We'll show that the above is close to  $\frac{1}{\pi(1-\pi)} \sum_{(i,j), i \neq j} \psi_k(U_i)(W_j - \pi) \Gfcn(U_i, U_j)$. Applying Lemma \ref{lemma:cross_vanish} to the difference, we get $\Var{\sum_{(i,j), i \neq j} \psi_k(U_i)(W_j - \pi) \p{E_{ij} - \Gfcn(U_i, U_j)}} \leq 2 \sum_{(i,j), i \neq j} \EE{\psi_k(U_i)^2(W_j - \pi)^2 \p{E_{ij} - \Gfcn(U_i, U_j)}^2} \leq 2 \sum_{(i,j), i \neq j} \EE{\psi_k(U_i)^2 \Gfcn(U_i, U_j)} \leq 2 \sum_{(i,j), i \neq j} \sqrt{\EE{\psi_k(U_i)^4} \EE{\Gfcn(U_i, U_j)^2}} \leq C n^2  \rhon$, where the last inequality comes from \eqref{eqn:satisfy_bernstein} and Lemma \ref{lemma:Eij_expectaions}. This implies that
\[ \sum_{(i,j), i \neq j} \psi_k(U_i)(W_j - \pi) E_{ij} =\sum_{(i,j), i \neq j} \psi_k(U_i)(W_j - \pi) \Gfcn(U_i, U_j) + \oo_p\p{n\sqrt{\rhon}}.\]
For the term $\sum_{(i,j), i \neq j} \psi_k(U_i)(W_j - \pi) \Gfcn(U_i, U_j)$, note that $\Var{\psi_k(U_i)(W_j - \pi) \Gfcn(U_i, U_j)} \leq \EE{\psi_k(U_i)^2 \Gfcn(U_i, U_j)^2} \leq \sqrt{\EE{\psi_k(U_i)^4} \EE{\Gfcn(U_i, U_j)^2}}  \leq C \rhon^2$ again by \eqref{eqn:satisfy_bernstein} and Lemma \ref{lemma:Eij_expectaions}. Then by Lemma \ref{lemma:summation_decompose} part 3, the variance of $\sum_{(i,j), i \neq j} \psi_k(U_i)(W_j - \pi) \Gfcn(U_i, U_j)$ will be upper bounded by $C n^3 \rhon^2$. This further implies that 
\[\sum_{(i,j), i \neq j} \psi_k(U_i)(W_j - \pi) \Gfcn(U_i, U_j) = \oo_p\p{n^{\frac{3}{2}} \rhon}.\] Therefore
\begin{equation}
\label{eqn:beta_hat_first_term}
	\sum_i \psi_{k}(U_i)\p{\frac{M_i}{\pi} - \frac{N_i - M_i}{1 - \pi}} =  \oo_p\p{n^{\frac{3}{2}} \rhon + n\sqrt{\rhon}} =  \oo_p\p{n^{\frac{3}{2}} \rhon}.  
\end{equation}
 
For the second term, combining Lemma \ref{lemm:M_pi_N_ai} and \ref{lemm:E_psi_close}, we get
\begin{equation}
\label{eqn:beta_hat_second_term}
\sum_i \p{\psiRh_{ki}  - \psi_{k} (U_i)}\p{\frac{M_i}{\pi} - \frac{N_i - M_i}{1 - \pi}} = \oo_p \p{n \rhon}.   
\end{equation}
Combining the two terms, we get
\[\hbetaR_k = \oo_p\p{\p{n^{\frac{3}{2}} \rhon + n \rhon}/n} =  \oo_p\p{\sqrt{n} \rhon}.\]
And $\hbeta_k =\oo_p\p{\sqrt{n} \rhon}$ follows as a direct corollary. 
\end{proof}

\subsection{Proof of Proposition \ref{prop:term_2}}
\label{subsection: proof_term2}
We are interested in the term 
\[ \frac{1}{n} \sum_i f_i'(W_i, \pi)\p{\frac{M_i}{N_i} - \pi} \p{ \frac{1}{\pi(1-\pi)}\p{M_i - \pi N_i}}. \]
We can write 
\[ f_i'(W_i, \pi) = (W_i - \pi)(f_i'(1, \pi)-f_i'(0, \pi)) + \Bu(U_i) + D_i, \]
where $\Bu(U_i) = \EE{ \pi f_i'(1, \pi)  + (1-\pi) f_i'(0, \pi)|U_i }$, and $D_i = \pi f_i'(1, \pi)  + (1-\pi) f_i'(0, \pi) - \Bu(U_i)$. Hence we have $\EE{D_i | U_i} = 0$. 
Define 
\begin{align*}
S_{a1} &= \frac{1}{n} \sum_i (W_i - \pi)(f_i'(1, \pi)-f_i'(0, \pi)) \p{\frac{M_i}{N_i} - \pi} \p{ \frac{1}{\pi(1-\pi)}\p{M_i - \pi N_i}}, \\
S_{a2} &= \frac{1}{n} \sum_i D_i \p{\frac{M_i}{N_i} - \pi} \p{ \frac{1}{\pi(1-\pi)}\p{M_i - \pi N_i}},\\
S_{a3} &= \frac{1}{n} \sum_i \Bu(U_i) \p{\frac{M_i}{N_i} - \pi} \p{ \frac{1}{\pi(1-\pi)}\p{M_i - \pi N_i} }.
\end{align*}
We'll analyze them one by one.

\paragraph{$\mathbf{S_{a1}}$}

For $S_{a1}$,
\[S_{a1} = \frac{1}{n \pi(1-\pi)} \sum_i (W_i - \pi) (f_i'(1,\pi) - f_i'(0,\pi)) \frac{(M_i - \pi N_i)^2}{N_i}.\]
Consider $\sum_i (W_i - \pi) (f_i'(1,\pi) - f_i'(0,\pi)) \frac{(M_i - \pi N_i)^2}{N_i}$. For $i \neq j$, we'll show that the term $(W_i - \pi) (f_i'(1,\pi) - f_i'(0,\pi)) \frac{(M_i - \pi N_i)^2}{N_i}$ is roughly uncorrelated with the term $(W_j - \pi) (f_j'(1,\pi) - f_j'(0,\pi)) \frac{(M_j - \pi N_j)^2}{N_j}$. For $i,j$ distinct,
\begin{align*}
& \quad \quad \EE{(W_i - \pi) (f_i'(1,\pi) - f_i'(0,\pi)) \frac{(M_i - \pi N_i)^2}{N_i}   (W_j - \pi) (f_j'(1,\pi) - f_j'(0,\pi)) \frac{(M_j - \pi N_i)^2}{N_j}|f(\cdot),E}\\
& = \frac{(f_i'(1,\pi) - f_i'(0,\pi))(f_j'(1,\pi) - f_j'(0,\pi))}{N_i N_j} \times\\
& \quad \quad \EE{(W_i - \pi)(W_j - \pi) \p{\sum_{k \neq i}E_{ik}(W_k - \pi)}^2 \p{\sum_{k \neq j}E_{jk}(W_k - \pi)}^2 \Bigg|f(\cdot), E}\\
&\\
& = \frac{4(f_i'(1,\pi) - f_i'(0,\pi))(f_j'(1,\pi) - f_j'(0,\pi))}{N_i N_j}  \EE{(W_i - \pi)^2(W_j - \pi)^2(W_j - \pi)^2 E_{ij}^2 \sum_{k \neq i ,j}E_{ik}E_{ij} \Bigg|f(\cdot), E}\\
& = 4(\pi(1-\pi))^3 \frac{(f_i'(1,\pi) - f_i'(0,\pi))(f_j'(1,\pi) - f_j'(0,\pi))}{N_i N_j} \p{E_{ij}^2 \sum_{k \neq i ,j}E_{ik}E_{ij} }
\end{align*}
Hence $\EE{(W_i - \pi) (f_i'(1,\pi) - f_i'(0,\pi)) \frac{(M_i - \pi N_i)^2}{N_i}   (W_j - \pi) (f_j'(1,\pi) - f_j'(0,\pi)) \frac{(M_j - \pi N_i)^2}{N_j}} \leq \frac{C B^2 \rhon }{n}$.  Note also that $\EE{\p{  (W_i - \pi) (f_i'(1,\pi) - f_i'(0,\pi)) \frac{(M_i - \pi N_i)^2}{N_i} }^2} \leq C B^2$. Hence 
\[\EE{\p{\sum_i (W_i - \pi) (f_i'(1,\pi) - f_i'(0,\pi)) \frac{(M_i - \pi N_i)^2}{N_i}}^2} \leq CB^2 n \rhon.\]
Therefore $S_{a1} = \oo_p\p{\frac{\sqrt{\rhon}}{\sqrt{n}}} $. 

\paragraph{$\mathbf{S_{a2}}$}
For $S_{a2}$, note that $A_i = \p{\frac{M_i}{N_i} - \pi} \p{M_i - \pi N_i} = \frac{(M_i - \pi N_i)^2}{N_i}$ is independent of $D_i$ given $U_i$. Therefore for $i,j$ distinct, $\EE{A_i D_i A_j D_j} = \EE{ \EE{D_i | U_i} \EE {D_j|U_j} \EE{A_i A_j|U_i U_j) }} = 0$. Hence $\EE{\p{\sum_i D_i A_i}^2} \leq C n B^2$. $S_{a2} = \frac{1}{n\pi(1 - \pi)}\sum_i D_i A_i = \oo_p\p{\frac{B}{\sqrt{n}}}$. 

\paragraph{$\mathbf{S_{a3}}$}
For $S_{a3}$,
\begin{align*}
S_{a3} &= \frac{1}{n \pi(1-\pi)} \sum_i \Bu(U_i)\p{\frac{M_i}{N_i} - \pi}(M_i  - \pi N_i)  = \frac{1}{n \pi(1-\pi)} \sum_i \frac{\Bu(U_i)}{N_i}(M_i  - \pi N_i)^2\\
& = \frac{1}{n \pi(1-\pi)} \sum_i \frac{\Bu(U_i)}{N_i}\p{\sum_{j \neq i} E_{ij}(W_j - \pi)}^2\\
& = \frac{1}{n \pi(1-\pi)} \sum_i \frac{\Bu(U_i)}{N_i} \sum_{j \neq i} E_{ij}(W_j - \pi)^2 + \frac{1}{n} \sum_{\substack{i,j,k\\\text{distinct}}} \frac{E_{ij} E_{ik} \Bu(U_i)}{N_i} (W_j - \pi)(W_k - \pi)\\
& = S_{a31} + S_{a32}.
\end{align*}

For $S_{a31}$,
$S_{a31} = \frac{1}{n \pi(1-\pi)} \sum_i \frac{\Bu(U_i)}{N_i} \sum_{j \neq i} E_{ij}(W_j - \pi)^2$. We'll show that it's close to $\frac{1}{n \pi(1-\pi)} \sum_i \frac{\Bu(U_i)}{N_i} \sum_{j \neq i} E_{ij} \pi(1-\pi)$. Consider $ \sum_i \frac{\Bu(U_i)}{N_i} \sum_{j \neq i} E_{ij}\sqb{(W_j - \pi)^2 - \pi(1-\pi)}$, for $i_1, i_2$ distinct,
\begin{align*}
& \quad \quad \EE{  \frac{\Bu(U_{i_1})}{N_{i_1}} \p{\sum_{j \neq i_1} E_{i_1j}\sqb{(W_j - \pi)^2 - \pi(1-\pi)}} \frac{\Bu(U_{i_2})}{N_{i_2}} \p{\sum_{j \neq i_2} E_{i_2j}\sqb{(W_j - \pi)^2 - \pi(1-\pi)}} \Big| E, U}\\
&= \frac{\Bu(U_{i_1})\Bu( U_{i_2})}{N_{i_1}N_{i_2}} \sum_{ j \neq i_1, i_2} E_{i_1j} E_{i_2j} \EE{ \p{(W_k - \pi)^2 - \pi(1-\pi)}^2 }.
\end{align*}
Hence the unconditional expectation can be bounded by $\frac{CB}{n}$. With the cross terms small, we therefore have $E\sqb{ \p{ \sum_i \frac{\Bu(U_i)}{N_i} \sum_{j \neq i} E_{ij}\sqb{(W_j - \pi)^2 - \pi(1-\pi)}}^2} \leq C n B^2$. Hence 
\begin{align*}
S_{a31} &= \frac{1}{n\pi(1-\pi)} \sum_i \frac{\Bu(U_i)}{N_i} \sum_{j \neq i}E_{ij} \pi (1-\pi) + \oo_p\p{\frac{B}{\sqrt{n}}} \\
&= \frac{1}{n} \sum_i \Bu(U_i) + \oo_p\p{\frac{B}{\sqrt{n}}} = \EE{\Bu(U_i)} + \oo_p\p{\frac{B}{\sqrt{n}}}, 
\end{align*}
by law of large numbers. 

Recall $\Bu(U_i) =\EE{ \pi f_i'(1, \pi)  + (1-\pi) f_i'(0, \pi)|U_i }$. Note also that by Proposition \ref{prop:estimands}, $\tau_{\IND} = \EE{\Bu(U_i)} + \oo\p{\frac{B}{\sqrt{n\rhon}}}= \EE{ \pi f_i'(1, \pi)  + (1-\pi) f_i'(0, \pi)} + \oo\p{\frac{B}{\sqrt{n\rhon}}}$, hence
\[S_{a31} = \tau_{\IND} +  \oo\p{\frac{B}{\sqrt{n\rhon}}}. \]

For $S_{a32}$,
\begin{align*}
S_{a32} = \frac{1}{n}\sum_{\substack{(j,k) \\ j \neq k}} (W_j - \pi) (W_k - \pi) \p{\sum_{i \neq j,k} \frac{E_{ij}E_{ik} \Bu(U_i) }{N_i}}.  
\end{align*}
Let $\Au(U_k, U_j) = n \EE{\frac{E_{ij}E_{ik}\Bu(U_i)}{N_i} \Big| U_j, U_k}$. We will show that $S_{a32}$ can be approximated by $\frac{1}{n}\sum_{\substack{(j,k) \\ j \neq k}} (W_j - \pi) (W_k - \pi)  \Au(U_k, U_j)$. For $i_1, i_2, j, k$ all different, we will show that conditioning on $U_j$ and $U_k$, $\frac{ E_{i_1j} E_{i_1k}\Bu(U_{i_1})}{N_{i_1}}$ and $\frac{ E_{i_2j} E_{i_2k}\Bu(U_{i_2})}{N_{i_2}}$ will be roughly uncorrelated. Specifically, 
\[\frac{E_{i_1j} E_{i_1k}\Bu(U_{i_1})}{N_{i_1}} = \frac{E_{i_1j} E_{i_1k}\Bu(U_{i_1})}{E_{i_1 i_2} + (N_{i_1} - E_{i_1 i_2})}. \]
We can apply Lemma \ref{lemma:cov_ineq} to the case and get $
\EE{\Cov{\frac{E_{i_1j} E_{i_1k}\Bu(U_{i_1})}{N_{i_1}},\frac{E_{i_2j} E_{i_2k}\Bu(U_{i_2})}{N_{i_2}} \bigg| U_j,U_k  }} \leq  \frac{C B^2}{n^3 \rhon}$. Therefore we have
\[\EE{\p{\frac{E_{i_1j} E_{i_1k}\Bu(U_{i_1})}{N_{i_1}} - \frac{\Au(U_j, U_k)}{n}}\p{\frac{E_{i_2j} E_{i_2k}\Bu(U_{i_2})}{N_{i_2}} - \frac{\Au(U_j, U_k)}{n}}} \leq \frac{C B^2}{n^3 \rhon}.\]
Hence the term 
\[\frac{1}{n}\sum_{\substack{(i,j,k) \\ \text{all distinct}}} \p{\frac{E_{ij}E_{ik}\Bu(U_i)}{N_i} - \frac{\Au(U_j, U_k)}{n}}(W_j - \pi)(W_k - \pi)\]
has its second moment being
\begin{align*}
&\quad\quad\frac{2}{n^2}\sum_{\substack{(i,j,k) \\\text{all distinct}}} \EE{\frac{E_{ij}E_{ik}\Bu(U_i)}{N_i} - \frac{\Au(U_j, U_k)}{n}}^2(W_j - \pi)^2(W_k - \pi)^2\\
& \quad + \frac{2}{n^2}\sum_{\substack{(i_1,i_2,j,k) \\ \text{all distinct}}} (W_j - \pi)^2(W_k - \pi)^2 \times \\
&  \quad \quad \quad \quad \quad \quad \quad E\sqb{\p{\frac{E_{i_1j} E_{i_1k}\Bu(U_{i_1})}{N_{i_1}} - \frac{\Au(U_j, U_k)}{n}}\p{\frac{E_{i_2j} E_{i_2k}\Bu(U_{i_2})}{N_{i_2}} - \frac{\Au(U_j, U_k)}{n}}}  \\
&\\
& = \oo\p{\frac{B^2 \rhon^2}{n} + \frac{B^2 }{n \rhon }} = \oo\p{\frac{B^2}{n \rhon }} .
\end{align*}
Hence we have 
\[ S_{a32} = \frac{1}{n}\sum_{\substack{(j,k) \\ j \neq k}} (W_j - \pi) (W_k - \pi)  \Au(U_k, U_j) + \oo_p \p{\frac{B}{\sqrt{n \rhon}}}.  \]

We'll show that $\frac{1}{n}\sum_{\substack{(j,k) \\ j \neq k}} (W_j - \pi) (W_k - \pi)  \Au(U_k, U_j) = \oo_p(\rhon)$. Recall that $\Bu(U_i) = \EE{\pi f_i'(1,\pi) + (1-\pi)f_i'(0,\pi)| U_i}$, hence we have $\abs{\Bu(U_i)} \leq B$. Also recall that $\Au(U_k, U_j) = n \EE{\frac{E_{ij}E_{ik}\Bu(U_i)}{N_i} \Big| U_j, U_k}$. We'll bound $\abs{\Au(U_k, U_j)}$ first. Note that 
\begin{align*}
\abs{\Au(U_k, U_j)} & \leq nB  \EE{\frac{E_{ij}E_{ik}}{N_i} \Big| U_j, U_k}
\leq nB  \EE{\frac{E_{ij}E_{ik}}{N_i - E_{ij} - E_{ik}} 1_{\cb{N_i - E_{ij} - E_{ik} >0}} \Big| U_j, U_k}\\
& = nB \EE{ \EE{\frac{1_{\cb{N_i - E_{ij} - E_{ik} >0}}}{N_i - E_{ij} - E_{ik}} \Big| U_j, U_k, U_i}    \EE{E_{ij}E_{ik}  | U_j, U_k, U_i}    \Big| U_j, U_k}\\
& = nB \EE{ \EE{\frac{1_{\cb{N_i - E_{ij} - E_{ik} >0}}}{N_i - E_{ij} - E_{ik}} \Big| U_j, U_k, U_i}  \Gfcn(U_i, U_j) \Gfcn(U_i, U_k)  \Big| U_j, U_k}\\
& \leq nB \frac{C}{n\rhon \clower} \EE{ \Gfcn(U_i, U_j) \Gfcn(U_i, U_k)  \Big| U_j, U_k}  \text{\quad  by Lemma \ref{lemma:boundMN}}\\
&= \frac{C B}{\rhon \clower} \Hfcn(U_j, U_k).
\end{align*}
Hence by lemma \ref{lemma:Eij_expectaions}, $\EE{\Au(U_k, U_j)^2} \leq \frac{C^2 B^2}{\rhon^2 \clower^2} \EE{\Hfcn(U_j, U_k)^2} \leq \frac{C^2 B^2\cupper^4}{\clower^2 \rhon^2} \rhon^4 = \frac{C^2 B^2\cupper^4}{\clower^2} \rhon^2 $. By Lemma \ref{lemma:cross_vanish}, 
\begin{align*}
\Var{\frac{1}{n}\sum_{\substack{(j,k) \\ j \neq k}} (W_j - \pi) (W_k - \pi)  \Au(U_k, U_j)} 
&\leq \frac{2}{n^2} \sum_{\substack{(j,k) \\ j \neq k}}  \Var{(W_j - \pi) (W_k - \pi)  \Au(U_k, U_j)}\\
&= \frac{2}{n^2} \sum_{\substack{(j,k) \\ j \neq k}}  \EE{(W_j - \pi)^2 (W_k - \pi)^2  \Au(U_k, U_j)^2}\\
& =  \frac{2\pi^2(1-\pi)^2}{n^2} n(n-1) \EE{ \Au(U_1, U_2)^2}\\
& = \frac{2C^2 B^2\cupper^4}{\clower^2} \rhon^2
\end{align*}
Hence $\frac{1}{n}\sum_{\substack{(j,k) \\ j \neq k}} (W_j - \pi) (W_k - \pi)  \Au(U_k, U_j) = \oo_p(\rhon)$. Hence
\[S_{a32} = \oo_p\p{\rhon}.\]

\subsection{Proof of Proposition \ref{prop:term_3}}
We are interested in the term 
\[ \sum_i \hgamma_i f'_i(W_i, \pi) \p{\frac{M_i}{N_i} - \pi} = \frac{1}{n} \sum_i f_i'(W_i, \pi)\p{\frac{M_i}{N_i} - \pi} \p{ \sum_{k=1}^r \hbeta_k \psih_{k i}}. \]
Again as in the proof of Section \ref{subsection: proof_term2}, we can write 
\[ f_i'(W_i, \pi) = (W_i - \pi)(f_i'(1, \pi)-f_i'(0, \pi)) + \Bu(U_i) + D_i, \]
where $\Bu(U_i) = \EE{ \pi f_i'(1, \pi)  + (1-\pi) f_i'(0, \pi)|U_i }$, and $D_i = \pi f_i'(1, \pi)  + (1-\pi) f_i'(0, \pi) - \Bu(U_i)$. Hence we have $\EE{D_i | U_i} = 0$. 
Define 
\begin{align*}
S_{b1} &= \frac{1}{n} \sum_i (W_i - \pi)(f_i'(1, \pi)-f_i'(0, \pi)) \p{\frac{M_i}{N_i} - \pi} \p{ \sum_{k=1}^r \hbeta_k \psih_{k i}}, \\
S_{b2} &= \frac{1}{n} \sum_i D_i \p{\frac{M_i}{N_i} - \pi} \p{ \sum_{k=1}^r \hbeta_k \psih_{k i}},\\
S_{b3} &= \frac{1}{n} \sum_i \Bu(U_i) \p{ \sum_{k=1}^r \hbeta_k \psih_{k i}}.
\end{align*}
We'll analyze them one by one.

\paragraph{$\mathbf{S_{b1}}$}
For $S_{b1}$, 
\[S_{b1} = \sum_k\frac{\hbeta_k}{n} \sum_i (W_i - \pi) (f_i'(1,\pi) - f_i'(0,\pi))\p{ \frac {M_i}{N_i} - \pi}\psih_{ki}. \]
We'll analyze $\sum_i (W_i - \pi) (f_i'(1,\pi) - f_i'(0,\pi))\p{ \frac {M_i}{N_i} - \pi}\psih_{ki}$ for a fixed $k$. Again we'll show that $(W_i - \pi) (f_i'(1,\pi) - f_i'(0,\pi))\p{ \frac {M_i}{N_i} - \pi}\psih_{ki}$ and $(W_j - \pi) (f_j'(1,\pi) - f_j'(0,\pi))\p{ \frac {M_j}{N_j} - \pi}\psih_{kj}$ will be roughly uncorrelated for $i \neq j$. 
For simplicity of notation, define $\delta_i = f_i'(1,\pi) - f_i'(0,\pi)$. 
For $i,j$ distinct,
\begin{align*}
& \quad \quad \EE{ (W_i - \pi) \delta_i\p{ \frac {M_i}{N_i} - \pi}\psih_{ki} (W_j - \pi) \delta_j\p{ \frac {M_j}{N_j} - \pi}\psih_{kj}\Bigg| f(\cdot), E}\\
& = \delta_i\delta_j \psih_{ki} \psih_{kj} \frac{1}{N_i N_j} \EE{ (W_i - \pi) (W_j - \pi)(M_i - \pi N_i)(M_j - \pi N_j)\Bigg| f(\cdot), E}\\
& = \delta_i\delta_j \psih_{ki} \psih_{kj} \frac{1}{N_i N_j} \EE{ (W_i - \pi)^2 (W_j - \pi)^2 E_{ij}\Bigg| f(\cdot), E}\\
& = (\pi(1-\pi))^2 (f_i'(1,\pi) - f_i'(0,\pi))(f_j'(1,\pi) - f_j'(0,\pi)) \psih_{ki} \psih_{kj} \frac{E_{ij} }{N_i N_j}. 
\end{align*}
The above result implies that 
\begin{align*}
& \quad \quad \sum_{(i,j), i \neq j} \mathbb{E} \Big[ (W_i - \pi) \delta_i\p{ \frac {M_i}{N_i} - \pi} \psih_{ki} (W_j - \pi) \delta_j\p{ \frac {M_j}{N_j} - \pi}\psih_{kj} 
\Big]\\
& \leq C B^2 \sum_{(i,j), i \neq j} \EE{ \abs{\psih_{ki} \psih_{kj}} \frac{E_{ij} }{N_i N_j}}
\leq C B^2 \sum_{(i,j), i \neq j} \EE{  \frac{\psih_{ki}^2}{N_i^2}}
\leq C B^2/\rhon^2,
\end{align*}
where the last inequality comes from Lemma \ref{lemma:boundMN}.
Note that Lemma \ref{lemma:boundMN} also implies that $\sum_i \EE{\p{ (W_i - \pi) \delta_i\p{ \frac {M_i}{N_i} - \pi} \psih_{ki}  }^2} \leq C B^2/ \rhon$. Hence
\[ \EE{\p{\sum_i(W_i - \pi) \p{f_j'(1,\pi) - f_j'(0,\pi)}\p{ \frac {M_i}{N_i} - \pi} \psih_{ki} }^2} \leq C B^2/ \rhon .\]
This implies that 
\[\sum_i(W_i - \pi) \p{f_j'(1,\pi) - f_j'(0,\pi)}\p{ \frac {M_i}{N_i} - \pi} \psih_{ki} = \oo_p\p{1/ \rhon}. \]
As $\hbeta_k = \oo_p\p{\sqrt{n}\rhon}$ by Lemma \ref{lemm:beta_hat_k}, 
\[S_{b1} = \oo_p\p{\frac{1}{\sqrt{n}}}.\] 

\paragraph{$\mathbf{S_{b2}}$}
For $S_{b2}$, it follows from the same logic as for $S_{a2}$ in \ref{subsection: proof_term2}. Note that $ \p{\frac{M_i}{N_i} - \pi} \psih_{k i}$ is independent of $D_i$ given $U_i$. Therefore for $i,j$ distinct, $\EE{ \p{\frac{M_i}{N_i} - \pi} \psih_{k i} D_i \p{\frac{M_j}{N_j} - \pi} \psih_{k j} D_j} = \EE{ \EE{D_i | U_i} \EE {D_j|U_j} \EE{ \p{\frac{M_i}{N_i} - \pi} \psih_{k i} \p{\frac{M_j}{N_j} - \pi} \psih_{k j}|U_i U_j) }} = 0$. This further implies that $\EE{\p{\sum_i D_i  \p{\frac{M_i}{N_i} - \pi} \psih_{k i}}^2}$
$ = \sum_i \EE{\p{ D_i  \p{\frac{M_i}{N_i} - \pi} \psih_{k i}}^2}$
 $\leq CB/\rhon$, where the inequality follows from Lemma \ref{lemma:boundMN}. As $\hbeta_k = \oo_p\p{\sqrt{n}\rhon}$ by Lemma \ref{lemm:beta_hat_k}, 
\[S_{b2} = \sum_k\frac{\hbeta_k}{n} \sum_i  D_i\p{\frac{M_i}{N_i} - \pi} \psih_{k i}= \oo_p\p{1/\sqrt{n}}.\]

\paragraph{$\mathbf{S_{b3}}$}
For $S_{b3}$,
\[ S_{b3} =  \sum_{k=1}^r \frac{\hbeta_k}{n} \sum_i \Bu(U_i) \p{\frac{M_i}{N_i} - \pi}  \psih_{k i}
= \sum_{k=1}^r \frac{\hbeta_k}{n} \sum_i \Bu(U_i) \frac{\psih_{k i}}{N_i} \p{M_i - \pi N_i}.  
 \]
For each $k$, consider $\sum_i \Bu(U_i) \frac{\psih_{k i}}{N_i} \p{M_i - \pi N_i}$. By Lemma \ref{lemm:M_pi_N_ai}, 
\[\sum_i \Bu(U_i) \frac{\psih_{k i}}{N_i} \p{M_i - \pi N_i} = \oo_p\p{\Norm{Ea}}, \]
where $a_i = \Bu(U_i) \frac{\psih_{k i}}{N_i}$. Note that $\Norm{Ea} \leq \Norm{E}_{op} \Norm{a}$. By Lemma \ref{lemm:E_op_norm}, we know $\Norm{E}_{op} = \oo_p\p{n\rhon}$. We also have that $\EE{\Norm{a}^2} = \sum_i \EE{\Bu(U_i)^2 \frac{\psih_{k i}^2}{N_i^2}} \leq \sum_i B^2 \EE{\frac{\psih_{k i}^2}{N_i^2}} \leq \frac{C}{n \rhon^2}$, where the inequality comes from Lemma \ref{lemma:boundMN}. Combining the two, we get $\Norm{Ea} \leq \sqrt{n}$. This further implies that 
\[ \sum_i \Bu(U_i) \frac{\psih_{k i}}{N_i} \p{M_i - \pi N_i} = \oo_p\p{\sqrt{n}}. \]
As $\hbeta_k = \oo_p\p{\sqrt{n}\rhon}$ by Lemma \ref{lemm:beta_hat_k}, we have 
\[S_{b3} =  \sum_{k=1}^r \frac{\hbeta_k}{n} \sum_i \Bu(U_i) \frac{\psih_{k i}}{N_i} \p{M_i - \pi N_i} = \oo_p\p{\frac{\sqrt{n}\rhon}{n}\sqrt{n}} = \oo_p\p{\rhon}.\]

\subsection{Proof of Proposition \ref{prop:term_4}}
We'll analyze the second derivative term:
\[ \frac{1}{n} \sum_i f_i''(W_i, \pi^*)\p{\frac{M_i}{N_i} - \pi}^2 \p{ \frac{1}{\pi(1-\pi)}\p{M_i - \pi N_i} + \sum_{k=1}^r \hbeta_k \psih_{k i}}. \]
Note that
\begin{align*}
&\EE{\abs{\frac{1}{n} \sum_i f_i''(W_i, \pi^*)\p{\frac{M_i}{N_i} - \pi}^2 \p{ \frac{1}{\pi(1-\pi)}\p{M_i - \pi N_i} + \hbeta_k \psih_{k i}}}}\\
& \leq \frac{B}{n} \sum_i \sqrt{\EE{\p{\frac{M_i}{N_i} - \pi}^4} \EE{\p{ \frac{1}{\pi(1-\pi)}\p{M_i - \pi N_i} + \hbeta_k \psih_{k i}}^2}} \\
& \leq \frac{CB}{n} n \frac{1}{n \rhon} \sqrt{n\rhon} = \frac{C B}{\sqrt{n\rhon}},
\end{align*}
where the inequality follows from lemma \ref{lemma:boundMN} and \ref{lemm:beta_hat_k}. Hence 
\[\frac{1}{n} \sum_i f_i''(W_i, \pi^*)\p{\frac{M_i}{N_i} - \pi}^2 \p{ \frac{1}{\pi(1-\pi)}\p{M_i - \pi N_i} + \betac N_i} = \oo_p\p{\frac{B}{\sqrt{n \rhon}}}.\] 

\subsection{Proof of Proposition \ref{prop:term_5}}
We will analyze $\frac{1}{n}\sum_{k = 1}^r \mu_k \sum_i \p{\psi_k(U_i) - \psiRh_{ki}}  \p{ \frac{1}{\pi(1-\pi)}\p{M_i - \pi N_i} + \sum_l \hbeta_l \psih_{l i}} $. 
By Lemma \ref{lemm:beta_hat_k}, this term equals
$\frac{1}{n}\sum_{k = 1}^r \mu_k \sum_i \p{\psi_k(U_i) - \psiRh_{ki}}  \p{ \frac{1}{\pi(1-\pi)}\p{M_i - \pi N_i} + \sum_l \hbetaR_l \psiRh_{l i}}. $
Let $S_{c1}$, $S_{c2}$ correspond to the two summations involving $\frac{1}{\pi(1-\pi)}\p{M_i - \pi N_i}$ and $\hbetaR N_i$ respectively. 

For $S_{c1}$,
\[S_{c1} =  \frac{1}{\pi(1-\pi) n } \sum_{k = 1}^r \mu_k \sum_i \p{\psi_k(U_i) - \psiRh_{ki}} \p{M_i - \pi N_i} \]
We've showed in the proof of Lemma \ref{lemm:beta_hat_k} (check Equation (\ref{eqn:beta_hat_second_term}) for more details) that 
\[\sum_i \p{\psi_k(U_i) - \psiRh_{ki}} \p{M_i - \pi N_i}  = \oo_p\p{n \rhon}. \]
Together with the fact that $\abs{\mu_k} \leq B$. This implies that
\[ S_{c1} = \oo_p\p{ \frac{n \rhon}{n} } = \oo_p\p{\rhon}.\]

For $S_{c2}$,
\begin{align*}
S_{c2} 
&= \frac{1}{n}\sum_{k = 1}^r \mu_k \sum_i \p{\psi_k(U_i) - \psiRh_{ki}}  \sum_{l = 1}^r \hbetaR_l \psiRh_{l i} 
 =\sum_{k = 1}^r  \sum_{l = 1}^r  \frac{\mu_k \hbetaR_l}{n} \sum_i \p{\psi_k(U_i) - \psiRh_{ki}}  \psiRh_{l i}
\end{align*}
For each fixed pair of $k, l$, we'll analyze $ \sum_i \p{\psi_k(U_i) - \psiRh_{ki}}  \psiRh_{l i} = \p{\psi_k - \psiRh_k}^T\psiRh_l$. The absolute value of it can be bounded by
\begin{align*}
\abs{\p{\psi_k - \psiRh_k}^T\psiRh_l}
 &\leq \abs{\p{\psi_k - \psiRh_k}^T\p{\psi_l - \psiRh_l}} + \abs{\p{\psi_k - \psiRh_k}^T\psi_l}\\
 &\leq \Norm{\psi_k - \psiRh_k} \Norm{\psi_l - \psiRh_l} + \abs{\p{\psi_k - \psiRh_k}^T\psi_l}\\
 & = \oo_p\p{ \frac{\p{\log n}^4}{\rhon} + \sqrt{n}} = \oo_p\p{ \sqrt{n}},
\end{align*}
where the last line comes from \eqref{eq:psi_close} and Lemma \ref{lemm:a_psi_close}. Combining with the fact that $\abs{\mu_k} \leq B$ and $\hbetaR_k = \oo_p\p{\sqrt{n} \rhon}$, we get
\[S_{c2} = \oo_p\p{\frac{1}{n} \sqrt{n} \rhon \sqrt{n}} = \oo_p\p{\rhon}. \]

\subsection{Proof of Proposition \ref{prop:summary}}
With \eqref{eq:baseline_split} and \eqref{eq:align_remainder} plugged into \eqref{eq:ind_decomp}, we can rewrite the first line of \eqref{eq:ind_decomp} as
\begin{align*}
&\quad \quad \frac{1}{n}\sum_i \p{\frac{M_i}{\pi} -\frac{N_i - M_i}{1 - \pi} + \sum_{k=1}^r \hbeta_k \psih_{k i}} f_i(W_i,\pi) \\
&	=  \frac{1}{n}\sum_i \p{\frac{M_i}{\pi} -\frac{N_i - M_i}{1 - \pi} + \sum_{k=1}^r \hbeta_k \psih_{k i}} (W_i - \pi) \sqb{f_i(1 , \pi) - f_i(0 , \pi)} \\
&	 \quad + \frac{1}{n} \sum_{k = 1}^r \mu_k \sum_i \p{\frac{M_i}{\pi} -\frac{N_i - M_i}{1 - \pi} + \sum_{k=1}^r \hbeta_k \psih_{k i}} \p{\psi_k(U_i) - \psiRh_{ki}}\\
&	 \quad +  \frac{1}{n}\sum_i \p{\frac{M_i}{\pi} -\frac{N_i - M_i}{1 - \pi} + \sum_{k=1}^r \hbeta_k \psih_{k i}}  \eta_i\\
& = S_{d1} + S_{d2} + S_{d3}
\end{align*}
where $S_{d1}, S_{d2}, S_{d3}$ are the three summations respectively. We'll analyze them one by one. 

\paragraph{$\mathbf{S_{d1}}$}
We analyze $S_{d1} =  \frac{1}{n} \sum_i (W_i - \pi) \sqb{f_i(1 , \pi) - f_i(0 , \pi)} \p{ \frac{1}{\pi(1-\pi)}\p{M_i - \pi N_i} + \sum_{k=1}^r \hbeta_k \psih_{k i}} $. Let $S_{d11}$ and $S_{d12}$ correspond to the two summations involving $\frac{1}{\pi(1-\pi)}\p{M_i - \pi N_i}$ and $\sum_{k=1}^r \hbeta_k \psih_{k i}$ respectively. $S_{d1} = S_{d11} + S_{d12}$. 

For $S_{d11}$,
\begin{align*}
S_{d11} &= \frac{1}{n\pi (1-\pi)} \sum_i (W_i - \pi) (M_i - \pi N_i) \p{f_i(1, \pi) - f_i(0, \pi)}\\
& = \frac{1}{n\pi (1-\pi)}\sum_i (W_i - \pi) \sum_{j \neq i} E_{ij} (W_j - \pi) \p{f_i(1,\pi) - f_i(0,\pi)} \\
& = \frac{1}{n\pi (1-\pi)}\sum_{\substack{i,j\\ i\neq j}}(W_i - \pi)  E_{ij} (W_j - \pi) \p{f_i(1,\pi) - f_i(0,\pi)}. 
\end{align*}
$S_{d11}$ is one of the leading terms. 

For $S_{d12}$,
\begin{align*}
S_{d12} 
&=\frac{1}{n} \sum_i (W_i - \pi) \sqb{f_i(1 , \pi) - f_i(0 , \pi)} \p{ \sum_{k=1}^r \hbeta_k \psih_{k i}}\\
& = \sum_{k=1}^r \frac{\mu_k \hbeta_k}{n} \sum_i (W_i - \pi) \sqb{f_i(1 , \pi) - f_i(0 , \pi)} \psih_{k i} 
\end{align*}
For each $k \in \cb{1,2, \dots, r}$, we'll study the term $\sum_i(W_i - \pi) \sqb{f_i(1 , \pi) - f_i(0 , \pi)} \psih_{k i}$. When computing its second moment, cross terms vanish as $\EE{(W_i- \pi) (W_j - \pi)} = 0$ for $i \neq j$. Hence its second moment equals to
\begin{align*}
& \quad \quad \EE{\p{\sum_i(W_i - \pi) \sqb{f_i(1 , \pi) - f_i(0 , \pi)} \psih_{k i}}^2}\\ 
& = \sum_i \EE{(W_i - \pi)^2}\EE{\p{f_i(1 , \pi) - f_i(0 , \pi)}^2 \psih_{k i}^2}
 \leq C B \EE{ \sum_i \psih_{k i}^2} \leq C B n.
\end{align*}
This implies that 
\[ \sum_i (W_i - \pi) \sqb{f_i(1 , \pi) - f_i(0 , \pi)} \psih_{k i} = \oo_p\p{\sqrt{n}}.  \]
Combining with Lemma \ref{lemm:beta_hat_k} and the fact that $\abs{\mu_k} \leq B$, 
\[S_{d12} = \oo_p\p{ \frac{\sqrt{n}\rhon}{n} \sqrt{n} } = \oo_p\p{\rhon}. \]

\paragraph{$\mathbf{S_{d2}}$}
$S_{d2} = \oo_p\p{\rhon}$, by Proposition \ref{prop:term_5}.

\paragraph{$\mathbf{S_{d3}}$}
\[S_{d3} = \sum_i \hgamma_i \eta_i =  \frac{1}{n} \sum_i \eta_i \p{ \frac{1}{\pi(1-\pi)}\p{M_i - \pi N_i} + \sum_{k=1}^r \hbeta_k \psih_{k i}} = S_{d31} + S_{d32}. \]

For $S_{d31}$, 
\[ S_{d31} = \frac{1}{n\pi(1-\pi)}\sum_i \eta_i (M_i - \pi N_i) = \frac{1}{n\pi(1-\pi)}\sum_{\substack{(i,j)\\i \neq j}} \eta_i E_{ij} (W_j - \pi). \]
This is another leading term. 

For $S_{d32}$,
\[ S_{d32} = \sum_k\frac{\hbeta_k}{n} \sum_i \eta_i \psih_{ki} = \sum_k\frac{\hbetaR_k}{n} \sum_i \eta_i \psiRh_{ki}.\]
We'll study $ \sum_i \eta_i \psiRh_{ki}$ for each $k$. We start by decomposing it into two terms
\[ \sum_i \eta_i \psiRh_{ki} =  \sum_i \eta_i \p{\psiRh_{ki} - \psi_k(U_i)} + \sum_i \eta_i \psi_k(U_i).\]
For the first term, by Lemma \ref{lemm:a_psi_close},
\[ \sum_i \eta_i \p{\psiRh_{ki} - \psi_k(U_i)}  = 
\oo_p\p{\sqrt{n}}. \]
For the second term, note first that for $i \neq j$, $\EE{\eta_i \psi_k(U_i) \eta_j \psi_k(U_j)} = 0$. This implies that
\[ \EE{\p{\sum_i \eta_i \psi_k(U_i)}^2}  = \sum_i \EE{\eta_i ^2 \psi_k^2 (U_i)} \leq \sum_i \EE{\eta_i ^4} \EE{\psi_k^4 (U_i)} \leq CBn. \]
Hence $\sum_i \eta_i \psi_k(U_i) = \oo_p\p{\sqrt{n}}$. Combining the results on first part, we show that
\[  \sum_i \eta_i \psiRh_{ki} =   \oo_p\p{\sqrt{n}},\]
and hence together with Lemma \ref{lemm:beta_hat_k},
\[ S_{d32} = \sum_k\frac{\hbetaR_k}{n} \sum_i \eta_i \psiRh_{ki} = \oo_p\p{\frac{\sqrt{n} \rhon}{n} \sqrt{n}  } =\oo_p\p{\rhon}.\]

\paragraph{Putting them together}

Combining $S_{d1}$, $S_{d2}$ and $S_{d3}$, we get the first line of \eqref{eq:ind_decomp} can be written as
\begin{align*}
&\quad \quad \frac{1}{n}\sum_i \p{\frac{M_i}{\pi} -\frac{N_i - M_i}{1 - \pi} + \sum_{k=1}^r \hbeta_k \psih_{k i}} f_i(W_i,\pi) 
 = S_{d11} + S_{d31} + \oo_p\p{\rhon}, 
\end{align*}
where the two leading terms $S_{d11} = \frac{1}{n\pi (1-\pi)}\sum_{\substack{(i,j)\\ i\neq j}}(W_i - \pi)  E_{ij} (W_j - \pi) \p{f_i(1,\pi) - f_i(0,\pi)},$ and $S_{d31}  = \frac{1}{n\pi(1-\pi)}\sum_{\substack{(i,j)\\i \neq j}} \eta_i E_{ij} (W_j - \pi)$. 

Together with bounds on the second, third and fourth line of  \eqref{eq:ind_decomp} (Proposition \ref{prop:term_2} - \ref{prop:term_4}), we get 
\[
\tauc_{\IND} - \tau_{\IND} = \frac{1}{n \pi (1 - \pi)} \sum_{(i,j), i \neq j} (W_i - \pi) E_{ij} \xi_j + o_p\p{\sqrt{\rho_n}}, \]
where $\xi_j = (W_j - \pi) \p{f_i(1, \, \pi) - f_i(0, \, \pi)} + \eta_j$.

\subsection{Proof of Proposition \ref{prop:CLT_leading}}

Recall that $\xi_i = (W_i - \pi) (f_i(1,\pi) - f_i(0,\pi)) + \eta_i$, $b_i = \pi f_i(1,\pi) + (1-\pi)f_i(0,\pi)$ and $\eta_i = b_i - \sum_{k=1}^r \EE{b_i \psi_k(U_i)}\psi_k(U_i)$. Hence $\EE{\xi_i \psi_k(U_i)} = 0$ for any $k$ and $\abs{\xi_i} \leq 2 (r+1) B$. 

Define $\Delta_{ij} = (\rhon \Gfcnc(U_i, U_j) - 1) \mathbf{1}\cb{\rhon \Gfcnc(U_i, U_j) >1 }$. Define $F_{ij} = E_{ij} + \Delta_{ij}$. Then $\EE{F_{ij}|U_i, U_j} = \rhon \Gfcnc(U_i, U_j)$. Define 
\[   \vep_n = \frac{1}{n \pi (1-\pi)} \sum_{\substack{(i,j)\\	i \neq j}} (W_i - \pi) F_{ij} \xi_j. \]

We'll firstly show that $\vep_n$ is close enough to $\epsilon_n$. Then proceed to deal with $\vep_n$. Before we start, we prove a lemma. 
\begin{lemm}
\label{lemma:bound_F_E}
For $i,j,k$ distinct, $c_1, c_2 \leq 2$, under the conditions of Theorem \ref{theo:IND_CLT}, 
\[\EE{\p{F_{ij}^{c_1} F_{jk}^{c_2} - E_{ij}^{c_1 \wedge 1} E_{jk}^{c_2 \wedge 1}}^2} \leq Ce^{-Cn^{\kappa_1/2
}},\]
for some constant $C$. 
\end{lemm}

\begin{proof}
Recall that $F_{ij} = E_{ij} + \Delta_{ij}$ and $\Delta_{ij} = (\rhon \Gfcnc(U_i, U_j) - 1) \mathbf{1}\cb{\rhon \Gfcnc(U_i, U_j) >1 }$. Hence 
\begin{align*}
\EE{\p{F_{ij}^{c_1} F_{jk}^{c_2} - E_{ij}^{c_1 \wedge 1} E_{jk}^{c_2 \wedge 1}}^2} 
&\leq \EE{F_{ij}^{2c_1} F_{jk}^{2c_2}\mathbf{1}\cb{\rhon \Gfcnc(U_i, U_j) >1 } }\\
& \leq \EE{(\Gfcnc(U_i, U_j) + 1)^{2c_1} (\Gfcnc(U_j, U_k) + 1)^{2c_1}\mathbf{1}\cb{\rhon \Gfcnc(U_i, U_j) >1 } }\\
& \leq \sqrt{\EE{(\Gfcnc(U_i, U_j) + 1)^{4c_1} (\Gfcnc(U_j, U_k) + 1)^{4c_1}} \PP{\rhon \Gfcnc(U_i, U_j) >1}}\\
& \leq Ce^{-Cn^{\kappa_1/2}},
\end{align*}
by definition of the Berstein condition and Lemma \ref{lemma:berstein_bound}. 
\end{proof}

We'll  show that $\vep_n$ is close enough to $\epsilon_n$. The difference between the two is $\vep_n - \epsilon_n = \frac{1}{n\pi(1-\pi)} \sum_{(i,j), i \neq j} (W_i - \pi) \Delta_{ij} \xi_j$. Its second moment can be bounded by
\begin{align*}
\EE{(\vep_n - \epsilon_n)^2}
& \leq \frac{1}{n^2\pi^2 (1-\pi)^2} n^2 \sum_{\substack{(i,j)\\ i \neq j}} \EE{(W_i - \pi)^2 \Delta_{ij}^2 \xi_j^2} \\
& \leq \frac{C B^2}{\pi^2 (1-\pi)^2} \sum_{\substack{(i,j)\\ i \neq j}} \EE{\Delta_{ij}^2} 
\leq C e^{-Cn^\kappa_1/2},
\end{align*}
where the last inequality follows from Lemma \ref{lemm:G_Gn_bound}. Specifically, this implied that $(\vep_n - \epsilon_n)/\sqrt{\rhon}  = o_p (1)$. 

Now we'll proceed to study the asymptotic distribution of $\vep_n$.  Note first that by Lemma \ref{lemma:cross_vanish}, $\EE{\vep_n}$ = 0. We'll then compute its asymptotic variance. Again by Lemma \ref{lemma:cross_vanish}, 
\begin{align*}
\Var{\vep_n} 
& = \frac{1}{n^2 \pi^2 (1-\pi)^2} \sum_{\substack{(i,j)\\	i \neq j}} \p{\EE{(W_i - \pi)^2 F_{ij}^2 \xi_j^2} + \EE{(W_i - \pi)(W_j - \pi) F_{ij}^2 \xi_j \xi_j}}\\
& = \frac{n(n-1)}{n^2 \pi^2 (1-\pi)^2} \p{ \EE{(W_1 - \pi)^2 F_{12}^2 \xi_2^2} + \EE{(W_1 - \pi)(W_2 - \pi) F_{12}^2 \xi_1 \xi_2}}
\end{align*}
Note that by Lemma \ref{lemma:bound_F_E}, the two expectation terms satisfy
\begin{align*}
& \quad \quad \EE{(W_1 - \pi)^2 F_{12}^2 \xi_2^2 + (W_1 - \pi)(W_2 - \pi) F_{12}^2 \xi_1 \xi_2}\\
& = \EE{(W_1 - \pi)^2 E_{12}^2 \xi_2^2 + (W_1 - \pi)(W_2 - \pi) E_{12}^2 \xi_1 \xi_2} + \oo(e^{-Cn^{\kappa_1/2}})\\
& = \EE{(W_1 - \pi)^2 E_{12} \xi_2^2 + (W_1 - \pi)(W_2 - \pi) E_{12} \xi_1 \xi_2} + \oo(Ce^{-Cn^{\kappa_1/2}})\\
& = \EE{(W_1 - \pi)^2 F_{12} \xi_2^2 + (W_1 - \pi)(W_2 - \pi) F_{12} \xi_1 \xi_2} + \oo(Ce^{-Cn^{\kappa_1/2}})\\
& =  \rhon \EE{(W_1 - \pi)^2 G(U_1, U_2) \xi_2^2 + (W_1 - \pi)(W_2 - \pi) G(U_1, U_2) \xi_1 \xi_2} + \oo(Ce^{-Cn^{\kappa_1/2}}). 
\end{align*}
Let $\sigma_{\IND}^2 = \frac{1}{\pi^2 (1-\pi)^2}\EE{(W_1 - \pi)^2 G(U_1, U_2) \xi_2^2 + (W_1 - \pi)(W_2 - \pi) G(U_1, U_2) \xi_1 \xi_2}$, then 
\[\Var{\vep_n} = \rhon \sigma_{\IND}^2 + \oo(\rhon/n). \]

Here we want to find a better expression for $\sigma_{\IND}^2$. Recall that $\xi_i = (W_i - \pi) (f_i(1,\pi) - f_i(0,\pi)) + \eta_i$, and $\eta_i = b_i - \sum_{k=1}^r \EE{b_i \psi_k(U_i)}\psi_k(U_i)$ where $b_i = \pi f_i(1,\pi) + (1-\pi)f_i(0,\pi)$. Define $\alpha_i  = (f_i(1,\pi) - f_i(0,\pi))$. Hence $\xi_i = \eta_i + (W_i - \pi)\alpha_i$, $\xi_2^2 = \eta_2^2 + (W_2 - \pi)^2\alpha_2^2 + 2(W_2 - \pi)\alpha_2\eta_2$ and $\xi_1 \xi_2 = \eta_1\eta_2 + (W_1 - \pi)(W_2 - \pi)\alpha_1\alpha_2 + \eta_1(W_2 - \pi)\alpha_2 + \eta_2(W_1 - \pi)\alpha_1$. 
Hence $\EE{(W_1 - \pi)^2 \Gfcnc(U_1,U_2) \xi_2^2} = \pi(1-\pi)\EE{\Gfcnc(U_1,U_2) (\eta_2^2 + (W_2 - \pi)^2\alpha_2^2 + 2(W_2 - \pi)\alpha_2\eta_2)}$. 
This further equals to $\pi(1-\pi)\EE{\Gfcnc(U_1,U_2) \eta_2^2} + \pi^2(1-\pi)^2\EE{\Gfcnc(U_1,U_2) \alpha_2^2}$. 
Note that $\EE{\Gfcnc(U_1,U_2) \eta_2^2} = \EE{\Gfcnc(U_1,U_2) \eta_1^2} =  \EE{\gsmlc(U_1) \eta_1^2}$. And for the term $\mathbb{E}[(W_1 - \pi)(W_2 - \pi) \Gfcnc(U_1, U_2) \xi_1 \xi_2]$, note that all terms in $\xi_1\xi_2$ except for $(W_1 - \pi)(W_2 - \pi)\alpha_1\alpha_2$ are uncorrelated with $(W_1 - \pi)(W_2 - \pi) \Gfcnc(U_1, U_2)$. 
Hence $\EE{(W_1 - \pi)(W_2 - \pi) \Gfcnc(U_1, U_2) \xi_1 \xi_2} = \mathbb{E} [(W_1 - \pi)^2(W_2 - \pi)^2 \Gfcnc(U_1, U_2) \alpha_1 \alpha_2] = \pi^2(1-\pi)^2 \EE{\Gfcnc(U_1, U_2)\alpha_1 \alpha_2}$,
Combining the results, we find
\[\sigma_{\IND}^2 = \EE{\Gfcnc(U_1,U_2)\p{\alpha_1^2+ \alpha_1\alpha_2}} + \EE{\gsmlc(U_1) \eta_1^2}/(\pi(1-\pi)),\]
where in the above expression, $\alpha_i  = (f_i(1,\pi) - f_i(0,\pi))$, $b_i = \pi f_i(1,\pi) + (1-\pi)f_i(0,\pi)$ and $\eta_i = b_i - \sum_{k=1}^r \EE{b_i \psi_k(U_i)}\psi_k(U_i)$. 

We then move on to show a central limit theorem. We'll show  that $\vep_n/\p{\sqrt{\rhon}\sigma_{\IND}} \stackrel{d}{\to} \mathcal{N}(0,1) $. We make use of Theorem 3.6 and its proof in \citep{ross2011fundamentals}. Specifically, we make use of the following result. 

\begin{theo}[Central limit theorem for sums of random variables with local dependence]
\label{theo:stein}
We say that a collection of random variables $\left(X_{1}, \ldots, X_{m}\right)$ has dependency neighborhoods $N_{a} \subseteq\{1, \ldots, m\}, a=1, \ldots, m,$ if $a \in N_{a}$ and $X_{a}$ is independent of $\left\{X_{b}\right\}_{b \notin N_{a}}$. Let $X_{1}, \ldots, X_{m}$ be random variables such that $\EE{X_{a}}=0, \sigma^{2}=\Var{\sum_{a} X_{a}},$ and define $Y=\sum_{i} X_{i} / \sigma .$ Let the collection
$\left(X_{1}, \ldots, X_{m}\right)$ have dependency neighborhoods $N_{a}, a=1, \ldots, m$. Then for $Z$ a standard normal random variable, the Wasserstein distance between $Y$ and $Z$ is bounded above by
\begin{equation}
\label{eqn:stein}
d_{\mathrm{W}}(Y, Z) \leq\frac{1}{\sigma^{3}} \sum_{a=1}^{m} \mathbb{E}\left|X_{a}\left(\sum_{b \in N_{a}} X_{b}\right)^{2}\right| 
+ \frac{\sqrt{2}}{\sqrt{\pi} \sigma^{2}} \sqrt{\Var{\sum_{a=1}^{m} X_{a}   \sum_{b \in N_{a}} X_{b}}}. 
\end{equation}
\end{theo}

Here we take $m = n(n-1)$ and each index $a$ corresponds to a pair of $(i,j)$. We take $X_{(i,j)} = (W_i - \pi)F_{ij} \xi_j$ and $\sigma^2 = \Var{n \pi (1-\pi) \vep}$. Note that by the above variance calculation, we know that $\sigma \sim n \sqrt{\rhon}$. The dependency neighborhood of $(i,j)$ corresponds to 
\begin{equation}
\label{eqn:stein_sum_of_neighbor}
\sum_{b \in N_{(i,j)}} X_b = X_{(i,j)} + X_{(j,i)} + \sum_{k \neq i,j}  X_{(i,k)} + \sum_{k \neq i,j} X_{(j,k)} + \sum_{k \neq i,j}  X_{(k,i)} + \sum_{k \neq i,j} X_{(k,j)} . 
\end{equation}

\paragraph{The term $\frac{1}{\sigma^{3}} \sum_{a=1}^{m} \mathbb{E}\left|X_{a}\left(\sum_{b \in N_{a}} X_{b}\right)^{2}\right| $ in (\ref{eqn:stein})}
\quad

We firstly look at the term $\frac{1}{\sigma^{3}} \sum_{a=1}^{m} \mathbb{E}\left|X_{a}\left(\sum_{b \in N_{a}} X_{b}\right)^{2}\right| $ in (\ref{eqn:stein}). Note that $\abs{X_{(i,j)}} \leq C B F_{ij}$. Hence it suffices to bound 
$\EE{F_{ij}\left(\sum_{b \in N_{(i,j)}} X_{b}\right)^2  }$. We decompose $\sum_{b \in N_{(i,j)}} X_{b}$ into a few different parts as in equation (\ref{eqn:stein_sum_of_neighbor}), and analyze them one by one. Note first that by Lemma \ref{lemma:bound_F_E}, $\EE{F_{ij} (X_{(i,j)} + X_{(j,i)})^2} \leq CB^2\EE{F_{ij}^2} \leq Ce^{-Cn^{\kappa_1/2}} + C\EE{E_{ij}} \leq CB^2\rhon$. 

Then for $\EE{F_{ij} \p{ \sum_{k \neq i,j}  X_{(i,k)}}^2}$, 
\begin{align}
\EE{F_{ij} \p{ \sum_{k \neq i,j}  X_{(i,k)}}^2} 
& = \EE{F_{ij}  \sum_{k \neq i,j}  X_{(i,k)}^2 } +  \EE{F_{ij}  \sum_{\substack{(k_1,k_2)\\k_1, k_2, i, j \text{ distinct}}}  X_{(i,k_1)} X_{(i,k_2)}}\\
\label{eqn:stein_term1}
& = \EE{F_{ij} \sum_{k \neq i, j}(W_i - \pi)^2 F_{ik}^2 \xi_k^2}\\
\label{eqn:stein_term2}
& \quad \quad + \EE{F_{ij}  \sum_{\substack{(k_1,k_2)\\k_1, k_2, i, j \text{ distinct}}} (W_i - \pi)^2 F_{i k_1} \xi_{k_1}F_{i k_2}\xi_{k_2} }
\end{align}
For (\ref{eqn:stein_term1}), by Lemma \ref{lemma:bound_F_E}, $\EE{F_{ij} \sum_{k \neq i, j}(W_i - \pi)^2 F_{ik}^2 \xi_k^2}
\leq 
C B^2 \EE{F_{ij}  \sum_{k \neq i,j}  F_{ik}^2}
\leq
C B^2$ $  \mathbb{E} \Big[ E_{ij}   \sum_{k \neq i,j}  E_{ik} \Big]+ Ce^{-Cn^{\kappa_1/2}} = C B^2  (n-2)  \EE{\Hfcn(U_j, U_k)} + Ce^{-Cn^{\kappa_1/2}} \leq C B^2  n \rhon^2 \EE{\Hfcnc(U_j, U_k)} + Ce^{-Cn^{\kappa_1/2}} $. 
For (\ref{eqn:stein_term2}), for $k_1, k_2, i, j$ all distinct, $\EE{(W_i - \pi)^2 F_{ij} F_{i k_1} \xi_{k_1}F_{i k_2}\xi_{k_2} } = \pi(1-\pi) \rhon^3 \mathbb{E} \big[ \Gfcnc(U_i,U_j) \Gfcnc(U_i,U_{k_1})\Gfcnc(U_i,U_{k_2}) \xi_{k_1} \xi_{k_2}   \big]$. By the low rank assumption (\ref{eqn:rankr_assumption}), each $\Gfcnc(U_i, U_j)$ can be written as $\Gfcnc(U_i, U_j) =  \sum_{k = 1}^r \lambda_k \psi_k(U_i) \psi_k(U_j)$. For indices $l_1, l_2, l_3 \in \cb{1,2, \dots, r}$, consider $\EE{\psi_{l_1}(U_i)\psi_{l_1}(U_j) \psi_{l_2}(U_i) \psi_{l_2}(U_{k_1}) \psi_{l_3}(U_i) \psi_{l_3}(U_{k_2}) \xi_{k_1} \xi_{k_2}  }$. This is 0 by the property that $\EE{\xi_i  \psi_l(U_i)} = 0$ for any $l \leq r$. Hence this implies that $\mathbb{E} \big[ \Gfcnc(U_i,U_j) \Gfcnc(U_i,U_{k_1}) $  $\Gfcnc(U_i,U_{k_2}) \xi_{k_1} \xi_{k_2}   \big] = 0$. Combining the two bounds on (\ref{eqn:stein_term1}) and (\ref{eqn:stein_term2}), we get 
\begin{equation}
\label{eqn:stein_term3}
\EE{F_{ij} \p{ \sum_{k \neq i,j}  X_{(i,k)}}^2}  \leq C B^2 n\rhon^2.
\end{equation}
By symmetry of $i$ and $j$, $\EE{F_{ij} \p{ \sum_{k \neq i,j}  X_{(j,k)}}^2}$ can be bounded by the same bound as in (\ref{eqn:stein_term3}). 

Now for $\EE{F_{ij} \p{ \sum_{k \neq i,j}  X_{(k,i)}}^2}$, 
\begin{align}
\EE{F_{ij} \p{ \sum_{k \neq i,j}  X_{(k,i)}}^2} 
& = \EE{F_{ij}  \sum_{k \neq i,j}  X_{(k,i)}^2 } +  \EE{F_{ij}  \sum_{\substack{(k_1,k_2)\\k_1, k_2, i, j \text{ distinct}}}  X_{(k_1,i )} X_{(k_2,i)}}\\
\label{eqn:stein_term4}
& = \EE{F_{ij} \sum_{k \neq i, j}(W_k - \pi)^2 F_{ik}^2 \xi_i^2}\\
\label{eqn:stein_term5}
& \quad \quad + \EE{F_{ij}  \sum_{\substack{(k_1,k_2)\\k_1, k_2, i, j \text{ distinct}}} (W_{k_1} - \pi)(W_{k_2} - \pi) F_{i k_1} F_{i k_2}\xi_{i}^2 }
\end{align}
(\ref{eqn:stein_term4}) can be bounded the same way as (\ref{eqn:stein_term1}), and (\ref{eqn:stein_term5}) is zero as $W_{k_1} - \pi$ is mean zero and independent of everything else. Hence  
\begin{equation}
\label{eqn:stein_term6}
\EE{F_{ij} \p{ \sum_{k \neq i,j}  X_{(k,i)}}^2}  \leq C B^2 n\rhon^2.
\end{equation}
Again by symmetry of $i$ and $j$, $\EE{F_{ij} \p{ \sum_{k \neq i,j}  X_{(k,j)}}^2}$ can be bounded by the same bound as in (\ref{eqn:stein_term3}). 

With the decomposition in (\ref{eqn:stein_sum_of_neighbor}), combining the bounds in (\ref{eqn:stein_term3}), (\ref{eqn:stein_term6}), their corresponding $j$ version, and the bound on $\EE{F_{ij} (X_{(i,j)} + X_{(j,i)})^2}$, we get
\[
\EE{F_{ij}\left(\sum_{b \in N_{(i,j)}} X_{b}\right)^2  } \leq CB^2 (n\rhon^2 + \rhon) \leq C B^2 n\rhon^2.
\]
Together with the fact that $\abs{X_{(i,j)}} \leq C B F_{ij}$, we get
\[
\frac{1}{\sigma^{3}} \sum_{a=1}^{m} \mathbb{E}\left|X_{a}\left(\sum_{b \in N_{a}} X_{b}\right)^{2}\right| 
\leq \frac{n^2}{\sigma^3} CB^3n \rhon^2 = \frac{C}{\sigma^3} B^3 n^3 \rhon^2. 
\]
As $\sigma \sim n \sqrt{\rhon}$, we therefore have 
\[\frac{1}{\sigma^{3}} \sum_{a=1}^{m} \mathbb{E}\left|X_{a}\left(\sum_{b \in N_{a}} X_{b}\right)^{2}\right| = \oo_p(\sqrt{\rhon}). \]

\paragraph{The term $ \frac{\sqrt{2}}{\sqrt{\pi} \sigma^{2}} \sqrt{\Var{\sum_{a=1}^{m} X_{a}   \sum_{b \in N_{a}} X_{b}}}$ in (\ref{eqn:stein})} 
\quad

Again we'll decompose $\sum_{b \in N_{(i,j)}} X_{b}$ into a few different parts as in equation (\ref{eqn:stein_sum_of_neighbor}), and then analyze them one by one. We start with $\sum_{(i,j), i,j \text{ distinct}} X_{(i,j)} (X_{(i,j)} + X_{(j,i)})$. For sake of notation, define $Y_{ij} = X_{(i,j)} \p{X_{(i,j)} + X_{(j,i)}}$. Note that $Y_{ij}^2 \leq C B^4 F_{ij}^2$. Hence by Lemma \ref{lemma:bound_F_E}, $\EE{Y_{ij}^2} \leq C B^4 \EE{E_{ij}} + C e^{-C n^{\kappa_1}} \leq C B^4 \rhon$. Note also that for $i,j,k,l$ all distinct, $\Cov{Y_{ij}, Y_{kl}} = 0$ as they are independent. Hence by Lemma \ref{lemma:summation_decompose}, 
\begin{align*}
\Var{\sum_{(i,j), i \neq j } X_{(i,j)} (X_{(i,j)} + X_{(j,i)})}
&\leq 
\Var{\sum_{(i,j),i \neq j} Y_{ij}} 
\leq CB^4 n^3\rhon.
\end{align*}

For $\sum_{(i,j), i,j \text{ distinct}} X_{(i,j)} \sum_{k \neq i, j} X_{(i,k)}$. We can rewrite this term as $\sum_{(i,j,k) \text{ all distinct}}$ $X_{(i,j)} X_{(i,k)} = \sum_{(i,j,k) \text{ all distinct}} (W_i - \pi)^2 F_{ij} \xi_j F_{ik} \xi_k$. Note that
\begin{align*}
& \quad \quad \EE{\p{ \sum_{\substack{(i,j,k)\\ \text{ all distinct}}} (W_i - \pi)^2 F_{ij} \xi_j F_{ik} \xi_k}^2} \\
&=  \sum_{\substack{(i_1,j_1,k_1) \text{ distinct}\\(i_2,j_2,k_2) \text{ distinct}}} 
\EE{(W_{i_1} - \pi)^2 F_{i_1 j_1} \xi_{j_1} F_{i_1k_1} \xi_{k_1}  (W_{i_2} - \pi)^2 F_{i_2 j_2} \xi_{j_2} F_{i_2k_2} \xi_{k_2} }\\
\end{align*}
We can simplify the above expression by replacing all $F_{ab}$ by $\rhon G(U_a, U_b)$. By the rank-$R$ assumption, each $G(U_a, U_b)$ can be further decomposed into a linear combination of $\psi_l(U_a) \psi_l(U_b)$. With these operations, the expression becomes a summation over terms of form: 
$\mathbb{E}
 \Big[(W_{i_1} - \pi)^2 \psi_{l_1}(U_{i_1}) \psi_{l_1}(U_{j_1}) \xi_{j_1} 
\psi_{l_2}(U_{i_1}) \psi_{l_2}(U_{k_1})
 \xi_{k_1}  (W_{i_2} - \pi)^2 
\psi_{l_3}(U_{i_2}) \psi_{l_3}(U_{j_2})
\xi_{j_2} 
$
$
\psi_{l_4}(U_{i_2}) \psi_{l_4}(U_{k_2})
 \xi_{k_2} \Big]$,
where $l_1, l_2, l_3, l_4 \in \cb{1,\dots, r}$. If $k_1$ appears only once in $(i_1, j_1, k_1, i_2, j_2,k_2)$, the expectation is zero as we can separate out $\EE{\xi_{k_1} \psi_l(U_{k_1})}$ for some $l$. Same for $j_1$: if $j_1$ appears only once in $(i_1, j_1, k_1, i_2, j_2,k_2)$, the expectation is zero as we can separate out $\EE{\xi_{j_1} \psi_l(U_{j_1})}$. This implies that the summation above is the same as the summation over $(i_1, j_1, k_1, i_2, j_2, k_2)$ such that $i_1, j_1, k_1$ are distinct, $i_2, j_2, k_2$ are distinct, and they only take most 4 different values. If they take only 3 different values, 
\begin{align*}
&\quad \quad  
\EE{(W_{i_1} - \pi)^2 F_{i_1 j_1} \xi_{j_1} F_{i_1k_1} \xi_{k_1}  (W_{i_2} - \pi)^2 F_{i_2 j_2} \xi_{j_2} F_{i_2k_2} \xi_{k_2} }\\
&\leq 
\EE{(W_{i_1} - \pi)^2 E_{i_1 j_1} \xi_{j_1} E_{i_1k_1} \xi_{k_1}  (W_{i_2} - \pi)^2 E_{i_2 j_2} \xi_{j_2} E_{i_2k_2} \xi_{k_2} } + Ce^{-Cn^{\kappa_1/2}}\\
& \leq CB^4 \rhon^2.
\end{align*}
But there are at most $C n^3$ many combinations of such $(i_1, j_1, k_1, i_2, j_2, k_2)$. If they take 4 different values, then 
\begin{align*}
&\quad \quad  
\EE{(W_{i_1} - \pi)^2 F_{i_1 j_1} \xi_{j_1} F_{i_1k_1} \xi_{k_1}  (W_{i_2} - \pi)^2 F_{i_2 j_2} \xi_{j_2} F_{i_2k_2} \xi_{k_2} }\\
&\leq 
\EE{(W_{i_1} - \pi)^2 E_{i_1 j_1} \xi_{j_1} E_{i_1k_1} \xi_{k_1}  (W_{i_2} - \pi)^2 E_{i_2 j_2} \xi_{j_2} E_{i_2k_2} \xi_{k_2} } + Ce^{-Cn^{\kappa_1/2}}\\
& \leq CB^4 \rhon^3.
\end{align*}
There are at most $C n^4$ many combinations of such $(i_1, j_1, k_1, i_2, j_2, k_2)$. Combining the above arguments, we have
\[
\EE{\p{ \sum_{(i,j), i,j \text{ distinct}} X_{(i,j)} \sum_{k \neq i, j} X_{(i,k)}}^2}
\leq CB^4 (n^3 \rhon^2 + n^4 \rhon^3 )  \leq CB^4  n^4 \rhon^3.
\]

For the other three terms in (\ref{eqn:stein_sum_of_neighbor}), we bound them following the exact same logic as above. We make use of the fact that some expectations are zero, if one index appears only once. We'll omit the details here. Following the arguments, we can get
\[ \EE{\p{\sum_{(i,j), i,j \text{ distinct}} X_{(i,j)} \sum_{k \neq i, j} X_{(j,k)}}^2} \leq CB^4 n^4 \rhon^3. \]
\[ \EE{\p{\sum_{(i,j), i,j \text{ distinct}} X_{(i,j)} \sum_{k \neq i, j} X_{(k,i)}}^2} \leq CB^4 n^4 \rhon^3. \]
\[ \EE{\p{\sum_{(i,j), i,j \text{ distinct}} X_{(i,j)} \sum_{k \neq i, j} X_{(k,j)}}^2} \leq CB^4 n^4 \rhon^3. \]

Combining bounds corresponding to different terms in (\ref{eqn:stein_sum_of_neighbor}), we get
\[\Var{\sum_{(i,j), i,j \text{ distinct}} X_{(i,j)}   \sum_{b \in N_{(i,j)}} X_{b}} \leq CB^4 n^4 \rhon^3 + CB^4 n^3 \rhon \leq CB^4 n^4 \rhon^3.\]
Therefore
\[ \frac{\sqrt{2}}{\sqrt{\pi} \sigma^{2}} \sqrt{\Var{\sum_{a=1}^{m} X_{a}   \sum_{b \in N_{a}} X_{b}}} = \oo\p{ \frac{n^2 \rhon^{3/2}}{n^2\rhon}  } = \oo\p{\sqrt{\rhon}}. \]

Then by Theorem \ref{theo:stein}, $ d_{\operatorname{W}}\p{\vep_n/\sqrt{\Var{\vep_n}}, Z} =  \oo\p{\sqrt{\rhon}}$. Together with the fact that $(\vep_n - \epsilon_n)/\sqrt{\rhon} = o_p(1)$ and $\Var{\vep_n} = \rhon \sigma_{\IND}^2 + \oo\p{\rhon/n}$, we get 
\[  \frac{\epsilon_n}{\sqrt{\rhon}} \stackrel{d}{\to} \mathcal{N}(0,\sigma_{\IND}^2). \]

\section{What Do Standard Variance Estimators Estimate under Interference?}
\label{section:variance_estimator}
The goal of this section is to revisit the claim made in Section \ref{sec:sensitivity} that, when looking at the H\'ajek estimator in a
randomized controlled trial, standard variance estimators that ignore interference effects should be re-interpreted as
estimators for $\pi(1 - \pi) V_0 + \sigma_0^2$ in our model.
To this end, we focus on the following basic plug-in variance estimator that would be consistent in the absence of interference. Let
\begin{align*}
&\hat{\nu}_0 = \sum_{i=1}^n {(Y_i - \hmu_0)^2 (1-W_i)}/{\p{\sum_{i = 1}^n(1-W_i)}}, \\
&\hat{\nu}_1 = \sum_{i=1}^n {(Y_i -  \hmu_1)^2 W_i}/{\p{\sum_{i = 1}^n{W_i}}}
\end{align*}
where $\hmu_1 = \sum_{i=1}^n Y_i W_i/{(\sum_{i = 1}^n{W_i})}$ is the sample mean of the treated group and $\hmu_0 = \sum_{i=1}^n Y_i (1-W_i)/{(\sum_{i = 1}^n(1 - W_i))}$ is the sample mean of the control group. Then
$$\hat{\nu} = \hat{\nu}_1/\pi + \hat{\nu}_0/(1- \pi) $$
is a natural plug-in variance estimator for the H\'ajek estimator in the no-interference setting.
The following proposition establishes that, furthermore, this variance estimator is consistent for $\pi(1 - \pi) V_0 + \sigma_0^2$ in our model.

\begin{prop}
Under Assumptions
\ref{assu:undirected}, \ref{assu:anon} and \ref{assu:smooth}, if $\min_i N_i \to \infty$, then $\hat{\nu}$ converges to $\pi(1 - \pi) V_0 + \sigma_0^2$ in probability. 
\end{prop}

\begin{proof}
As a preliminary step, we start by expressing the variance $\pi(1 - \pi) V_0 + \sigma_0^2$ in a simpler form. 
\begin{equation}
\begin{split}
\pi(1 - \pi) V_0 &= \pi(1 - \pi) \Var{R_i} = \pi(1 - \pi) \Var{\frac{f_i(1 , \pi)}{\pi} + \frac{f_i(0, \pi)}{1-\pi}}\\
& = \frac{1 - \pi}{\pi} \Var{f_i(1 , \pi)} + \frac{\pi}{1 - \pi} \Var{f_i(0 , \pi)} + 2\Cov{f_i(1 , \pi) , f_i(0, \pi)}\\
& =  \frac{1}{\pi} \Var{f_i(1 , \pi)} + \frac{1}{1 - \pi} \Var{f_i(0 , \pi)} - \Var{f_i(1 , \pi) - f_i(0, \pi)}. 
\end{split}
\end{equation}
This implies that $\pi(1 - \pi) V_0 + \sigma_0^2 = \Var{f_i(1 , \pi)}/\pi +  \Var{f_i(0 , \pi)}/(1 - \pi)$. 

We now seek to establish that $\hat{\nu}_1 \stackrel{p}{\to} \Var{f_i(1 , \pi)}$ and that $\hat{\nu}_0 \stackrel{p}{\to} \Var{f_i(0 , \pi)}$. To show this, we start with analyzing $\sum_{i = 1}^n Y_i^2 W_i$. Note that 
\begin{equation}
\label{eqn:var_esi1}
\begin{split}
\frac{1}{n}\sum_{i = 1}^n Y_i^2 W_i 
&= \frac{1}{n}\sum_{i = 1}^n f_i^2(1, M_i/N_i) W_i
= \frac{1}{n} \sum_{i = 1}^n \sqb{f_i(1, \pi) + f_i'(1, \pi^{\star})(M_i/N_i - \pi) }^2 W_i\\
& = \frac{1}{n} \sum_{i = 1}^n f_i^2(1, \pi) W_i + \Delta_1 + \Delta_2,
\end{split}
\end{equation}
where $\pi^{\star}$ is between $M_i/N_i$ and $\pi$, $\Delta_1 = \frac{2}{n}\sum_{i=1}^n f_i(1, \pi) f_i'(1, \pi^\star)W_i (M_i/N_i - \pi)$, and  $\Delta_2 = \frac{1}{n}\sum_{i=1}^n f_i'(1, \pi^\star)^2 W_i (M_i/N_i - \pi)^2$. The two error terms $\Delta_1$ and $\Delta_2$ satisfy
\begin{equation}
\begin{split}
\EE{\Delta_1^2} \leq \frac{4}{n}\sum_{i = 1}^n B^4 \EE{(M_i / N_i - \pi)^2}
= \frac{4B^4}{n} \sum_{i = 1}^n 1/N_i 
\leq \frac{4B^4}{\min_i N_i}
\end{split}
\end{equation}
\begin{equation}
\begin{split}
\EE{\abs{\Delta_2}}  = \EE{\Delta_2} \leq \frac{B^2}{n} \sum_{i = 1}^n 1/N_i  \leq \frac{B^2}{\min_i N_i}. 
\end{split}
\end{equation}
Therefore, $\Delta_1 \stackrel{p}{\to} 0$ and $\Delta_2 \stackrel{p}{\to} 0$. Note also that by law of large numbers, $\frac{1}{n} \sum_{i = 1}^n f_i^2(1, \pi) W_i \stackrel{p}{\to} \EE{f_i^2(1, \pi) W_i} = \pi \EE{f_i^2(1, \pi)}$. Thus $\frac{1}{n}\sum_{i = 1}^n Y_i^2 W_i \stackrel{p}{\to}  \pi \EE{f_i^2(1, \pi)}$. It has also been established in the proof of Theorem \ref{theo:directCLT} that $\frac{1}{n}\sum_{i = 1}^n Y_i W_i \stackrel{p}{\to}  \pi \EE{f_i(1, \pi)}$. Followed easily from the above facts, we have $\hat{\nu}_1 = \sum_{i=1}^n {(Y_i - \hmu_1)^2 W_i}/{(\sum_{i = 1}^n{W_i})} \stackrel{p}{\to} \EE{f_i^2(1, \pi)} - \p{\EE{f_i(1, \pi)}}^2 = \Var{f_i(1, \pi)}$. 

The other part $\hat{\nu}_0 \stackrel{p}{\to} \Var{f_i(0 , \pi)}$ can be showed using similar arguments. Therefore, 
\begin{equation}
\hat{\nu} = \hat{\nu}_1/\pi + \hat{\nu}_0/(1- \pi)   \stackrel{p}{\to}  \Var{f_i(1 , \pi)}/\pi +  \Var{f_i(0 , \pi)}/(1 - \pi) =\pi(1 - \pi) V_0 + \sigma_0^2.
\end{equation}

\end{proof}
\end{appendix}

\fi

\end{document}